\documentclass[11pt]{amsart}%
\usepackage{amsmath}%
\usepackage{amssymb}%
\usepackage{amscd}%
\usepackage{color}%
\usepackage{mathrsfs}%
\usepackage[all]{xy}%
\usepackage{newtxtext,newtxmath}%
\usepackage[all, warning]{onlyamsmath}
\def\ol#1{\overline{#1}}
\def\wt#1{\widetilde{#1}}
\def\ul#1{\underline{#1}}
\theoremstyle{plain}
    \newtheorem{theorem}{Theorem}[section]
    
    \newtheorem{proposition}[theorem]{Proposition}
    \newtheorem{lemma}[theorem]{Lemma}
    \newtheorem{corollary}[theorem]{Corollary}
      
\theoremstyle{definition}
    \newtheorem{definition}[theorem]{Definition}
    
    \newtheorem{example}[theorem]{Example}
    \newtheorem{remark}[theorem]{Remark}
\def\Alphabet{A,B,C,D,E,F,G,H,I,J,K,L,M,N,O,P,Q,R,S,T,U,V,W,X,Y,Z}
\def\alphabet{a,b,c,d,e,f,g,h,i,j,k,l,m,n,o,p,q,r,s,t,u,v,w,x,y,z}
\def\endpiece{xxx}
\def\makeAlphabet[#1]{\expandafter\makeA#1,xxx,}
\def\makealphabet[#1]{\expandafter\makea#1,xxx,}
\def\makeA#1,{\def\temp{#1}\ifx\temp\endpiece\else%
\mkbb{#1}\mkfrak{#1}\mkbf{#1}\mkcal{#1}\mkscr{#1}\mkbs{#1}\expandafter\makeA\fi}%
\def\makea#1,{\def\temp{#1}\ifx\temp\endpiece\else\mkfrak{#1}\mkbf{#1}\mkbs{#1}\expandafter\makea\fi}%
\def\mkbb#1{\expandafter\def\csname bb#1\endcsname{\mathbb{#1}}}
\def\mkfrak#1{\expandafter\def\csname fr#1\endcsname{\mathfrak{#1}}}
\def\mkbf#1{\expandafter\def\csname b#1\endcsname{\mathbf{#1}}}
\def\mkcal#1{\expandafter\def\csname c#1\endcsname{\mathcal{#1}}}
\def\mkscr#1{\expandafter\def\csname s#1\endcsname{\mathscr{#1}}}
\def\mkbs#1{\expandafter\def\csname bs#1\endcsname{{\boldsymbol{#1}}}}
\def\makeop[#1]{\xmakeop#1,xxx,}
\def\mkop#1{\expandafter\def\csname #1\endcsname{{\mathrm{#1}}}} %
\def\xmakeop#1,{\def\temp{#1}\ifx\temp\endpiece\else\mkop{#1}\expandafter\xmakeop\fi}%
\def\makeup[#1]{\xmakeup#1,xxx,}
\def\mkup#1{\expandafter\def\csname #1\endcsname{{\mathrm{#1}\,}}} %
\def\xmakeup#1,{\def\temp{#1}\ifx\temp\endpiece\else\mkup{#1}\expandafter\xmakeup\fi}%
\makeAlphabet[\Alphabet]
\makealphabet[\alphabet]
\makeop[Alt,End,Ext,Hom,Sym,Tor,dR,Li,Cone,Tot,res,Res,id,pol,Gr,Ker,Coker,MHS,Im,Re,Vec,Rep,Rex,bil,GL,Ind,nil,OF,Fil,BG]
\makeup[Spec]

\def\wMHS{\sM}
\def\hint#1{}%
\begin{document}
\title[Mixed Plectic Hodge Structures]{Category of Mixed Plectic Hodge Structures}
\author[Bannai]{Kenichi Bannai$^{\star*\diamond}$}
\email{bannai@math.keio.ac.jp}
\address{${}^\star$Keio Institute of Pure and Applied Sciences (KiPAS), Graduate School of Science and Technology, Keio University, 3-14-1 Hiyoshi, Kouhoku-ku, Yokohama 223-8522, Japan}
\address{${}^*$Department of Mathematics, Faculty of Science and Technology, Keio University, 3-14-1 Hiyoshi, Kouhoku-ku, Yokohama 223-8522, Japan}
\address{${}^\diamond$Mathematical Science Team, RIKEN Center for Advanced Intelligence Project (AIP),1-4-1 Nihonbashi, Chuo-ku, Tokyo 103-0027, Japan}
\address{${}^\dagger$Faculty of Mathematics, Kyushu University 744, Motooka, Nishi-Ku, Fukuoka 819-0395, Japan}
\address{${}^\ddagger$Department of Mathematics, Graduate School of Science, Osaka University, Toyonaka, Osaka 560-0043, Japan}
\author[Hagihara]{Kei Hagihara$^{*\diamond}$}
\author[Kobayashi]{Shinichi Kobayashi$^\dagger$}
\author[Yamada]{Kazuki Yamada$^*$}
\author[Yamamoto]{Shuji Yamamoto$^{\star*\diamond}$}
\author[Yasuda]{Seidai Yasuda$^{\ddagger\diamond}$}
\date{\today \quad (Version 2.0)}
\begin{abstract}
	The purpose of this article is to investigate the properties of the category of mixed plectic Hodge structures defined by Nekov\'{a}\v{r} and Scholl \cite{NS}.
	We give an equivalent description of mixed plectic Hodge structures in terms of the weight and partial Hodge filtrations.
	We also construct an explicit complex calculating the extension groups in this category. 
\end{abstract}
\thanks{This research was conducted as part of the KiPAS program FY2014--2018 of the Faculty of Science and Technology at Keio University.
This research was supported in part by KAKENHI 26247004, 16J01911, 16K13742, 18H05233 as well as the 
JSPS Core-to-Core program ``Foundation of a Global Research Cooperative
Center in Mathematics focused on Number Theory and Geometry''. }
\maketitle
\setcounter{tocdepth}{1}
\setcounter{section}{0}
\tableofcontents%

%
%
%
\section{Introduction}
%
%
%

Let $g$ be an integer $\geq 0$.
In a very insightful article \cite{NS}, Nekov\'{a}\v{r} and Scholl  introduced the category of 
\textit{mixed $g$-plectic $\bbR$-Hodge structures}, which is a generalization of the category $\MHS_\bbR$ of 
mixed $\bbR$-Hodge structures originally defined by Deligne \cite{D1}.
 If we let $\cG$ be the tannakian fundamental group of $\MHS_\bbR$, then the category of mixed $g$-plectic $\bbR$-Hodge
structures was defined in \cite[\S 16]{NS} to be the category $\Rep_\bbR(\cG^g)$ of finite $\bbR$-representations of the pro-algebraic group $\cG^g$.
The purpose of this article is to investigate some properties of the category $\Rep_\bbR(\cG^g)$.
In particular we give a description of objects in $\Rep_\bbR(\cG^g)$ in terms of the weight and partial Hodge filtrations.
We then give an explicit complex calculating the extension groups in this category.  
This article arose as an attempt by the authors to understand the beautiful theory proposed by Nekov\'{a}\v{r} and Scholl.

The detailed content of this article is as follows.  We will mainly deal with the complex case, and will return to the real case at the end of the article.
 In \S 2, we review the properties of mixed $\bbC$-Hodge structures, and will review the construction of the tannakian fundamental
group $\cG_\bbC$ of the category of mixed $\bbC$-Hodge structures $\MHS_\bbC$.   
We will then give in Proposition \ref{prop: gexplicit} the following explicit description of objects in
the category $\Rep_\bbC(\cG^g_\bbC)$:

\begin{proposition}[=Proposition \ref{prop: gexplicit}]
	An object in $\Rep_\bbC(\cG^g_\bbC)$  corresponds to a triple
	\[U = (U_\bbC, \{  U^{\bsp,\bsq} \}, \{ t_\mu\}),\] where $U_\bbC$ is a finite dimensional $\bbC$-vector space,
	$\{ U^{\bsp,\bsq}\}$ is a $2g$-grading of $U_\bbC$ by $\bbC$-linear subspaces 
	\begin{equation*}
		U_\bbC = \bigoplus_{\bsp,\bsq \in \bbZ^g} U^{\bsp,\bsq},
	\end{equation*}
	and $t_\mu$ for $\mu =1, \ldots, g$ are
 	$\bbC$-linear automorphisms of $U_\bbC$ commutative with each other, satisfying 
 	\begin{equation*}
	 	(t_\mu -1)(U^{\bsp,\bsq}) \subset \bigoplus_{\substack{\bsr,\bss\in\bbZ^g\\ (r_\nu,s_\nu)=(p_\nu, q_\nu)
		\text{ for }\nu\neq\mu\\(r_\mu,s_\mu) < (p_\mu, q_\mu)\quad}} U^{\bsr\!,\!\,\bss}
 	\end{equation*}
	for any $\bsp =  (p_1, \ldots, p_g), \bsq =(q_1, \ldots, q_g) \in \bbZ^g$, where the direct sum is over the indices
	$\bsr = (r_1, \ldots, r_g), \bss =(s_1, \ldots, s_g) \in \bbZ^g$ satisfying $r_\nu=p_\nu$, $s_\nu=q_\nu$ for $\nu\neq\mu$ and
	$r_\mu < p_\mu$, $s_\mu < q_\mu$.
\end{proposition}

Let $V =(V_\bbC, \{W^\mu_\bullet\}, \{F^\bullet_\mu\},\{\ol F^\bullet_\mu\})$ be a quadruple
consisting of a finite dimensional $\bbC$-vector space $V_\bbC$, 
a family of finite ascending filtrations $W^\mu_\bullet$ for $\mu=1,\ldots,g$ by $\bbC$-linear subspaces on $V_\bbC$,
and families of finite descending filtrations $F_\mu^\bullet$ and $\ol F_\mu^\bullet$ for $\mu=1,\ldots,g$ by $\bbC$-linear subspaces on $V_\bbC$.
We say that $V$ as above is a \textit{$g$-orthogonal family of mixed $\bbC$-Hodge structures}, 
if for any $\mu$, the quadruple $(V_\bbC, W^\mu_\bullet, F^\bullet_\mu, \ol F^\bullet_\mu)$ is a mixed $\bbC$-Hodge structure, 
and for any $\mu$ and $\nu \neq \mu$, the $\bbC$-linear subspaces 
$W^\mu_n V_\bbC$, $F^m_\mu V_\bbC$, $\ol F^m_\mu V_\bbC$
with the weight and Hodge  filtrations induced from $W^\nu_\bullet$, $F_\nu^\bullet$, $\ol F_\nu^\bullet$ are mixed $\bbC$-Hodge structures.
We call the filtrations $\{W^\mu_\bullet\}$ the partial weight filtrations and the filtrations $\{F^\bullet_\mu\}$,$\{\ol F^\bullet_\mu\}$ the partial
Hodge filtrations of $V$.

We denote by $\OF^g_\bbC$ the category whose objects are $g$-orthogonal family of mixed $\bbC$-Hodge structures.
A morphism in  $\OF^g_\bbC$ is a $\bbC$-linear homomorphism of underlying $\bbC$-vector spaces compatible with the partial weight and Hodge filtrations.
The main result of \S 3 is the following:

\begin{proposition}[=Corollary \ref{cor: main cor}]
	For $g\geq 0$, we have an equivalence of categories
	\begin{equation*}
		\Rep_\bbC(\cG^g_\bbC) \cong\OF^g_\bbC.
	\end{equation*}	
\end{proposition}

While writing this paper, Nekov\'{a}\v{r} and Scholl released a new preprint \cite{NS2}, which contains a result similar to Proposition 1.2.

Suppose $V =(V_\bbC, \{W^\mu_\bullet\}, \{F^\bullet_\mu\},\{\ol F^\bullet_\mu\})$ is a $g$-orthogonal family of mixed $\bbC$-Hodge structures.
We define the total weight filtration $W_\bullet$ of $V$ to be the finite ascending filtration by $\bbC$-linear subspaces of $V_\bbC$ given by
\begin{equation*}
	W_nV_\bbC:=\sum_{n_1+\cdots+n_g=n}(W^1_{n_1}\cap\cdots\cap W^g_{n_g})V_\bbC.
\end{equation*}
The purpose of \S 4 is to give a characterization of a quadruple $(V_\bbC,W_\bullet, \{F^\bullet_\mu\},\{\ol F^\bullet_\mu\})$
which is constructed from $\OF^g_\bbC$.  In particular, we will give in Definition \ref{def: pMHS} the definition of the category of
mixed $g$-plectic $\bbC$-Hodge structures $\MHS^g_\bbC$, whose objects are the quadruple $(V_\bbC,W_\bullet, \{F^\bullet_\mu\},\{\ol F^\bullet_\mu\})$
satisfying certain conditions.  We will then show in Theorem \ref{thm: equiv for C} that we have
an equivalence of categories as follows:

\begin{theorem}[=Theorem \ref{thm: equiv for C}]
	We have an equivalence of categories
	\begin{equation*}
		\OF^g_\bbC\cong\MHS^g_\bbC.
	\end{equation*}	
\end{theorem}

In \S 5, we will introduce the category of mixed $g$-plectic $\bbR$-Hodge structures, and show the
corresponding results in the real case.
We will then prove in Corollary \ref{cor: subobject} that an object in $\Rep_\bbR(\cG^g)$
may be given as a subquotient of exterior products of objects in $\MHS_\bbR$. 
The main result of \S 5 is Theorem \ref{thm: old theorem}, which gives an explicit complex calculating the extension groups in $\Rep_\bbR(\cG^g)$.

%
%
%
\section{Mixed Hodge structures}
%
%
%

In this section, we will review the definition of the category of mixed Hodge structures $\MHS_\bbC$ and the tannakian 
fundamental group $\cG_\bbC$ associated to $\MHS_\bbC$.  We will then give an explicit description of objects in the category 
$\Rep_\bbC(\cG^g_\bbC)$ of finite dimensional $\bbC$-representations of $\cG^g_\bbC$, 
where $\cG^g_\bbC$ for an integer $g\geq0$ is the $g$-fold product of $\cG_\bbC$.

%
\subsection{Definition of the category of mixed plectic $\bbC$-Hodge structures}
%

In this subsection, we first give the definitions of pure and mixed $\bbC$-Hodge structures, and review their properties.

\begin{definition}[pure $\bbC$-Hodge structure]\label{def: pure}	
	Let $V_\bbC$ be a finite dimensional $\bbC$-vector space,
	and let $F^\bullet$ and $\ol F^\bullet$ be finite descending filtrations by $\bbC$-linear subspaces on $V_\bbC$.
	We say that the triple $V:= (V_\bbC, F^\bullet, \ol F^\bullet)$ is a \textit{pure $\bbC$-Hodge structure of weight $n$},
	if it satisfies	
	\begin{equation}\label{eq: n-oppose}
		V_\bbC = F^p V_\bbC \oplus \ol{F}^{n+1-p} V_\bbC
	\end{equation}
	for any $p \in \bbZ$.  We call the filtrations $F^\bullet$ and $\ol F^\bullet$ the \textit{Hodge filtrations} of $V$.
\end{definition}

\begin{example}
	The Tate object $\bbC(n) := (V_\bbC, F^\bullet, \ol F^\bullet)$, which is a $\bbC$-vector
	space $V_\bbC =  \bbC$, with the Hodge filtrations given by $F^{-n} V_\bbC = \ol F^{\,-n} V_\bbC  =V_\bbC$ and $F^{-n+1} V_\bbC = \ol F^{\,-n+1} V_\bbC = 0$
	is an example of a pure $\bbC$-Hodge structure of weight $-2n$.
\end{example}

It is known that pure $\bbC$-Hodge structures may be described as follows.

\begin{lemma}[\cite{D1} Proposition 1.2.5, Proposition 2.1.9]\label{lemma: splitting}
	Let $V_\bbC$ be a finite dimensional $\bbC$-vector space,
	and let $F^\bullet$ be a finite descending filtration by $\bbC$-linear subspaces on $V_\bbC$.
	Then $V:= (V_\bbC, F^\bullet, \ol F^\bullet)$ is a \textit{pure $\bbC$-Hodge structure of weight $n$} if and only if we have
	\begin{equation}\label{eq: bigrading}
		V_\bbC = \bigoplus_{\substack{p,q\in\bbZ\\p+q=n}} (F^p \cap \ol F^q)V_\bbC,
	\end{equation}
	where $(F^p \cap \ol F^q)V_\bbC := F^pV_\bbC\cap \ol F^qV_\bbC$.
\end{lemma}

Let $V$ be a pure $\bbC$-Hodge structure of weight $n$.
The Hodge filtration may be described in terms of this splitting as follows.

\begin{lemma}\label{lem: candy}
	If $V$ is a pure $\bbC$-Hodge structure of weight $n$, then for any $p,q \in \bbZ$, we have
	\begin{align}\label{eq: PHSFm}
		F^p V_\bbC&=\bigoplus_{\substack{r+s=n,\\r \geq p}}(F^r \cap\ol F^{s}) V_\bbC,&
		\ol F^q V_\bbC&=\bigoplus_{\substack{r+s=n,\\s \geq q}}(F^r \cap\ol F^{s}) V_\bbC.
	\end{align}
\end{lemma}

\begin{proof}
	If $r \geq p$, then we have $(F^r \cap\ol F^{n-r}) V_\bbC \subset F^p V_\bbC$, and if $r < p$, then $n-r \geq n+1-p$,
	hence $(F^r \cap\ol F^{n-r}) V_\bbC \subset\ol F^{n+1-p}V_\bbC$.  The first equality follows from Lemma \ref{lemma: splitting} 
	and \eqref{eq: n-oppose}.  The second equality is proved in a similar manner.
\end{proof}

The definition of mixed $\bbC$-Hodge structures is given as follows.

\begin{definition}[mixed $\bbC$-Hodge structure]
	Let $V_\bbC$ be a finite dimensional $\bbC$-vector space.
	Let $W_\bullet$ be a finite ascending filtration by $\bbC$-linear subspaces on $V_\bbC$,
	and let $F^\bullet$ and $\ol F^\bullet$ be finite descending filtrations by $\bbC$-linear subspaces on $V_\bbC$.
	We say that the quadruple $V=(V_\bbC,  W_\bullet, F^\bullet, \ol F^\bullet)$ is a \textit{mixed $\bbC$-Hodge structure} if,
	for each $n \in \bbZ$, the structure induced by $F^\bullet$ and $\ol F^\bullet$ on $\Gr_n^W V_\bbC$ is a pure 
	$\bbC$-Hodge structure of weight $n$.
\end{definition}

If $V=(V_\bbC,  W_\bullet, F^\bullet, \ol F^\bullet)$ is a mixed $\bbC$-Hodge structure, then we call $W_\bullet$
the \textit{weight filtration} and $F^\bullet$, $\ol F^\bullet$ the \textit{Hodge filtrations} of $V$.
The Deligne splitting below gives a generalization of \eqref{eq: bigrading} for
mixed $\bbC$-Hodge structures.

\begin{proposition}[Deligne splitting]\label{prop: splitting}
	Let $V = (V_\bbC, W_\bullet, F^\bullet, \ol F^\bullet)$ be a mixed $\bbC$-Hodge structure,
	and let
	\begin{equation}\label{eq: Apq}
		A^{p,q}(V) := (F^p \cap W_{n}) V_\bbC \cap \biggl((\ol F^q \cap W_{n})V_\bbC+\sum_{j\geq 0}(\ol F^{q-j}\cap W_{n-j-1}) V_\bbC\biggr)
	\end{equation}	
	for $p,q\in\bbZ$ and $n:=p+q$.
	Then $\{A^{p,q}(V)\}$ gives a bigrading of $V_\bbC$ by $\bbC$-linear subspaces
	\begin{equation}\label{eq: Deligne splitting}
		V_\bbC =\bigoplus_{p,q \in \bbZ} A^{p,q}(V).
	\end{equation}
	Moreover, for $n, p\in\bbZ$, the weight and Hodge filtrations on $V$ satisfy
	\begin{align*}
		W_n V_\bbC &=\bigoplus_{\substack{p,q\in\bbZ\\p+q \leq n}} A^{p,q}(V), &
		F^p V_\bbC &=\bigoplus_{\substack{r,s\in\bbZ\\r \geq p}} A^{r,s}(V).
	\end{align*}	
\end{proposition}

We call the bigrading $\{A^{p,q}(V)\}$ of $V_\bbC$ given in Proposition \ref{prop: splitting} the \textit{Deligne splitting} of 
the mixed $\bbC$-Hodge structure $V$.  The key ingredient for the proof of Proposition \ref{prop: splitting} is the following lemma.

\begin{lemma}\label{lem: Apq}
	Let $V$ be a mixed $\bbC$-Hodge structure, and let $\{A^{p,q}(V)\}$ be the Deligne splitting of $V$ as in \eqref{eq: Apq}.
	Then for any $p,q\in\bbZ$ and $n:=p+q$, the canonical surjection
	$
		W_{n}V_\bbC \rightarrow  \Gr_{n}^W V_\bbC
	$
	induces a $\bbC$-linear isomorphism
	\begin{align*}
		A^{p,q}(V)&\xrightarrow\cong (F^p \cap\ol F^q) \Gr_{n}^W V_\bbC.
	\end{align*}
	Here, $(F^p \cap\ol F^q) \Gr_{n}^W V_\bbC: =  F^p \Gr_{n}^W V_\bbC \cap\ol F^q \Gr_{n}^W V_\bbC$.
\end{lemma}

\begin{proof}
	See for example \cite[ Lemma-Definition 3.4]{PS}.
\end{proof}

We may now prove Proposition \ref{prop: splitting} as follows.

\begin{proof}[Proof of Proposition \ref{prop: splitting}]
	Let $\{A^{p,q}(V)\}$ be the Deligne splitting of $V$.  By Lemma \ref{lem: Apq}, we have an isomorphism
	\begin{equation*}
		\bigoplus_{\substack{p,q\in\bbZ\\p+q=n}} A^{p,q}(V) \xrightarrow\cong \bigoplus_{\substack{p,q\in\bbZ\\p+q=n}} 
		 (F^p \cap\ol F^q)\Gr^W_nV_\bbC 
	\end{equation*}
	for any integer $n \in \bbZ$.  By the definition of the weight filtration on mixed $\bbC$-Hodge structures, 
	$\Gr^W_n V$ is a pure $\bbC$-Hodge structure of weight $n$, hence we have	
	\begin{equation*}
		\Gr^W_nV_\bbC =\bigoplus_{p+q=n}  (F^p \cap\ol F^q)\Gr^W_nV_\bbC.
	\end{equation*}
	by Lemma \ref{lemma: splitting}.
	This shows that $V_\bbC =\bigoplus_{p,q\in\bbZ} A^{p,q}(V)$ as desired. The statements for the Hodge and weight filtrations follow from this result.
\end{proof}

\begin{remark}\label{rem: splitting}
	Exchanging the roles of $F^\bullet$ and $\ol F^\bullet$, we define
	\begin{equation*}
		\ol A^{p,q}(V) := (\ol F^p \cap W_{n}) V_\bbC \cap \biggl((F^q \cap W_{n})V_\bbC+\sum_{j\geq 0}(F^{q-j}\cap W_{n-j-1}) V_\bbC\biggr).
	\end{equation*}
	Then for any $p,q\in\bbZ$ and $n:=p+q$, the canonical surjection
	$
		W_{n}V_\bbC \rightarrow  \Gr_{n}^W V_\bbC
	$
	induces a $\bbC$-linear isomorphism
	\begin{align*}
		\ol A^{p,q}(V)&\xrightarrow\cong (\ol F^p \cap F^q) \Gr_{n}^W V_\bbC,
	\end{align*}	
	$\{\ol A^{p,q}(V)\}$ gives a bigrading of $V_\bbC$, and we have for any $n,p\in\bbZ$
	\begin{align*}
		W_n V_\bbC &=\bigoplus_{\substack{p,q\in\bbZ\\p+q \leq n}} \ol A^{p,q}(V), &
		\ol F^p V_\bbC &=\bigoplus_{\substack{r,s\in\bbZ\\r \geq p}} \ol A^{r,s}(V).
	\end{align*}
\end{remark}

We will use Proposition \ref{prop: splitting} and Remark \ref{rem: splitting}
to prove the strictness with respect to the weight and Hodge filtrations of morphism of mixed Hodge structures.
We first prepare some terminology.

\begin{definition}
	Suppose $U$ and $V$ are finite dimensional $\bbC$-vector spaces with $\bbC$-linear subspaces $WU \subset U$ and $WV \subset V$.  
	We say that a $\bbC$-linear homomorphism
	\begin{equation*}
		\alpha: U \rightarrow V
	\end{equation*}
	is \textit{compatible with $W$} if $\alpha(WU) \subset WV$, and that $\alpha$ is \textit{strict with respect to $W$} if we have
	\begin{equation*}
		\alpha(W U) = \alpha(U) \cap WV.
	\end{equation*}
\end{definition}

We denote by $\MHS_\bbC$ the category of mixed $\bbC$-Hodge structures.  A morphism $\alpha: U \rightarrow V$ in this category is a $\bbC$-linear
homomorphism $\alpha: U_\bbC \rightarrow V_\bbC$ of underlying $\bbC$-vector spaces compatible with the weight and Hodge filtrations.
Then we have the following.

\begin{proposition}\label{prop: strictness of A}
	Let $\alpha:U\rightarrow V$ be a morphism in $\MHS_\bbC$, and let $\cS$ be a subset of $\bbZ\times\bbZ$.
	Then we have
	\begin{align}\label{A is strict}
		\alpha\left(\bigoplus_{(p,q)\in \cS}A^{p,q}(U)\right)=\alpha(U_\bbC)\cap\left(\bigoplus_{(p,q)\in \cS}A^{p,q}(V)\right)
	\end{align}
	and
	\begin{equation}\label{WF is strict}
		\alpha\left(\sum_{(p,n)\in \cS}(F^p \cap W_n)U_\bbC\right)=\alpha(U_\bbC)\cap\left(\sum_{(p,n)\in \cS}(F^p \cap W_n)V_\bbC\right).
	\end{equation}
	Statements \eqref{A is strict} and \eqref{WF is strict} with $F^p$ replaced by $\ol F^p$ and $A^{p,q}$ replaced by $\ol A^{p,q}$
	are also true.  
	In particular, $\alpha$ is strict with respect to the filtrations $F^\bullet\cap W_\bullet$ and $\ol F^\bullet\cap W_\bullet$.
	Furthermore, if $U$ and $V$ are both pure $\bbC$-Hodge structures of weight $n$, then we have
	\begin{equation}\label{eq: pure}
		\alpha((F^p \cap\ol F^q)U_\bbC) = \alpha(U_\bbC) \cap (F^p \cap\ol F^q)V_\bbC.
	\end{equation}
\end{proposition}

\begin{proof}
	Since $\alpha(A^{p,q}(U))\subset A^{p,q}(V)$, assertion \eqref{A is strict} follows from 
	the fact that the Deligne splitting gives a bigrading \eqref{eq: Deligne splitting} of $U_\bbC$ and $V_\bbC$.
	Equality \eqref{WF is strict} 
	follows from the fact that 
	\begin{align*}
		(F^p\cap W_n)U_\bbC &= \bigoplus_{\substack{r \geq p, \\r+s \leq n}} A^{r,s}(U), &
		(F^p\cap W_n)V_\bbC &= \bigoplus_{\substack{r \geq p,\\r+s \leq n}} A^{r,s}(V),
	\end{align*}
	and this proves the strictness of $\alpha$ with respect to $F^\bullet\cap W_\bullet$. The strictness of $\alpha$ with respect to $\ol F^\bullet\cap W_\bullet$
	follows from a parallel argument with $A^{p,q}$ replaced by $\ol A^{p,q}$.  The assertion \eqref{eq: pure} for the pure case follows from \eqref{A is strict}, noting the fact that
	\begin{align*}
		A^{p,q}(U)&= (F^p\cap\ol F^q)U_\bbC, &
		A^{p,q}(V)&= (F^p\cap\ol F^q)V_\bbC
	\end{align*}
	if $p+q=n$ and is zero otherwise.
\end{proof}

Using Proposition \ref{prop: strictness of A}, one can prove that $\MHS_\bbC$ is an abelian category (\cite{D1} Th\'eor\`eme 2.3.5).
The following result will be used in the proof of Proposition \ref{prop: pApq}.

\begin{corollary}\label{cor: sum and intersection}
	Let $V$ be a mixed $\bbC$-Hodge structure.
	For any $\bbC$-linear subspace $U_\bbC$ of $V_\bbC$, the weight and Hodge filtrations on $V$ induce
	the filtrations
	\begin{align*}
			W_nU_\bbC&:= U_\bbC \cap W_nV_\bbC,&F^p U_\bbC&:=U_\bbC \cap F^pV_\bbC,&
			\ol F^q U_\bbC&:=U_\bbC \cap\ol F^qV_\bbC
	\end{align*}
	on $U_\bbC$.
	Suppose two $\bbC$-linear subspaces $U_\bbC$ and $U'_\bbC$ of $V_\bbC$ with the induced filtrations as above are mixed $\bbC$-Hodge structures.
	Then $U_\bbC+U'_\bbC$ and $U_\bbC\cap U'_\bbC$ with the induced filtrations are also mixed $\bbC$-Hodge structures.
	Moreover, we have
	$W_n(U_\bbC+U'_\bbC)=W_nU_\bbC+W_nU'_\bbC$, $F^p(U_\bbC+U'_\bbC)=F^pU_\bbC+F^pU'_\bbC$,
	and $\ol F^q(U_\bbC+U'_\bbC)=\ol F^q U_\bbC+\ol F^q U'_\bbC$
	which by definition is equivalent to
	\begin{align}
		\bigl(U_\bbC+U'_\bbC\bigr) \cap W_n V_\bbC&=  U_\bbC\cap W_n V_\bbC+U'_\bbC \cap W_n V_\bbC \label{filtration of sum:1},\\  
		 (U_\bbC+U'_\bbC)\cap   F^p V_\bbC&= U_\bbC\cap   F^p V_\bbC+U'_\bbC\cap   F^p V_\bbC\label{filtration of sum:2},\\
		  (U_\bbC+U'_\bbC)\cap   \ol F^q V_\bbC&= U_\bbC\cap   \ol F^q V_\bbC+U'_\bbC\cap   \ol F^q V_\bbC\label{filtration of sum:3}.
	\end{align}
\end{corollary}

\begin{proof}
	The map $U_\bbC\oplus U'_\bbC\rightarrow V_\bbC$ sending $(u,u')$ to $u+u'$ is a morphism of mixed $\bbC$-Hodge structures, 
	hence is strictly compatible with the filtrations $W_{\bullet}$, $F^{\bullet}$ and $\ol F^{\bullet}$. 
	This implies \eqref{filtration of sum:1}, \eqref{filtration of sum:2} and \eqref{filtration of sum:3}, and
	we see that the image $U_\bbC+U'_\bbC$ is also a mixed $\bbC$-Hodge structure.  The natural map 
	$U_\bbC\rightarrow(U_\bbC+U'_\bbC)/U'_\bbC$ is also a morphism of mixed $\bbC$-Hodge structures, hence we see that the 
	kernel $U_\bbC \cap U'_\bbC$ is also a mixed $\bbC$-Hodge structure.
\end{proof}


The category $\MHS_\bbC$ is known to be a neutral tannakian category with respect to the natural tensor product 
and the fiber functor $\omega: \MHS_\bbC \rightarrow \Vec_\bbC$ obtained by associating to $V$ the 
$\bbC$-vector space $\Gr^W_\bullet V_\bbC:=\bigoplus_n\Gr^W_nV_\bbC$. If we denote by $\cG_\bbC$ 
the tannakian fundamental group of $\MHS_\bbC$, then $\cG_\bbC$ is an affine group scheme over $\bbC$.  
By the definition of the tannakian fundamental group, we have a natural equivalence of categories
\begin{equation*}
	 \MHS_\bbC \xrightarrow\cong \Rep_\bbC(\cG_\bbC)
\end{equation*}
induced by the fiber functor $\omega$, where $\Rep_\bbC(\cG_\bbC)$ is the category of $\bbC$-linear representations  of $\cG_\bbC$ on finite dimensional 
$\bbC$-vector spaces.

%
\subsection{The tannakian fundamental group of $\MHS_\bbC$.}
%

In this subsection, we will review the construction of the tannakian fundamental group $\cG_\bbC$ of the category $\MHS_\bbC$,
and give an explicit description of objects in $\Rep_\bbC(\cG_\bbC)\cong \MHS_\bbC$.

We denote by $\frL_n$ the free Lie algebra over $\bbC$ generated by symbols $T^{i,j}$ for positive integers $i,j$ with $i+j\leq n$. 
We define the degree of elements of $\frL_n$ by $\deg(T^{i,j}):=i+j$, and denote by $I_n$ the ideal of $\frL_n$ generated by elements of degree larger than $n$.
Then $\fru_n:=\frL_n/I_n$ is a nilpotent Lie algebra over $\bbC$.  
The category $\Rep^\nil_\bbC(\fru_n)$ of nilpotent representations of $\fru_n$ form a neutral
tannakian category over $\bbC$, hence there exists a simply connected unipotent algebraic group $\cU_n$ over $\bbC$ such that
$
	\Rep^\nil_\bbC(\fru_n) = \Rep_\bbC(\cU_n).
$

Let $\bbS_\bbC :=\bbG_m \times \bbG_m$ be the product over $\bbC$ of the multiple group $\bbG_m$ defined over $\bbC$.
We give an action of $\bbS_\bbC(\bbC)$ on the Lie algebra $\fru_n$ over $\bbC$ by 
\begin{equation}\label{eq: action}
	(x,y)\cdot T^{i,j}:=x^{-i}y^{-j}T^{i,j},
\end{equation} 
for any $(x,y)\in\bbC^\times\times\bbC^\times=\bbS_\bbC(\bbC)$,
hence by functoriality an action of the algebraic group $\bbS_\bbC$ on $\cU_n$.
If we denote by $\cU$ the projective limit of $\cU_n$, then $\bbS_\bbC$ acts on $\cU$,
and we let $\cG_\bbC:=\bbS_\bbC\ltimes\cU$ be the semi-direct product with respect to this action. 

We will show that $\cG_\bbC$ is the tannakian fundamental group of $\MHS_\bbC$.
To compare the categories $\Rep_\bbC(\cG_\bbC)$ and $\MHS_\bbC$, we give an explicit description of objects in $\Rep_\bbC(\cG_\bbC)$.

\begin{proposition}\label{prop: explicit}
	An object in $\Rep_\bbC(\cG_\bbC)$ corresponds to a triple $U=(U_\bbC,\{U^{p,q}\},t)$,
	where $U_\bbC$ is a finite dimensional $\bbC$-vector space, $\{U^{p,q}\}$ is a bigrading
	of $U_\bbC$ by $\bbC$-linear subspaces  
	\begin{equation*}
		U_\bbC=\bigoplus_{p,q\in\bbZ}U^{p,q},
	\end{equation*}
	and $t$ is a $\bbC$-linear automorphism of $U_\bbC$ satisfying
	\begin{equation}\label{eq: unipotence}
		(t-1)(U^{p,q})\subset\bigoplus_{\substack{r,s\in\bbZ \\r<p,\,s<q}}U^{r,s}
	\end{equation}
	for any $p,q\in\bbZ$.
	The morphisms in $\Rep_\bbC(\cG_\bbC)$ correspond to $\bbC$-linear homomorphisms of underlying $\bbC$-vector spaces
	compatible with the bigradings and commutative with $t$.
\end{proposition}

\begin{proof}
	Suppose that $U_\bbC$ is a finite $\bbC$-representation of the pro-algebraic group $\cG_\bbC$.  Then $U_\bbC$ is a representation of both $\bbS_\bbC$ and $\cU$,
	and
	\begin{equation*}
		U^{p,q} := \{  u \in U_\bbC \mid (x,y) \cdot u = x^{p} y^{q} u \text{ for all } (x,y)\in \bbS_\bbC(\bbC) \}
	\end{equation*}
	gives a bigrading of $U_\bbC$.  
	If $n$ is a sufficiently large natural number, then
	$U_\bbC$ is a representation of $\cU_n$, hence it is also a representation of $\fru_n$.
	Hence we have a nilpotent endomorphism $T^{i,j}: U_\bbC \rightarrow U_\bbC$ for any positive integers $i,j$.
	For any $u \in U^{p,q}$, we have $(x,y)\cdot(T^{i,j}(u)) = ((x,y)\cdot T^{i,j})((x,y)\cdot u)  = x^{p-i} y^{q-j}(T^{i,j}(u))$, 
	hence $T^{i,j}$ restricted to $U^{p,q}$ gives a morphism
	\begin{equation*}
		T^{i,j}: U^{p,q} \rightarrow U^{p-i,q-j}.
	\end{equation*}  
	If we let $T:= \sum_{i,j>0} T^{i,j}$, then $T$ is again a nilpotent endomorphism of $U_\bbC$, and $t:=\exp(T)$ satisfies \eqref{eq: unipotence} by construction.
	Hence $(U_\bbC, \{U^{p,q}\}, t)$ satisfies the required conditions.
	Conversely, suppose  $(U_\bbC, \{U^{p,q}\}, t)$ satisfies the conditions of the proposition.  Then we may define an action of $\bbS_\bbC(\bbC)$ on $U_\bbC$
	by $(x,y) \cdot u = x^p y^q u$ for any $(x,y) \in \bbS_\bbC(\bbC)$ and $u \in U^{p,q}$.	
	Furthermore, if we let $T:= \log(t) = \log(1 + (t-1))$, then $T$ is an endomorphism of $U_\bbC$ satisfying
	\begin{equation*}
		T(U^{p,q}) \subset \bigoplus_{\substack{r,s\in\bbZ\\r<p,s<q}} U^{r,s}
	\end{equation*}
	by \eqref{eq: unipotence}.  For positive integers $i,j>0$, we let $T^{i,j}: U_\bbC \rightarrow U_\bbC$ be the morphisms
	given as the direct sum of morphisms $U^{p,q} \rightarrow U^{p-i,q-j}$ induced from $T$, which
	gives a representation of the Lie algebra $\fru_n$ on $U_\bbC$ for a natural number $n$ sufficiently large.
	This shows that our representation gives a representation of $\fru_n$ on $U_\bbC$, hence a representation of the algebraic group $\cU_n$ on $U_\bbC$.
	This combined with the action of $\bbS_\bbC$ gives a representation of the algebraic group $\cG_\bbC=\bbS_\bbC\ltimes\cU$ on $U_\bbC$.  
	The above construction shows that a representation $U_\bbC$ of $\cG_\bbC$ is equivalent to the triple $(U_\bbC, \{U^{p,q}\}, t)$, proving our assertion.
\end{proof}

The category $\Rep_\bbC(\cG_\bbC)$ is known to be equivalent to the category of mixed $\bbC$-Hodge structures $\MHS_\bbC$.
We may define a functor $\varphi_\bbC:\Rep_\bbC(\cG_\bbC)\rightarrow\MHS_\bbC$ by associating to 
any object $U$ in $\Rep_\bbC(\cG_\bbC)$ the object 
\begin{equation}\label{eq: phi}
	\varphi_\bbC(U):=(V_\bbC,W_\bullet,F^\bullet,\ol F^\bullet),
\end{equation}
where $V_\bbC:=U_\bbC$, the weight and Hodge filtrations are defined by
\begin{equation*}
	W_nV_\bbC:=\bigoplus_{\substack{p,q\in\bbZ\\p+q\leq n}}U^{p,q}
\end{equation*}
for any $n\in\bbZ$, and 
\begin{align*}
	F^pV_\bbC&:=t\bigg(\bigoplus_{\substack{r,s\in\bbZ\\r\geq p}}U^{r,s}\biggr),  &
	\ol F^qV_\bbC&:=t^{-1}\bigg(\bigoplus_{\substack{r,s\in\bbZ\\s\geq q}}U^{r,s}\biggr)
\end{align*}
for any integers $p,q\in\bbZ$.  

\begin{proposition}[\cite{D3} Proposition 2.1]\label{prop: equiv for RMHS}
	The functor $\varphi_\bbC$ in \eqref{eq: phi} gives an equivalence of categories
	\begin{equation*}
		\Rep_\bbC(\cG_\bbC)\cong\MHS_\bbC.
	\end{equation*}
\end{proposition}

An quasi-inverse functor $\psi_\bbC:\MHS_\bbC\rightarrow\Rep_\bbC(\cG_\bbC)$ is given by associating to any $V\in\MHS_\bbC$ 
the object
\begin{equation}\label{eq: psi}
	\psi_\bbC(V):=(U_\bbC, \{ U^{p,q}\}, t)
\end{equation}
in $\Rep_\bbC(\cG_\bbC)$, where 
\begin{equation*}
	U_\bbC:= \bigoplus_{n\in\bbZ}\Gr^W_n V_\bbC,
\end{equation*}
$U^{p,q} = (F^p\cap\ol F^q)\Gr^W_{p+q}V_\bbC$ for any $p,q\in\bbZ$, and the $\bbC$-linear automorphism
$t$ is defined as follows: 
Let $\{A^{p,q}(V)\}$ be the Deligne splitting of $V$ given in \eqref{eq: Apq}.
By Lemma \ref{lem: Apq} we have an isomorphism
\begin{equation*}
	\bigoplus_{p+q=n} A^{p,q}(V) \xrightarrow\cong \bigoplus_{p+q=n}  (F^p \cap\ol F^q)\Gr^W_nV_\bbC = \Gr^W_nV_\bbC,
\end{equation*}
hence isomorphism
\begin{align*}
	\rho_\bbC:\bigoplus_{p,q\in\bbZ}A^{p,q}(V)&\rightarrow U_\bbC.
\end{align*}
Similarly, by Remark \ref{rem: splitting}, we have an isomorphism
\begin{align*}
	\ol\rho_\bbC:\bigoplus_{p,q\in\bbZ}\ol A^{p,q}(V)&\rightarrow U_\bbC. 
\end{align*}
We denote by $s$ the composition
\begin{equation*}
	U_\bbC\overset{\rho^{-1}_\bbC}{\longrightarrow}\bigoplus_{p,q\in\bbZ}A^{p,q}(V)=V_\bbC=\bigoplus_{p,q\in\bbZ}\ol A^{p,q}(V)\overset{\ol\rho_\bbC}{\longrightarrow}U_\bbC.
\end{equation*}
Then it is known that $s$ is unipotent, and $t$ is defined by
\begin{equation}\label{eq: t}
	t:=\sqrt{s}=\sum_{k=0}^\infty\bigg(\begin{array}{c}1/2\\ k\end{array}\bigg)(s-1)^k.
\end{equation}
Then we may prove that $\psi_\bbC\circ\varphi_\bbC=\id$ and $\varphi_\bbC\circ\psi_\bbC\simeq\id$.
The isomorphism of functors $\id \simeq\varphi_\bbC\circ\psi_\bbC$ is given by the composition
\begin{equation}\label{eq: equivalence isom}
	V_\bbC = \bigoplus_{p,q\in\bbZ}A^{p,q}(V)\overset{\rho_\bbC}{\longrightarrow}  \bigoplus_{n\in\bbZ}\Gr^W_n V_\bbC
	 \overset{t}{\longrightarrow} \bigoplus_{n\in\bbZ}\Gr^W_n V_\bbC
\end{equation}
for any object $V$ in $\MHS_\bbC$.\medskip

%
\subsection{The category $\Rep_\bbC(\cG^g_\bbC)$.}
%
Recall that $\cG_\bbC$ denotes the tannakian fundamental group of $\MHS_\bbC$ with respect to $\omega$. 
Let $g$ be an integer $\geq 0$.  In \cite[\S 16]{NS}, Nekov\'{a}\v{r} and Scholl defined the category of 
mixed $g$-plectic $\bbC$-Hodge structures to be the category $\Rep_\bbC(\cG^g_\bbC)$ of finite 
dimensional $\bbC$-linear representations of the $g$-fold product $\cG^g_\bbC:= \cG_\bbC \times \cdots \times \cG_\bbC$.
As a direct generalization of Proposition \ref{prop: explicit}, we have the following explicit description of objects in $\Rep_\bbC(\cG^g_\bbC)$.

\begin{proposition}\label{prop: gexplicit}
	A finite dimensional $\bbC$-linear representation of $\cG^g_\bbC$ corresponds to a triple
	$U := (U_\bbC, \{  U^{\bsp,\bsq} \}, \{ t_\mu\})$, where $U_\bbC$ is a finite dimensional $\bbC$-vector space,
	$\{ U^{\bsp,\bsq} \}$ is a $2g$-grading of $U_\bbC$ by $\bbC$-linear subspaces 
	\begin{equation*}
		U_\bbC = \bigoplus_{\bsp,\bsq \in \bbZ^g} U^{\bsp,\bsq},
	\end{equation*}
	and $t_\mu$ for $\mu =1, \ldots, g$ are 
 	$\bbC$-linear automorphisms of $U_\bbC$ commutative with each other, satisfying 
 	\begin{equation*}
	 	(t_\mu -1)(U^{\bsp,\bsq}) \subset \bigoplus_{\substack{\bsr,\bss\in\bbZ^g\\ (r_\nu,s_\nu)=(p_\nu,q_\nu)\text{ for }
		\nu\neq\mu\\(r_\mu,s_\mu) < (p_\mu, q_\mu)\quad}} U^{\bsr\!,\!\,\bss}
 	\end{equation*}
	for any $\bsp, \bsq\in \bbZ^g$, where the direct sum is over the indices
	$\bsr, \bss  \in \bbZ^g$ satisfying $r_\nu=p_\nu$, $s_\nu=q_\nu$ for $\nu\neq\mu$ and
	$r_\mu < p_\mu$, $s_\mu < q_\mu$.
	Morphisms in  $\Rep_\bbC(\cG^g_\bbC)$ correspond to $\bbC$-linear homomorphisms of underlying $\bbC$-vector spaces compatible 
	with the $2g$-gradings and commutes with $t_\mu$.
\end{proposition}

\begin{proof}
	For eqch $\mu=1,\ldots,g$, let $\{ U^{p_\mu,q_\mu}_\mu\} $ be the bigrading and $t_\mu$ the $\bbC$-linear automorphism of $U_\bbC$
	given by the action of the $\mu$-th component of $\cG^g_\bbC$.
	For any $\bsp,\bsq\in\bbZ^g$, let $U^{\bsp,\bsq} := U^{p_1,q_1}_1 \cap \cdots \cap U^{p_g,q_g}_g$.
	Our conditions on  $\{ U^{\bsp,\bsq} \}$  and $\{t_\mu\}$ 
	correspond to the commutativity of the actions of the $g$ components.
\end{proof}

The tensor product and the internal homomorphism in $\Rep_\bbC(\cG^g_\bbC)$ are given as follows.  Suppose $T = (T_\bbC, \{T^{\bsp,\bsq}\}, \{t'_\mu\})$ and 
$U=(U_\bbC,\{U^{\bsp,\bsq}\}, \{t''_\mu\})$ are object in $\Rep_\bbC(\cG^g_\bbC)$.
Then the tensor product $T \otimes U$ is given by the triple
\begin{equation}\label{eq: ttensor}
	T \otimes U = (T_\bbC\otimes_\bbC U_\bbC,  \{(T\otimes U)^{\bsp,\bsq}\}, \{t_\mu\}),
\end{equation}
where $T_\bbC\otimes_\bbC U_\bbC$ is the usual tensor product over $\bbC$,
\begin{equation*}
	(T\otimes U)^{\bsp,\bsq} = \bigoplus_{\substack{\bsp',\bsq',\bsp'',\bsq''\in\bbZ^g\\\bsp'+\bsp''=\bsp,\,\bsq'+\bsq''=\bsq}} T^{\bsp',\bsq'}\otimes_\bbC U^{\bsp'',\bsq''}
\end{equation*}
for any $\bsp,\bsq\in\bbZ^g$, and $t_\mu := t'_\mu\otimes t''_\mu$ for $\mu=1,\ldots,g$.  The internal homomorphism $\ul\Hom(T,U)$ is given by the triple
\begin{equation}\label{eq: tihom}
	\ul\Hom(T, U) = (\Hom_\bbC(T_\bbC, U_\bbC),  \{\Hom(T,U)^{\bsp,\bsq}\}, \{t_\mu\}),
\end{equation}
where $\Hom_\bbC(T_\bbC, U_\bbC)$ is the set of $\bbC$-linear homomorphisms of $T_\bbC$ to $U_\bbC$,
\begin{equation*}
	\Hom(T,U)^{\bsp,\bsq} =  \bigl\{\alpha\in \Hom_\bbC(T_\bbC, U_\bbC)  \mid\alpha(T^{\bsp',\bsq'}) \subset U^{\bsp'+\bsp,\bsq'+\bsq}\,\,
	\forall \bsp',\bsq'\in\bbZ^g\bigr\}
\end{equation*}
for any $\bsp,\bsq\in\bbZ^g$, and $t_\mu(\alpha) := t''_\mu \circ\alpha\circ t'^{-1}_\mu$ for any $\alpha\in\Hom_\bbC(T_\bbC, U_\bbC)$ and $\mu=1,\ldots,g$.

\begin{example}[Tate object]\label{example: tate}
	One of the simplest examples of an object in $\Rep_\bbC(\cG^g_\bbC)$  is the plectic Tate object
	\[
		\bbC(\boldsymbol 1_\mu) := (V_\bbC, \{V^{\bsp,\bsq}\}, \{t_\mu\}),
	\]
	where $V_\bbC :=  \bbC$ and the grading of $V_\bbC$
	is such that $V^{\bsp,\bsq} = V_\bbC$ if
	\begin{equation*}
		\bsp=\bsq = (0, \ldots, -1, \ldots, 0),
	\end{equation*}
	where $-1$ is at the $\mu$-th component, and $V^{\bsp,\bsq}=0$ otherwise,
	and $t_\mu$ is the identity map for $\mu=1, \ldots, g$.  For any $\bsn \in\bbZ^g$, we let 
	\begin{equation*}
		\bbC(\bsn) := \bigotimes_{\mu=1}^g \bbC(\boldsymbol 1_\mu)^{\otimes n_\mu} =  \bbC(\boldsymbol 1_1)^{\otimes n_1} \otimes \cdots 
		\otimes \bbC(\boldsymbol 1_g)^{\otimes n_g}.
	\end{equation*}
\end{example}

\begin{remark}\label{rem: inclusion}
	For any positive integer $\mu \leq g$, the natural projection $\cG^g_\bbC \rightarrow \cG^\mu_\bbC$ of pro-algebraic groups
	mapping $(u_1,\ldots,u_\mu,u_{\mu+1},\ldots,u_g)$ to $(u_1,\ldots,u_\mu)$ induces a natural functor
	$\Rep_\bbC(\cG^\mu_\bbC) \rightarrow \Rep_\bbC(\cG^g_\bbC)$, and the category $\Rep_\bbC(\cG^\mu_\bbC)$ is a full subcategory of
	 $\Rep_\bbC(\cG^g_\bbC)$ with respect to this functor.
	On the level of objects, this functor may be given by	
	associating to any 
	\begin{equation*}
		U' = (U_\bbC, \{ U^{\bsp',\bsq'}\}, \{ t'_\nu\} )
	\end{equation*} 
	in $\Rep_\bbC(\cG^\mu_\bbC)$ the object $U = (U_\bbC, \{ U^{\bsp,\bsq}\}, \{ t_\nu\})$
	in $\Rep_\bbC(\cG^g_\bbC)$, where the bigrading is defined by
	\begin{equation*}
		U^{\bsp,\bsq} := U^{(p_1, \ldots, p_{\mu}), (q_1, \ldots, q_{\mu})} 
	\end{equation*}
	if $(p_{\mu+1}, \ldots, p_g) = (q_{\mu+1}, \ldots, q_g) = (0, \ldots, 0)$ and $U^{\bsp,\bsq} := 0$ otherwise,
	and we let the automorphisms $t_\nu$ be $t_\nu := t'_\nu$ for $1 \leq \nu \leq \mu$ and $t_{\nu}:=\id$ for $\mu < \nu \leq g$.
\end{remark}

\begin{remark}\label{rem: exterior product}
	Let $g_1,g_2$ be integers $>0$, and let $T = (T_\bbC, \{T^{\bsp_1,\bsq_1}\}, \{t'_\mu\})$ and
	$U=(U_\bbC,\{U^{\bsp_2,\bsq_2}\}, \{t''_\mu\})$ be objects respectively in $\Rep_\bbC(\cG^{g_1}_\bbC)$ and $\Rep_\bbC(\cG^{g_2}_\bbC)$.
	Then the exterior product $T \boxtimes U$ in $\Rep_\bbC(\cG^{g_1+g_2}_\bbC)$ corresponds to the triple
	\begin{equation*}
		T\boxtimes U:= (T_\bbC\otimes_\bbC U_\bbC, (T\boxtimes U)^{\bsp,\bsq}, \{ t_\mu\}),
	\end{equation*}
	where $T_\bbC\otimes_\bbC U_\bbC$ is the usual tensor product over $\bbC$, 
	\begin{equation*}
		 (T\boxtimes U)^{\bsp,\bsq}=  T^{\bsp_1,\bsq_1} \otimes_\bbC  U^{\bsp_2,\bsq_2}
	\end{equation*}
	with the convention that $\bsp_1:= (p_1,\ldots, p_{g_1}), \bsp_2:= (p_{g_1+1},\ldots,p_{g_1+g_2})$,
	$\bsq_1 := (q_1,\ldots, q_{g_1})$, and $\bsq_2:= (q_{g_1+1},\ldots,q_{g_1+g_2})$ for any $\bsp = (p_\mu),\bsq= (q_\mu)\in\bbZ^{g_1+g_2}$,
	and $t_\mu$ is the $\bbC$-linear automorphism on $T_\bbC\otimes_\bbC U_\bbC$ given by $t_\mu = t'_\mu\otimes 1$ for $\mu=1,\ldots,g_1$
	and $t_\mu=1\otimes t''_{\mu-g_1}$ for $\mu=g_1+1,\ldots,g_1+g_2$.
\end{remark}

%
%
%
%
%
\section{Orthogonal families of mixed $\bbC$-Hodge structures}\label{section: OF}
%
%
%
%
%

Let $\cG_\bbC$ be the tannakian fundamental group of $\MHS_\bbC$ with respect to $\omega$.
The purpose of this section is to prove an equivalence of categories between the category $\Rep_\bbC(\cG^g_\bbC)$ and the 
category of $g$-orthogonal family of mixed $\bbC$-Hodge structures $\OF^g_\bbC$ defined in Definition \ref{def: orthogonal}.

%
\subsection{Categorical version of mixed $\bbC$-Hodge structures}
%

In this subsection, we will give an iterated description of the category $\Rep_\bbC(\cG^g_\bbC)$,
using the categorical version of mixed Hodge structures.
Using the result of Proposition \ref{prop: explicit} as an inspiration,
we first define the category of bigraded objects $\BG(\sA)$ for an abelian category $\sA$ as follows.

\begin{definition}
	We let $\BG(\sA)$ be the category whose objects consist of  a triple $U = (B, \{B^{p,q}\}, t)$, where $B$ is an object of 
	$\sA$, $\{B^{p,q}\}$ is a bigrading of $B$ by subobjects in $\sA$
	\[
		B = \bigoplus_{p,q\in\bbZ} B^{p,q},
	\]
	where $B^{p,q}=0$ for all but finitely many $(p,q)\in\bbZ^2$,
	and $t$ is an automorphism of $B$ satisfying
	\[
		(t-1)(B^{p,q}) \subset \bigoplus_{\substack{r,s\in\bbZ\\r <p,s<q}} B^{r,s}
	\]
	for any $p,q\in\bbZ$.  The morphisms in  $\BG(\sA)$ are morphisms of underlying objects in $\sA$
	compatible with the bigradings and commutative with $t$.
\end{definition}

If $\sA$ is the category of finite dimensional $\bbC$-vector spaces $\Vec_\bbC$,
then Proposition \ref{prop: explicit} shows that $\BG(\Vec_\bbC)$ is equivalent to 
the category $\Rep_\bbC(\cG_\bbC)$ of finite dimensional $\bbC$-representations of $\cG_\bbC$.

\begin{proposition}\label{prop: Rep G}
	For any integer $g > 0$, we have an isomorphism of categories
	\[
		\Rep_\bbC(\cG^g_\bbC) = \BG( \Rep_\bbC(\cG^{g-1}_\bbC)).
	\]
\end{proposition}

\begin{proof}
	Let $U=(U_\bbC,\{U^{\bsp,\bsq} \}, \{t_\mu\})$ be an object in $\Rep_\bbC(\cG^{g}_\bbC)$.
	For $\bsp'=(p_\mu),\bsq'=(q_\mu) \in \bbZ^{g-1}$, if we let
	\[
		U^{\bsp',\bsq'} : = \bigoplus_{p,q\in\bbZ} U^{(p_1,\ldots,p_{g-1}, p),(q_1,\ldots,q_{g-1}, q)}
	\]
	and $t'_\mu := t_\mu$ for $\mu=1,\ldots,g-1$, then the triple
	$B:=(U_\bbC,\{U^{\bsp',\bsq'} \}, \{t'_\mu\})$ defines an object in $\Rep_\bbC(\cG^{g-1}_\bbC)$.
	For any $p,q\in\bbZ$, if we let
	\begin{equation*}
		B^{p,q} :=  \bigoplus_{\substack{\bsp,\bsq\in\bbZ^g\\p_g=p,q_g=q}} U^{\bsp,\bsq}, 	
	\end{equation*}
	$(B^{p,q})^{\bsp',\bsq'}  :=  U^{(p_1,\ldots,p_{g-1},p),(q_1,\ldots,q_{g-1},q)}$ for any $\bsp',\bsq'\in\bbZ^{g-1}$ 
	and $t'_\mu:= t_\mu|_{B^{p,q}}$ for $\mu=1,\ldots,g-1$, then the triple $B^{p,q}:=(B^{p,q}, (B^{p,q})^{\bsp',\bsq'}, \{t'_\mu\})$
	defines an object in $\Rep_\bbC(\cG^{g-1}_\bbC)$.  If we let $t:=t_g$, then we see that the triple $(B, \{B^{p,q}\}, t)$ gives an object in $\BG( \Rep_\bbC(\cG^{g-1}_\bbC))$.
	Conversely, let $(B, B^{p,q}, t)$ be an object in $\BG( \Rep_\bbC(\cG^{g-1}_\bbC))$.
	Then $B$ is an object in $\Rep_\bbC(\cG^{g-1}_\bbC)$ hence is of the form $B = (U_\bbC, \{ U^{\bsp',\bsq'}\}, \{t'_\mu\})$.
	Since $B^{p,q}$ is an object in $\Rep_\bbC(\cG^{g-1}_\bbC)$, it is also of the form 
	$B^{p,q} = (U^{p,q}_\bbC, \{ (U^{p,q})^{\bsp',\bsq'}\}, \{t'_\mu\}))$.
	If we let 
	\[
		U^{\bsp,\bsq} := (U^{p_g,q_g})^{(p_1,\ldots,p_{g-1}),(q_1,\ldots,q_{g-1})}
	\]
	and $t_\mu:=t'_\mu$ for $\mu=1,\ldots,g-1$ and $t_g:=t$, then the triple $(U_\bbC, \{U^{\bsp,\bsq}\}, \{t_\mu\})$ gives
	an object in $\Rep_\bbC(\cG^g_\bbC)$.  The automorphism $t_g$ is commutative with $t_1,\ldots,t_{g-1}$ since
	$t$ is a morphism in $\Rep_\bbC(\cG^{g-1}_\bbC)$.
	The above constructions are inverse to each other, hence we have the desired isomorphism of categories.
\end{proof}

\begin{definition} 
	Let $A$ be an object in $\sA$.
	Let
	$W_\bullet$ be a finite ascending filtration by subobjects of $A$, and let $F^\bullet$ and $\ol F^\bullet$ be
	finite descending filtrations by subobjects of $A$.
	We say that the quadruple $V = (A, W_\bullet, F^\bullet, \ol F^\bullet)$ is a \textit{mixed Hodge structure} in $\sA$,
	if for each $n \in \bbZ$, the structure induced by $F^\bullet$ and $\ol F^\bullet$ on $\Gr^W_n A$ satisfies
	\[
		\Gr^W_n A = F^p \Gr^W_n A \oplus \ol F^{n+1-p} \Gr^W_n A
	\]
	for any $p \in \bbZ$.
\end{definition}

If $V=(A,  W_\bullet, F^\bullet, \ol F^\bullet)$ is a mixed Hodge structure in $\sA$, then we call $W_\bullet$
the \textit{weight filtration} and $F^\bullet$, $\ol F^\bullet$ the \textit{Hodge filtrations} of $V$.
We denote by $\MHS(\sA)$ the category whose objects consist of mixed Hodge structures in $\sA$ and whose morphisms are 
morphisms of underlying objects in $\sA$ compatible with the weight and Hodge filtrations.
If $\sA$ is the category $\Vec_\bbC$ of finite dimensional $\bbC$-vector spaces,
then we have an isomorphism of categories $\MHS(\Vec_\bbC) = \MHS_\bbC$.

As in the case of mixed $\bbC$-Hodge structures, we have the Deligne splitting for mixed Hodge structures in $\sA$.

\begin{proposition}[\cite{D3} \S1.1]
	Let $V = (A, W_\bullet, F^\bullet, \ol F^\bullet)$ be a mixed Hodge structure in $\sA$,
	and as in \eqref{eq: Apq}, we let
	\begin{equation*}
		A^{p,q}(V) := (F^p \cap W_{n}) A \cap \biggl((\ol F^q \cap W_{n})A+\sum_{j\geq 0}(\ol F^{q-j}\cap W_{n-j-1}) A\biggr)
	\end{equation*}	
	for $p,q\in\bbZ$ and $n:=p+q$.
	Then $\{A^{p,q}(V)\}$ gives a bigrading of $A$ by subobjects of $\sA$,
	\begin{equation}
		A =\bigoplus_{p,q \in \bbZ} A^{p,q}(V).
	\end{equation}
	Moreover, for $n, p\in\bbZ$, the weight and Hodge filtrations on $V$ satisfy
	\begin{align*}
		W_n A &=\bigoplus_{\substack{p,q\in\bbZ\\p+q \leq n}} A^{p,q}(V), &
		F^p A &=\bigoplus_{\substack{r,s\in\bbZ\\r \geq p}} A^{r,s}(V).
	\end{align*}	
\end{proposition}

As in the case of mixed $\bbC$-Hodge structures given in Remark \ref{rem: splitting}, 
a similar statement holds for $\ol A^{p,q}$, where $\ol A^{p,q}$ is defined by replacing the roles
of $F^\bullet$ and $\ol F^\bullet$.
As in the case of mixed $\bbC$-Hodge structures, the morphisms in $\MHS(\sA)$ are strictly compatible with the filtrations,
and we may prove that $\MHS(\sA)$ is an abelian category.  

We define the functor $\varphi:\BG(\sA)\rightarrow\MHS(\sA)$ by associating to 
any object $U=(B,\{B^{p,q}\},t)$ in $\BG(\sA)$ the object 
\begin{equation*}
	\varphi(U):=(A ,W_\bullet,F^\bullet,\ol F^\bullet),
\end{equation*}
where $A:=B$, the weight and Hodge filtrations are defined by
\begin{equation*}
	W_n A:=\bigoplus_{\substack{p,q\in\bbZ\\p+q\leq n}}B^{p,q}
\end{equation*}
for any $n\in\bbZ$ and 
\begin{align*}
	F^p A&:=t\bigg(\bigoplus_{\substack{r,s\in\bbZ\\r\geq p}}B^{r,s}\biggr),  &
	\ol F^q A&:=t^{-1}\bigg(\bigoplus_{\substack{r,s\in\bbZ\\s\geq q}} B^{r,s}\biggr)
\end{align*}
for any integers $p,q\in\bbZ$.  
By \cite[Proposition 1.2 and Remark 1.3]{D3}, we have the following result.

\begin{proposition}
	The functor $\varphi$ gives an equivalence of categories
	\[
		\varphi: \BG( \sA) \cong \MHS(\sA).
	\]
\end{proposition}

We can define a quasi-inverse functor $\psi$ as in \eqref{eq: psi}.\medskip

Next, for any integer $g \geq 0$, we inductively define the category $\MHS^g(\Vec_\bbC)$ by
$\MHS^0(\Vec_\bbC) := \Vec_\bbC$ and $\MHS^g(\Vec_\bbC) := 
\MHS( \MHS^{g-1}( \Vec_\bbC))$ for $g>0$.
Combining this result with Proposition \ref{prop: Rep G}, we have the following corollary.

\begin{corollary}\label{cor: mhs}
	We have equivalences of categories
	\begin{equation}\label{eq: functor4}
		\Rep_\bbC(\cG^g_\bbC) \cong \MHS( \Rep_\bbC(\cG^{g-1}_\bbC) ) \cong
		\cdots\cong \MHS^g(\Vec_\bbC).
	\end{equation}
\end{corollary}

In \S \ref{subsection: OF}, we will use this result to prove that $\Rep_\bbC(\cG^g_\bbC)$ is equivalent to
the category of $g$-orthogonal family of mixed $\bbC$-Hodge structures.

%
%
%
%
%
\subsection{Orthogonal families of mixed $\bbC$-Hodge structures}\label{subsection: OF}
%
%
%
%
%

In this subsection, we will define the category of $g$-orthogonal family of mixed $\bbC$-Hodge structures and
show that this category is equivalent to the category $\Rep_\bbC(\cG^g_\bbC)$.
We first define the category of multi-filtered $\bbC$-vector spaces $\Fil^l_{m}(\bbC)$.

\begin{definition}\label{def: multi-filtered}
	Let $l$ and $m$ be non-negative integers.
	An object in the category $\Fil^l_{m}(\bbC)$ is a quadruple $V=(V_\bbC,\{W^\lambda_\bullet\},\{F^\bullet_\mu\},\{\ol F^\bullet_\mu\})$ consisting of
	a finite dimensional $\bbC$-vector space $V_\bbC$,
	a family of finite ascending filtrations $W^\lambda_\bullet$ for $\lambda=1,\ldots,l$ by $\bbC$-linear subspaces on $V_\bbC$,
	and families of finite descending filtrations $F_\mu^\bullet$ and $\ol F_\mu^\bullet$ for $\mu=1,\ldots,m$ by $\bbC$-linear subspaces on $V_\bbC$.
	A morphism in $\Fil^l_{m}(\bbC)$ is a $\bbC$-linear homomorphism compatible with $W^\lambda_\bullet$, $F_\mu^\bullet$, and $\ol F_\mu^\bullet$.
\end{definition}

We define the notion of a $g$-orthogonal family of mixed $\bbC$-Hodge structures as follows.

 \begin{definition}[Orthogonal Family]\label{def: orthogonal}
 	We say that an object  $(V_\bbC, \{W^\mu_\bullet\}, \{F^\bullet_\mu\},\{\ol F^\bullet_\mu\})$ in $\Fil^g_{g}(\bbC)$ 
	is a \textit{$g$-orthogonal family of mixed $\bbC$-Hodge structures}, 
	if for any $\mu$, the quadruple $(V_\bbC, W^\mu_\bullet, F^\bullet_\mu, \ol F^\bullet_\mu)$ is a mixed $\bbC$-Hodge structure, 
	and for any $\mu$ and $\nu \neq \mu$, the $\bbC$-linear subspaces 
	$W^\mu_n V_\bbC$, $F^m_\mu V_\bbC$, $\ol F^m_\mu V_\bbC$
	with the weight and Hodge  filtrations induced from $W^\nu_\bullet$, $F_\nu^\bullet$, $\ol F_\nu^\bullet$ are mixed $\bbC$-Hodge structures.
	We denote by $\OF^g_\bbC$ the full subcategory of $\Fil^g_{g}(\bbC)$ whose objects are 
	$g$-orthogonal family of mixed $\bbC$-Hodge structures.
 \end{definition}
 
 If $V=(V_\bbC, \{W^\mu_\bullet\}, \{F^\bullet_\mu\},\{\ol F^\bullet_\mu\})$ is a $g$-orthogonal family of mixed $\bbC$-Hodge structures,
 then we call $\{W^\mu_\bullet\}$ the weight filtrations and  $\{F^\bullet_\mu\}$, $\{ \ol F^\bullet_\mu\}$ the Hodge filtrations of $V$.
 Note that $\OF^1_\bbC = \MHS_\bbC$.
 
Next, let $\MHS^g(\Vec_\bbC)$ be as in Corollary \ref{cor: mhs}.  An object $A$ in $\MHS^g(\Vec_\bbC)$ consists of a 
finite dimensional $\bbC$-vector space $V_\bbC$ with additional structures.
Then there exists a natural functor
\begin{equation}\label{eq: functor}
	\MHS^g(\Vec_\bbC) \rightarrow \Fil^g_{g}(\bbC)
\end{equation}
by associating to an object $A$ its underlying $\bbC$-vector space $V_\bbC$,
with the $\mu$-th weight and Hodge filtrations given by the image of the $\mu$-th weight and Hodge filtrations 
of $\MHS^g(\Vec_\bbC)$.  More precisely, for any $\mu=1,\ldots, g$, there exists an object $A^\mu$ in $\MHS^\mu(\Vec_\bbC)$
which underlies $A$, with the weight and Hodge filtrations $W^\mu_\bullet$, $F^\bullet_\mu$, $\ol F^\bullet_\mu$ given by subobjects of $A^\mu$
in $\MHS^{\mu-1}(\Vec_\bbC)$.  Then we define the filtrations $W^\mu_\bullet$, $F^\bullet_\mu$, $\ol F^\bullet_\mu$ by $\bbC$-linear subspaces on $V_\bbC$ 
to be the filtrations given as the images of the subobjects $W^\mu_\bullet$, $F^\bullet_\mu$, $\ol F^\bullet_\mu$ of $A^\mu$.

\begin{remark}\label{rem: reordering}
	Combining \eqref{eq: functor} with the functor in Corollary \ref{cor: mhs}, we have a functor
	\begin{equation}\label{eq: functor2}
		\varphi^g_\bbC: \Rep_\bbC(\cG^g_\bbC) \rightarrow \Fil^g_{g}(\bbC).
	\end{equation}
	By definition, this functor associates to an object $U=(U_\bbC, \{ U^{\bsp,\bsq}\}, \{t_\mu\})$ in $\Rep_\bbC(\cG^g_\bbC)$ the object
	$V:=(V_\bbC, \{W^\mu_\bullet\}, \{F^\bullet_\mu\},\{\ol F^\bullet_\mu\})$, where $V_\bbC := U_\bbC$,
	\begin{equation*}
		W^\mu_n V_\bbC :=\bigoplus_{\substack{\bsp,\bsq\in\bbZ^g\\p_\mu+q_\mu\leq n}}U^{\bsp,\bsq}
	\end{equation*}
	for any $n\in\bbZ$ and 
	\begin{align*}
		F^p_\mu V_\bbC&:=t_\mu\bigg(\bigoplus_{\substack{\bsr,\bss\in\bbZ^g\\r_\mu\geq p}}U^{\bsr,\bss}\biggr),  &
		\ol F^q_\mu V_\bbC&:=t_\mu^{-1}\bigg(\bigoplus_{\substack{\bsr,\bss\in\bbZ^g\\s_\mu\geq q}} U^{\bsr,\bss}\biggr)
	\end{align*}
	for any integers $p,q\in\bbZ$.  This shows that the functor \eqref{eq: functor2} is defined independently of the ordering of the index $\mu=1,\ldots,g$,
	hence if $V =(V_\bbC, \{W^\mu_\bullet\}, \{F^\bullet_\mu\},\{\ol F^\bullet_\mu\})\in  \Fil^g_{g}(\bbC)$ is an object in the essential image of the functor 
	\eqref{eq: functor}, then the object $V'=(V_\bbC, \{W^{\mu'}_\bullet\}, \{F^\bullet_{\mu'}\},\{\ol F^\bullet_{\mu'}\})$ given by a reordering $\mu' = \sigma(\mu)$ 
	of the index for some bijection $\sigma: \{1,\ldots,g\} \rightarrow \{1, \ldots,g\}$ is also in the essential image of \eqref{eq: functor}.
\end{remark}


\begin{theorem}\label{thm: main}
	For any integer $g\geq 0$, the functor \eqref{eq: functor} gives an isomorphism of categories
	\begin{equation}\label{eq: functor3}
		\MHS^g(\Vec_\bbC) \xrightarrow\cong \OF^g_\bbC.
	\end{equation}
\end{theorem}

\begin{proof}
	The statement is trivial for $g=0$.  Assume $g>0$, and let $A$ be an object in $\MHS^g(\Vec_\bbC)$, 
	and let $V = (V_\bbC, \{W^\mu_\bullet\},\{F^\bullet_\mu\},\{\ol F^\bullet_\mu\})$ be the image of $A$ in 
	$\Fil^g_{g}(\bbC)$ with respect to the functor \eqref{eq: functor}.  
	Then by construction, for any $\mu =1,\ldots, g$, the quadruple $(V_\bbC, W^\mu_\bullet, F^\bullet_\mu, \ol F^\bullet_\mu)$
	is a mixed $\bbC$-Hodge structure.  Furthermore, for any index $\nu < \mu$, the $\bbC$-linear subspaces $W^\mu_n V_\bbC$, $F^p_\mu V_\bbC$,
	$\ol F^p_\mu V_\bbC$ with the weight and Hodge filtrations induced from $W^\nu_\bullet$, $F^\bullet_\nu$, $\ol F^\bullet_\nu$ are mixed $\bbC$-Hodge
	structures.  Remark \ref{rem: reordering} shows that since we may reorder the index of the filtrations, hence by reordering the filtrations,
	we see that the $\bbC$-linear subspaces $W^\mu_n V_\bbC$, $F^p_\mu V_\bbC$,
	$\ol F^p_\mu V_\bbC$ with the weight and Hodge filtrations induced from $W^\nu_\bullet$, $F^\bullet_\nu$, $\ol F^\bullet_\nu$ are mixed $\bbC$-Hodge
	structures even for the case $\nu>\mu$.  This shows that $V$ is an object in $\OF^g_\bbC$, hence we see that the functor \eqref{eq: functor}
	induces the functor \eqref{eq: functor3}.
	
	Conversely,  let $V = (V_\bbC, \{W^\mu_\bullet\},\{F^\bullet_\mu\},\{\ol F^\bullet_\mu\})$ be an object in $\OF^g_\bbC$.
	Then for $\mu=1,\ldots,g$, the $\bbC$-linear subspaces $W^\mu_n V_\bbC$, $F^p_\mu V_\bbC$, $\ol F^p_\mu V_\bbC$ 
	with the weight and Hodge filtrations induced from $W^\nu_\bullet$, $F^\bullet_\nu$, $\ol F^\bullet_\nu$ for $\nu\neq\mu$ 
	are mixed $\bbC$-Hodge structures,
	hence the decomposition
	\[
		\Gr^{W^\mu}_n V_\bbC = F^p_\mu \Gr^{W^\mu}_n V_\bbC \oplus  \ol F^{n+1-p}_\mu \Gr^{W^\mu}_n V_\bbC
	\]
	is also a decomposition of mixed $\bbC$-Hodge structures.  This shows that $V$ gives an object in $\MHS^g(\Vec_\bbC)$.
	The above constructions are inverse to each other, hence we have the isomorphism of categories \eqref{eq: functor3} as desired.
\end{proof}

By combining Corollary \ref{cor: mhs} and Theorem \ref{thm: main}, we have the following.

\begin{corollary}\label{cor: main cor}
	For $g\geq 0$, the functor $\varphi^g_\bbC$ of \eqref{eq: functor2} gives an equivalence of categories
	\begin{equation}\label{eq: equivalence}
		\varphi^g_\bbC:  \Rep_\bbC(\cG^g_\bbC) \cong\OF^g_\bbC.
	\end{equation}	
\end{corollary}
We denote by $\psi^g_\bbC$ the quasi-inverse functor of $\varphi^g_\bbC$ obtained as the composition of the
inverse functor of \eqref{eq: functor3} with the quasi-inverse functor of \eqref{eq: functor4}.
%
%
%
%

%
%
%
%
%
\section{Mixed plectic $\bbC$-Hodge structures}\label{section: pMHS}
%
%
%
%
%
The main result of this section is Proposition \ref{prop: char of MHS}, which characterizes $g$-orthogonal families in terms of the total weight filtration instead of the partial weight filtrations.
First we will define the notion of a mixed weak $g$-plectic $\bbC$-Hodge structure as an object in $\Fil^1_g(\bbC)$ having the plectic Hodge decomposition and good systems of representatives of the decomposition.
A mixed $g$-plectic $\bbC$-Hodge structure will be defined to be a mixed weak $g$-plectic $\bbC$-Hodge structure satisfying certain compatibility of filtrations.
Then we will see that there is an isomorphism between the category $\OF_\bbC^g$ of $g$-orthogonal families of mixed $\bbC$-Hodge structures
and the category $\MHS^g_\bbC$ of mixed $g$-plectic $\bbC$-Hodge structures.

%
\subsection{Mixed weak plectic $\bbC$-Hodge structures}
%
In this subsection, we will define the category $\wMHS_\bbC^g$ of mixed weak $g$-plectic $\bbC$-Hodge structures.
In what follows, for any index $\bsn =(n_\mu) \in \bbZ^g$, we let $\lvert\bsn\rvert := n_1+\cdots + n_g$.  
Furthermore, for $\bsr = (r_\mu), \bsp = (p_\mu) \in \bbZ^g$, we say that $\bsr \geq \bsp$ if $r_\mu\geq p_\mu$ 
for any $\mu = 1,\ldots,g$.

For non-negative integers $l$ and $m$, we let $\Fil^l_{m}(\bbC)$ be the category of
multi-filtered $\bbC$-vector spaces defined in Definition \ref{def: multi-filtered}.
For an object $V=(V_\bbC,\{W^\lambda_\bullet\},\{F_\mu^\bullet\},\{\ol F_\mu^\bullet\})$ in $\Fil^l_{g}(\bbC)$ and a subset $I\subset\{1,\ldots,g\}$, we define the \textit{plectic filtrations} $\bsF_I^\bullet,\ol\bsF_I^\bullet$ and the \textit{total filtrations} $F_I^\bullet,\ol F_I^\bullet$ on $V_\bbC$ associated to $\{F_\mu^\bullet\}$ and $\{\ol F_\mu^\bullet\}$ with respect to $I$ by
\begin{align}\label{eq: plectic Hodge}
	\bsF_I^\bsp V_\bbC:=\bigcap_{\mu\not\in I}F_\mu^{p_\mu}V_\bbC\cap\bigcap_{\nu\in I}\ol F_\nu^{p_\nu}V_\bbC,&&
	\ol\bsF_I^\bsp V_\bbC:=\bigcap_{\mu\not\in I}\ol F_\mu^{p_\mu}V_\bbC\cap\bigcap_{\nu\in I}F_\nu^{p_\nu}V_\bbC
	\end{align}
for any $\bsp = (p_\mu) \in \bbZ^g$,
and
\begin{align}\label{eq: total Hodge}
	F^p_IV_\bbC:= \sum_{\substack{\bsp\in\bbZ^g,\,\lvert\bsp\rvert=p}} \bsF_I^\bsp V_\bbC,&&
	\ol F^p_IV_\bbC:= \sum_{\substack{\bsp\in\bbZ^g,\,\lvert\bsp\rvert=p}} \ol\bsF_I^\bsp V_\bbC
\end{align}
for any $p \in \bbZ$.
Note that there are natural inclusions 
$\bsF_I^\bsp V_\bbC\hookrightarrow F_I^{\lvert\bsp\rvert}V_\bbC$ and $\ol\bsF_I^\bsp V_\bbC\hookrightarrow \ol F_I^{\lvert\bsp\rvert}V_\bbC$.
We will often omit the subscript of the notation when $I = \emptyset$.  For example, $\bsF^\bsp V_\bbC:= \bsF^\bsp_\emptyset V_\bbC$.


We first define the notion of a pure weak $g$-plectic $\bbC$-Hodge structure.

\begin{definition}[pure weak plectic $\bbC$-Hodge structure]\label{def: pure weak plectic}
	Let $n$ be an integer. A \textit{pure weak $g$-plectic $\bbC$-Hodge structure of weight $n$} is
	an object $V=(V_\bbC,\{F_\mu^\bullet\},\{\ol F_\mu^\bullet\})$ in $\Fil^0_g(\bbC)$ satisfying
	\begin{equation}\label{eq: monday}
		\bsF_I^\bsp V_\bbC=\bigoplus_{\substack{\bsr,\bss\in\bbZ^g\\ \bsr\geq\bsp,\ \lvert\bsr+\bss\rvert=n}}(\bsF_I^\bsr\cap\ol\bsF_I^\bss)V_\bbC
	\end{equation}
	for any $\bsp\in\bbZ^g$ and $I\subset\{1,\ldots,g\}$.
\end{definition}

Note that since $F_\mu^\bullet$ and $\ol F_\mu^\bullet$ are finite filtrations, we have $\bsF^\bsp_IV_\bbC=V_\bbC$ for any $\bsp$ whose components are sufficiently small.
Hence \eqref{eq: monday} implies that we have
\begin{equation}\label{eq: sunday}
	V_\bbC=\bigoplus_{\substack{\bsp,\bsq\in\bbZ^g\\ \lvert\bsp+\bsq\rvert=n}}(\bsF_I^\bsp\cap\ol\bsF_I^\bsq)V_\bbC.
\end{equation}

\begin{remark}\label{rem: weak}
	For any subset $I\subset\{1,\ldots,g\}$ we have $\ol\bsF_I^\bullet=\bsF_{I^c}^\bullet$,
	where $I^c:=\{I,\ldots,g\}\setminus I$ is the complement of $I$ in $\{1,\ldots,g\}$.
	In particular, the equation \eqref{eq: monday} for $I^c$ implies that
	\begin{equation}\label{eq: monday2}
		\ol\bsF_I^\bsq V_\bbC=\bigoplus_{\substack{\bsr,\bss\in\bbZ^g\\ \bss\geq\bsq,\ \lvert\bsr+\bss\rvert=n}}(\bsF_I^\bsr\cap\ol\bsF_I^\bss)V_\bbC.
	\end{equation}
\end{remark}

\begin{remark}\label{rem: nuts}
	Let $V$ be a pure weak $g$-plectic $\bbC$-Hodge structure of weight $n$, and consider
	$\bsp,\bsq \in \bbZ^g$ such that $\lvert\bsp+\bsq\rvert >n$.
	If we let $r:= \lvert\bsp+\bsq\rvert-n > 0$ and $\bsr_1 := (r,0,\ldots, 0) \in \bbZ^g$,
	then we have $\lvert\bsp+\bsq-\bsr_1\rvert = n$.  
	Since $\bsp-\bsr_1 < \bsp$ and $\bsq-\bsr_1 < \bsq$, we have
	\begin{equation*}
		 (\bsF_I^\bsp \cap\ol\bsF_I^\bsq) V_\bbC \subset
		   (\bsF_I^{\bsp-\bsr_1}\cap\ol\bsF_I^\bsq) V_\bbC \cap ( \bsF_I^{\bsp}\cap\ol\bsF_I^{\bsq-\bsr_1})V_\bbC
	\end{equation*}
	for any subset $I\subset\{1,\ldots,g\}$.
	By \eqref{eq: sunday}, the right hand side is $\{0\}$, hence we have the equality
	\begin{equation}\label{eq: fudge}
		(\bsF_I^\bsp \cap\ol\bsF_I^\bsq) V_\bbC = \{0\}.
	\end{equation}
\end{remark}

	
\begin{remark}\label{rem: bolts}
	Let $V$ be a pure weak $g$-plectic $\bbC$-Hodge structure of weight $n$ and $I\subset\{1,\ldots,g\}$ a subset. 
	Then the total Hodge filtrations $F_I^\bullet$ and $\ol F_I^\bullet$ on $V_\bbC$ with respect to $I$ are given by
	\begin{align*}
		F^p_IV_\bbC &= \sum_{\substack{\lvert\bsp\rvert=p}} \bsF_I^\bsp V_\bbC = \bigoplus_{\substack{\lvert\bsr\rvert\geq p\\\lvert\bsr+\bss\rvert=n}} (\bsF_I^{\bsr}\cap\ol\bsF_I^{\bss})V_\bbC, \\
		\ol F_I^{n+1-p}V_\bbC &= \sum_{\substack{\lvert\bsq\rvert=n+1-p}} \ol\bsF_I^\bsq V_\bbC 
		= \bigoplus_{\substack{\lvert\bss\rvert\geq n+1-p\\\lvert\bsr+\bss\rvert=n}} (\bsF_I^{\bsr}\cap\ol\bsF_I^{\bss})V_\bbC
		=  \bigoplus_{\substack{\lvert\bsr\rvert<p\\\lvert\bsr+\bss\rvert=n}} (\bsF_I^{\bsr}\cap\ol\bsF_I^{\bss})V_\bbC
	\end{align*}
	for any $p \in \bbZ$.
	Hence by \eqref{eq: sunday}, we have
	$
		V_\bbC = F^p_IV_\bbC \oplus \ol F^{n+1-p}_I V_\bbC.
	$
	By \eqref{eq: n-oppose}, we see that $(V_\bbC,F_I^\bullet,\ol F_I^\bullet)$ is a pure $\bbC$-Hodge structure of weight $n$ in the usual sense.
\end{remark}

We next define the notion of mixed weak plectic $\bbC$-Hodge structures.
One subtlty is that for an object $V=(V_\bbC,W_\bullet,\{F_\mu^\bullet\},\{\ol F_\mu^\bullet\})$ in $\Fil^1_g(\bbC)$,
there are two natural ``plectic" filtraions on $\Gr^W_n V_\bbC$, which in general do not coincide.  
More precisely, the natural inclusion
\begin{equation}\label{eq: incluzion}
	(W_n \cap \bsF^{\bsp}_I) V_\bbC 
	 / (W_{n-1} \cap \bsF^{\bsp}_I) V_\bbC
	 \subset\bigcap_{\mu\not\in I}F_\mu^{p_\mu}\Gr^W_nV_\bbC \cap \bigcap_{\nu\in I}\ol F_\nu^{p_\nu}\Gr^W_nV_\bbC
\end{equation}
is not in general an equality (see Example \ref{ex: counter} below).  In what follows, we adopt the left hand side and let
\begin{equation}\label{eq: grplectic}\begin{split}
	\bsF^p_I \Gr^W_n V_\bbC &:= (W_n \cap \bsF^{\bsp}_I) V_\bbC 
	 / (W_{n-1} \cap \bsF^{\bsp}_I) V_\bbC, \\
	 \ol\bsF^q_I \Gr^W_n V_\bbC &:= (W_n \cap \ol\bsF^{\bsq}_I) V_\bbC 
	 / (W_{n-1} \cap \ol\bsF^{\bsq}_I) V_\bbC
\end{split}\end{equation}
for any $I \subset \{1,\ldots,g \}$.  

\begin{definition}[mixed weak plectic $\bbC$-Hodge structure]\label{def: mixed weak plectic}
	A \textit{mixed weak $g$-plectic $\bbC$-Hodge structure} is an object $V=(V_\bbC,W_\bullet,\{F_\mu^\bullet\},\{\ol F_\mu^\bullet\})$ in $\Fil^1_g(\bbC)$ satisfying the following conditions for any subset $I\subset\{1,\ldots,g\}$:
	\begin{enumerate}
		\renewcommand{\theenumi}{\alph{enumi}$_I$}%
		\item For any $n\in \bbZ$ and $\bsp\in\bbZ^g$, we have	
		\begin{equation}\label{eq: tuesday}
			\bsF_I^\bsp  \Gr^W_n V_\bbC=\bigoplus_{\substack{\bsr,\bss\in\bbZ^g\\\bsr\geq\bsp,\,\lvert\bsr+\bss\rvert=n}} (\bsF_I^{\bsr}\cap\ol\bsF_I^{\bss})\Gr^W_n V_\bbC,
		\end{equation}
		where $ (\bsF_I^{\bsr}\cap\ol\bsF_I^{\bss})\Gr^W_n V_\bbC :=  \bsF_I^{\bsr} \Gr^W_n V_\bbC \cap \ol\bsF_I^{\bss}\Gr^W_n V_\bbC$.
		\item The object $V_I:=(V_\bbC,W_\bullet,F_I^\bullet,\ol F_I^\bullet)$ in $\Fil^1_1(\bbC)$ is a mixed $\bbC$-Hodge structure in the usual sense.
		\item For any $\bsp,\bsq\in\bbZ^g$ and $n:=\lvert\bsp+\bsq\rvert$, we have
		\begin{equation*}
			((\bsF_I^\bsp \cap W_n + \ol\bsF_I^\bsq \cap W_n)\cap W_{n-1})V_\bbC 
			\subset (\bsF_I^\bsp\cap W_{n-1}) V_\bbC + \sum_{\bsj \geq \boldsymbol{0}} 
			(\ol\bsF_I^{\bsq-\bsj} \cap W_{n-\lvert\bsj\rvert-1}) V_\bbC.
		\end{equation*}		
	\end{enumerate}
	We denote by $\wMHS_\bbC^g\subset\Fil^1_g(\bbC)$ the full subcategory of mixed weak $g$-plectic $\bbC$-Hodge structures.
	If $V= (V_\bbC, W_\bullet,\{F^\bullet_\mu\},\{\ol F_\mu^\bullet\})$ is a mixed weak $g$-plectic $\bbC$-Hodge structure, then we call $W_\bullet$ the \textit{weight filtration},
	$F^\bullet_\mu$ and $\ol F_\mu^\bullet$ the \textit{partial Hodge filtrations}, $\bsF^\bullet_I$ and $\ol\bsF^\bullet_I$ the \textit{plectic Hodge filtrations} with respect to $I$, and $F^\bullet_I$ and $\ol F_I^\bullet$ the \textit{total Hodge filtrations} with respect to $I$ of $V$.
\end{definition}

Due to Remark \ref{rem: bolts}, we will view a pure weak $g$-plectic $\bbC$-Hodge structure $V$ of weight $n$ as a mixed weak $g$-plectic $\bbC$-Hodge structure 
by taking the weight filtration to satisfy $W_{n-1}V_\bbC:=\{0\}$ and $W_nV_\bbC:=V_\bbC$.

\begin{remark}\label{rem: filt on Gr}
	Let $V=(V_\bbC,W_\bullet,\{F_\mu^\bullet\},\{\ol F_\mu^\bullet\})$ be an object in $\Fil^1_g(\bbC)$.
	Then, we have natural inclusions
	\begin{equation}\label{eq: filt on Gr}\begin{split}
		 (W_n\cap F_I^p)V_\bbC&\supset\sum_{\bsp\in\bbZ^g,\ \lvert\bsp\rvert=p}(W_n\cap\bsF_I^\bsp)V_\bbC,\\
		 (W_n\cap F_I^p)V_\bbC/ (W_{n-1}\cap F_I^p)V_\bbC &\supset\sum_{\bsp\in\bbZ^g,\ \lvert\bsp\rvert=p}\bsF_I^\bsp\Gr^W_nV_\bbC
	\end{split}\end{equation}
	for any $I\subset\{1,\ldots,g\}$, which are not equalities in general.  In what follows, we let
	\[
		F^p_I \Gr^W_n V_\bbC :=  (W_n\cap F_I^p)V_\bbC/ (W_{n-1}\cap F_I^p)V_\bbC
	\]
	and similarly for $\ol F^p_I\Gr^W_n V_\bbC$.
\end{remark}

\begin{example}\label{ex: counter}
	We note that the definition of a mixed weak $g$-plectic $\bbC$-Hodge structure is in general strictly stronger than the condition that for any $n\in\bbZ$, 
	the triple $\Gr^W_n V:= (\Gr^W_nV_\bbC,\{F^\bullet_\mu\},\{\ol F_\mu^\bullet\})$ is a pure weak $g$-plectic $\bbC$-Hodge structure of weight $n$.
	Consider the case when $g=2$ and let $V_\bbC:= \bbC e_0\oplus \bbC e_{-4}$ with the filtrations defined by
	\[
		W_n V_\bbC := \begin{cases}
				0  &  n \leq -5, \\
				\bbC e_{-4} &  n=-4,\ldots,-1,\\
				V_\bbC  &  n \geq 0,
		\end{cases}
	\]
	\begin{align*}	
		F^{p_1}_1 V_\bbC=\ol F^{p_1}_1 V_\bbC := \begin{cases}
				V_\bbC &  p_1 < 0, \\
				\bbC e_0 &  p_1=0,\\
				0  &  p_1>0,
		\end{cases}
		&&\text{and}&&
		F^{p_2}_2 V_\bbC =\ol F^{p_2}_2V_\bbC:= \begin{cases}
				V_\bbC &  p_2 < 0, \\
				\bbC (e_0+e_{-4}) &  p_2=0,\\
				0  &  p_2>0.
		\end{cases}
	\end{align*}
	Then we have $\Gr^W_0 V_\bbC = \bbC e_0$ and $F^0_1 \Gr^W_0 V_\bbC=F^0_2 \Gr^W_0 V_\bbC = \bbC e_0$,
	which shows that $(F^0_1\cap F^0_2) \Gr^W_0 V_\bbC = \bbC e_0$.
	However, since $\bsF^{(0,0)} V_\bbC := (F^0_1 \cap F^0_2) V_\bbC = \{0\}$, we have $\bsF^{(0,0)}\Gr^W_0 V_\bbC=\{0\}$, hence 
	\begin{equation*}
		\bsF^{(0,0)}\Gr^W_0 V_\bbC \subsetneq (F^0_1\cap F^0_2) \Gr^W_0 V_\bbC.
	\end{equation*}
	One can show that for $V = (V_\bbC, W_\bullet, \{ F^\bullet_1, F^\bullet_2\},\{\ol F^\bullet_1,\ol F_2^\bullet\})$ defined as above, $\Gr^W_n V$ is
	a pure weak $2$-plectic $\bbC$-Hodge structure of weight $n$ for any $n\in \bbZ$,
	but $V$ does not satisfy \eqref{eq: tuesday}.
\end{example}

In the next subsection, we will see that \eqref{eq: incluzion} and \eqref{eq: filt on Gr} are actually equalities for objects in $\wMHS_\bbC^g$.	

\begin{proposition}
	A mixed $\bbC$-Hodge structure in the usual sense is a mixed weak $1$-plectic $\bbC$-Hodge structure.
	In particular, the category $\wMHS^1_\bbC$ is equal to the category $\MHS_\bbC$ of mixed $\bbC$-Hodge structures.
\end{proposition} 	

\begin{proof}
	By definition, an object in $\wMHS^1_\bbC$ is a mixed $\bbC$-Hodge structure in the usual sense.
	Conversely, consider an object $V$ in $\MHS_\bbC$.   Then ($\textrm{a}_I$) holds by Lemma \ref{lem: candy} and ($\textrm{b}_I$) holds by definition.
	We prove ($\textrm{c}_I$).
	Let $p,q\in\bbZ$ and $n:=p+q$.
	We prove by induction on $k\geq0$ that
	\begin{equation}\label{eq: thursday}
		W_{n-1}V_\bbC \subset (F^p \cap W_{n-1}) V_\bbC + \sum_{j=0}^k (\ol F^{q-j} \cap W_{n-j-1})V_\bbC + W_{n-k-2}	
		V_\bbC.
	\end{equation}
	Suppose $w \in W_{n-1} V_\bbC$. Since 
	$\Gr^W_{n-1}V$ is a pure $\bbC$-Hodge structure of weight $n-1$, we have a splitting
	\begin{equation*}
		\Gr^W_{n-1} V_\bbC = F^p \Gr^W_{n-1} V_\bbC \oplus \ol F^{q} \Gr^W_{n-1} V_\bbC,
	\end{equation*}
	hence $w$ is of the form $w = u_0 + v_0 + w_1$ for some $u_0 \in (F^p \cap W_{n-1})V_\bbC$, $v_0 \in (\ol F^q \cap W_{n-1})V_\bbC$
	and $w_1 \in W_{n-2}V_\bbC$, which proves \eqref{eq: thursday} for $k=0$.  Suppose \eqref{eq: thursday} is true for an integer $k \geq 0$.
	Then any element $w \in W_{n-1} V_\bbC$ is of the form
	\begin{equation*}
		w = u_k + \sum_{j=0}^k v_j + w_{k+1}
	\end{equation*}
	for some $u_k \in  (F^p \cap W_{n-1})V_\bbC$, $v_j \in (\ol F^{q-j} \cap W_{n-j-1})V_\bbC$, and $w_{k+1} \in W_{n-k-2}V_\bbC$.
	Since
	$\Gr^W_{n-k-2}V$ is a pure $\bbC$-Hodge structure of weight $n-k-2$, we have a splitting
	\begin{equation*}
		\Gr^W_{n-k-2} V_\bbC = F^p \Gr^W_{n-k-2} V_\bbC \oplus \ol F^{q-k-1} \Gr^W_{n-k-2} V_\bbC,
	\end{equation*}
 	hence $w_{k+1}$ is of the form $w_{k+1} = u'_{k+1} + v_{k+1} + w_{k+2}$ for some 
	$u'_{k+1} \in (F^p \cap W_{n-k-2})V_\bbC$, $v_{k+1} \in (\ol F^{q-k-1} \cap W_{n-k-2})V_\bbC$
	and $w_{k+2} \in W_{n-k-3}V_\bbC$.  Then $u_{k+1}:=u_k+ u'_{k+1} \in (F^p\cap W_{n-1})V_\bbC$, and we see that
	\begin{equation*}
		w =  u_{k+1} + \sum_{j=0}^{k+1} v_j + w_{k+2} \in  (F^p \cap W_{n-1}) V_\bbC + \sum_{j=0}^{k+1} (\ol F^{q-j} \cap W_{n-j-1})V_\bbC + W_{n-k-3}
		V_\bbC.
	\end{equation*}
	By induction, \eqref{eq: thursday} is true for any $k \geq 0$.   Since $W_{n-k-2} V_\bbC = \{0\}$ for $k$ sufficiently large, 
	we have
	\begin{equation*}
		((F^p\cap W_n+\ol F^q\cap W_n)\cap W_{n-1})V_\bbC \subset W_{n-1} V_\bbC \subset 
		(F^p \cap W_{n-1}) V_\bbC + \sum_{j\geq 0} (\ol F^{q-j} \cap W_{n-j-1})V_\bbC,
	\end{equation*}
	which proves condition ($\textrm{c}_I$) for $I=\emptyset$. 
	Since the quadruple $(V_\bbC,W_\bullet,\ol F^\bullet,F^\bullet)$ is also a mixed $\bbC$-Hodge structure, 
	condition ($\textrm{c}_I$) for $I=\{1\}$ also holds.
\end{proof}

\subsection{The plectic Deligne splitting}
%

In this subsection, we will prove Proposition \ref{prop: psplitting},
which is a plectic version of the Deligne splitting for objects in $\wMHS^g_\bbC$.
We will first define the plectic version of the bigradings $A^{p,q}$ and $\ol A^{p,q}$.

\begin{definition}\label{def: pApq}
	Let $V=(V_\bbC, W_\bullet, \{ F^\bullet_\mu\},\{\ol F_\mu^\bullet\})$ be an object in $\Fil^1_g(\bbC)$.
	For any $I\subset\{1,\ldots,g\}$, $\bsp,\bsq\in\bbZ^g$, and $n:=\lvert\bsp+\bsq\rvert$, we put
	\begin{equation}\label{eq: pApq}
		\bsA^{\bsp,\bsq}_I(V):= (\bsF_I^\bsp \cap W_{n})V_\bbC
		\cap\biggl((\ol\bsF_I^\bsq \cap W_{n})V_\bbC+\sum_{\bsj\geq\boldsymbol 0} (\ol\bsF_I^{\bsq-\bsj}\cap W_{n-\lvert\bsj\rvert-1}) V_\bbC\biggr).
	\end{equation}
	We denote by
	\begin{equation}\label{eq: peach}
		\rho_I:\bsA_I^{\bsp,\bsq}(V) \rightarrow (\bsF_I^\bsp\cap\ol\bsF_I^\bsq) \Gr^W_nV_\bbC
	\end{equation}
	the $\bbC$-linear homomorphism induced by the natural surjection $W_nV_\bbC\rightarrow\Gr^W_nV_\bbC$.
\end{definition}

Note that when $g=1$, the subspaces \eqref{eq: pApq} coincide with \eqref{eq: Apq} in Proposition \ref{prop: splitting}.
	
\begin{proposition}\label{prop: psplitting}
	Let $V$ be an object in $\Fil^1_g(\bbC)$. Consider the conditions ($\textrm{a}_I$),($\textrm{b}_I$),($\textrm{c}_I$) in Definition \ref{def: mixed weak plectic}.
	\begin{enumerate}
	\item ($\textrm{b}_I$) implies that $\rho_I$ is injective.
	\item ($\textrm{c}_I$) is equivalent to that $\rho_I$ is surjective.
	\item ($\textrm{a}_I$), ($\textrm{b}_I$), and ($\textrm{c}_I$) together imply that we have
		\begin{align}\label{eq: W F by A}
		W_n V_\bbC&=\bigoplus_{\substack{\bsp,\bsq\in\bbZ^g\\\lvert\bsp+\bsq\rvert\leq n}}\bsA_I^{\bsp,\bsq}(V), &
		\bsF_I^\bsp V_\bbC&=\bigoplus_{\substack{\bsr,\bss\in\bbZ^g\\ \bsr \geq \bsp}}\bsA_I^{\bsr,\bss}(V),
		\end{align}
		for any $n\in\bbZ$ and $\bsp\in\bbZ^g$, and in particular
		\begin{equation}\label{eq: psplitting}
		V_\bbC =\bigoplus_{\bsp,\bsq\in\bbZ^g}\bsA_I^{\bsp,\bsq}(V).
		\end{equation}
	\item If $V$ is an object in $\wMHS_\bbC^g$, then for any $I\subset\{1,\ldots,g\}$, $\rho_I$ is an isomorphism and the equalities \eqref{eq: W F by A} and \eqref{eq: psplitting} hold.
	\end{enumerate}
\end{proposition}

\begin{proof}
	(1) Let $p:=\lvert\bsp\rvert$, $q:=\lvert\bsq\rvert$, and $A^{p,q}_I(V):=A^{p,q}(V_I)$ for $V_I:=(V_\bbC, W_\bullet, F^\bullet_I, \ol F^\bullet_I)$.
	Since we have a commutative diagram
	\[\xymatrix{
	\bsA_I^{\bsp,\bsq}(V)\ar[r]\ar@{_(->}[d]&(\bsF^\bsp_I\cap\ol\bsF^\bsq_I)\Gr^W_nV_\bbC\ar@{_(->}[d]\\
	A_I^{p,q}(V)\ar[r]&(F^p_I\cap\ol F^q_I)\Gr^W_nV_\bbC,
	}\]
	the assertion follows from Lemma \ref{lem: Apq}.
	
	(2) Assume condition ($\textrm{c}_I$) and consider an element $\xi\in(\bsF_I^\bsp\cap\ol\bsF_I^\bsq) \Gr^W_n V_\bbC$.  
	Let $u\in(\bsF_I^\bsp \cap W_n)V_\bbC$ and $v\in(\ol\bsF_I^\bsq \cap W_n)V_\bbC$ be elements lifting $\xi$.
	Then since $u-v \equiv 0 \pmod{W_{n-1}}$, we have
	\begin{equation*}
		u-v \in ((\bsF_I^\bsp\cap W_n+\ol\bsF_I^\bsq\cap W_n)\cap W_{n-1}) V_\bbC.
	\end{equation*}
	By condition ($\textrm{c}_I$), there exist $u_0\in(\bsF_I^\bsp\cap W_{n-1})V_\bbC$ and $v_{\bsj}\in 
	(\ol\bsF_I^{\bsq-\bsj}\cap W_{n-\lvert\bsj\rvert-1})V_\bbC$ for $\bsj\geq\boldsymbol 0$ such that
	\begin{equation*}
		u-v=u_0+\sum_{\bsj \geq \boldsymbol 0}v_\bsj.
	\end{equation*}
	If we let $\wt\xi:=u-u_0= v+\sum_{\bsj\geq\boldsymbol 0}v_\bsj$, then we have $\wt\xi \in\bsA_I^{\bsp,\bsq}(V)$ and $\wt\xi\equiv\xi\pmod{W_{n-1}}$,
	hence this proves that $\rho_I$ is surjective as desired.\\
	Conversely assume $\rho_I$ is surjective.
	An element $w\in((\bsF_I^\bsp\cap W_n+\ol\bsF_I^\bsq \cap W_n)\cap W_{n-1})V_\bbC$ may be written in the form
	$w = u-v$, with $u \in (\bsF_I^\bsp\cap W_n)V_\bbC$, $v \in  (\ol\bsF_I^\bsq \cap W_n)V_\bbC$ and $w \in W_{n-1}V_\bbC$.
	If we let $\xi \equiv u \equiv v \pmod{W_{n-1}}$, then $\xi$ is an element in $(\bsF_I^\bsp \cap\ol\bsF_I^\bsq)\Gr^W_nV_\bbC$.
	Since $\rho_I$ is surjective, there exists $u_0\in\bsA_I^{\bsp,\bsq}(V)$ such that $u_0 \equiv\xi\pmod{W_{n-1}}$,  where
	by \eqref{eq: pApq}, 
	we have $u_0\in(\bsF_I^\bsp\cap W_n)V_\bbC$ and $u_0$ is of the form 
	\begin{equation*}
		u_0 = v_0+\sum_{\bsj\geq\boldsymbol 0} w_\bsj
	\end{equation*}
	for $v_0 \in(\ol\bsF_I^\bsq\cap W_n)V_\bbC$ and $w_\bsj\in(\ol\bsF_I^{\bsq-\bsj}\cap W_{n-\lvert\bsj\rvert-1})V_\bbC$.
	Since $u_0\equiv u\pmod{W_{n-1}}$ and $v_0\equiv v\pmod{W_{n-1}}$, we have 
	$u_0 = u-w_0$ and $v_0 = v+w_1$ for some $w_0,w_1 \in W_{n-1} V_\bbC$.
	Note that $w_0 = u-u_0 \in (\bsF_I^\bsp\cap W_{n-1})V_\bbC$ and $w_1 = v_0-v \in  (\ol\bsF_I^\bsq\cap W_{n-1}) V_\bbC$.
	Then we have
	\begin{equation*}
		w = u-v = w_0+w_1+\sum_{\bsj\geq\boldsymbol 0} w_\bsj,
	\end{equation*}
	hence $w \in  (\bsF_I^\bsp\cap W_{n-1})V_\bbC+(\ol\bsF_I^\bsq \cap W_{n-1}) 
	V_\bbC+\sum_{\bsj\geq\boldsymbol 0}(\ol\bsF_I^{\bsq-\bsj}\cap W_{n-\lvert\bsj\rvert-1}) V_\bbC$ as desired.
	
	(3) We prove by induction on $n$ that
	\begin{equation}\label{eq: statement W}
		(\bsF_I^\bsp \cap W_n) V_\bbC= \bigoplus_{\substack{\bsr,\bss\in\bbZ^g,\,\bsr\geq\bsp\\\lvert\bsr+\bss\rvert \leq n}} \bsA_I^{\bsr,\bss}(V).
	\end{equation}
	If $n$ is sufficiently small so that $W_n V_\bbC=\{0\}$, then the statement is trivially true.  
	Next suppose that \eqref{eq: statement W} is true for $n-1$.
	We have a commutative diagram
	\begin{equation*}\xymatrix{
		0 \ar[r] &  \bigoplus_{\substack{\bsr,\bss\in\bbZ^g,\,\bsr\geq\bsp\\\lvert\bsr+\bss\rvert \leq n-1}} \bsA_I^{\bsr,\bss}(V) \ar[r] \ar[d]^\cong&
		 \bigoplus_{\substack{\bsr,\bss\in\bbZ^g,\,\bsr\geq\bsp\\\lvert\bsr+\bss\rvert \leq n}} \bsA_I^{\bsr,\bss}(V) \ar[r] \ar[d]&
		  \bigoplus_{\substack{\bsr,\bss\in\bbZ^g,\,\bsr\geq\bsp\\\lvert\bsr+\bss\rvert =n}} \bsA_I^{\bsr,\bss}(V) \ar[r] \ar[d]^\cong&0\\
		 0 \ar[r] & (\bsF_I^\bsp\cap W_{n-1}) V_\bbC \ar[r] & (\bsF_I^\bsp\cap W_n)V_\bbC \ar[r] &\bsF_I^\bsp\Gr^W_n V_\bbC \ar[r]&0,
	}\end{equation*}
	where the left and middle vertical arrows are the sum of the natural inclusions.
	The left vertical arrow is an isomorphism by the induction hypothesis, and the right vertical arrow is an isomorphism by (1),(2), and condition ($\textrm{a}_I$).
	This shows that the central vertical arrow is also an isomorphism, hence by induction, \eqref{eq: statement W} is true for any $n\in\bbZ$.
	This proves our assertion, noting that $W_n V_\bbC = V_\bbC$ for $n$ sufficiently large and $\bsF_I^\bsp V_\bbC=V_\bbC$ for $\bsp$ sufficiently small.

	(4) Follows from (1), (2), and (3).
	\end{proof}
	
Let $V$ be an object in $\wMHS^g_\bbC$.
Then by Proposition \ref{prop: psplitting}, $\rho_I$ is an isomorphism and the equalities \eqref{eq: W F by A} and \eqref{eq: psplitting} hold for any $I\subset\{1,\ldots,g\}$.
We call the $2g$-grading $\{\bsA_I^{\bsp,\bsq}(V)\}$ of $V_\bbC$ the \textit{plectic Deligne splitting} of $V$ with respect to $I$.

For an object $V=(V_\bbC,W_\bullet,\{F_\mu^\bullet\},\{\ol F_\mu^\bullet\})$ in $\Fil^1_g(\bbC)$ and $n\in\bbZ$, we define an object $W_nV$ (resp. $\Gr^W_nV$) in $\Fil^1_g(\bbC)$ to be the quadruple consisting of the $\bbC$-vector space $W_nV_\bbC$ (resp. $\Gr^W_nV_\bbC$) and the filtrations induced from those of $V$. We often regard $\Gr^W_nV$ as an object in $\Fil^0_g(\bbC)$ by forgetting the weight filtration.
Then we obtain additive functors
\begin{align}
	W_n:\Fil^1_g(\bbC)\rightarrow\Fil^1_g(\bbC)&&\text{and}&&\Gr^W_n:\Fil^1_g(\bbC)\rightarrow\Fil^0_g(\bbC).
\end{align}

\begin{corollary}\label{cor: coincides}
	Let $V$ be an object in $\wMHS_\bbC^g$.
	Then for any $n\in\bbZ$, the plectic (resp. total) Hodge filtrations of $W_nV$ and $\Gr^W_nV$ 
	coincide with the filtrations induced from the plectic (resp. total) Hodge filtrations of $V$.
	In particular, $W_nV$ is also an object in $\wMHS_\bbC^g$, and $\Gr^W_nV$ is a pure weak $g$-plectic $\bbC$-Hodge structure of weight $n$.
\end{corollary}

\begin{proof}
	By the direct decompositions \eqref{eq: W F by A}, the natural inclusions \eqref{eq: incluzion} and \eqref{eq: filt on Gr} are 
	actually equalities. Then the conditions ($\textrm{a}_I$), ($\textrm{b}_I$), ($\textrm{c}_I$) for $W_nV$ and $\Gr^W_nV$ follow from those for $V$.
\end{proof}

\begin{corollary}\label{cor: strictness of pA}
	Let $\alpha:U\rightarrow V$ be a morphism in $\wMHS^g_\bbC$. For any subsets $\boldsymbol{\cS}\subset\bbZ^g\times\bbZ^g$ and $I\subset\{1,\ldots,g\}$,
	we have
	\begin{equation}\label{eq: p1}
		\alpha\left(\bigoplus_{(\bsp,\bsq)\in \boldsymbol{\cS}}\bsA_I^{\bsp,\bsq}(U)\right)=\alpha(U_\bbC)
		\cap\left(\bigoplus_{(\bsp,\bsq)\in \boldsymbol{\cS}}\bsA_I^{\bsp,\bsq}(V)\right).
	\end{equation}
	In particular, if $\boldsymbol{\cS}'$ is a subset of $\bbZ^g \times \bbZ$, then we have
	\begin{equation}\label{eq: p2}
		\alpha\left(\sum_{(\bsp,n)\in \boldsymbol{\cS}'}(\bsF_I^\bsp \cap W_n)U_\bbC\right)
		=\alpha(U_\bbC)\cap\left(\sum_{(\bsp,n)\in \boldsymbol{\cS}'}(\bsF_I^\bsp \cap W_n)V_\bbC\right).
	\end{equation}
	In particular, $\alpha$ is strict with respect to the filtration $(\bsF_I^\bullet\cap W_\bullet)$.
\end{corollary}

\begin{proof}
	Since $\alpha(\bsA_I^{\bsp,\bsq}(U))\subset \bsA_I^{\bsp,\bsq}(V)$, the equality \eqref{eq: p1} follows
	from the fact that $\bsA_I^{\bsp,\bsq}$ gives $2g$-gradings \eqref{eq: pApq} of $U_\bbC$ and $V_\bbC$.
	Since we have by Proposition \ref{prop: psplitting}
	\[
		(\bsF_I^\bsp\cap W_n)U_\bbC= \bigoplus_{\substack{\bsr,\bss\in\bbZ^g\\\bsr \geq \bsp,\,\lvert\bsr+\bss\rvert \leq n}}\bsA_I^{\bsr,\bss}(U)
	\]
	for any $\bsp\in\bbZ^g$ and $n\in\bbZ$, the equality \eqref{eq: p2} follows from equality \eqref{eq: p1} for
	\[
		\cS:=\bigcup_{(\bsp,n)\in\cS'}\{(\bsr,\bss)\in\bbZ^g\times\bbZ^g\mid\bsr\geq\bsp,\lvert\bsr+\bss\rvert\leq n\}.
	\qedhere\]
\end{proof}

%
\subsection{Plectic Hodge decomposition of orthogonal families}
%
Let $g$ be a positive integer.
We define a functor $T^g_\bbC:\Fil^g_{g}(\bbC)\rightarrow\Fil^1_g(\bbC)$ by taking the total filtration of $\{W^\mu_\bullet\}$.
Namely, for an object $V=(V_\bbC,\{W^\mu_\bullet\},\{F_\mu^\bullet\},\{\ol F_\mu^\bullet\})$, we have
$T^g_\bbC(V)=(V_\bbC,W_\bullet,\{F_\mu^\bullet\},\{\ol F_\mu^\bullet\})$ with
\begin{equation}\label{eq: total weight}
	W_nV_\bbC:=\sum_{n_1+\cdots+n_g=n}(W^1_{n_1}\cap\cdots\cap W^g_{n_g})V_\bbC.
\end{equation}

The purpose of this subsection is to prove the following proposition.
\begin{proposition}\label{prop: Hodge}
	Let $V$ be an object in $\OF_\bbC^g$.
	Then the quadruple $T^g_\bbC(V)$ is an object in $\wMHS^g_\bbC$.
\end{proposition}

Let $V$ be an object in $\OF_\bbC^g$ and $I\subset\{1,\ldots,g\}$ a subset.
For each $\mu=1,\ldots,g$, we define
\begin{equation}
A_{I,\mu}^{p,q}(V):=\begin{cases}
(F_\mu^p\cap W^\mu_{p+q})V_\bbC\cap\biggl((\ol F_\mu^q\cap W_{p+q}^\mu)V_\bbC+\sum_{j\geq 0}(\ol F_\mu^{q-j}\cap W^\mu_{p+q-j-1})V_\bbC\biggr),& \mu\not\in I,\\
\biggl((F_\mu^p\cap W^\mu_{p+q})V_\bbC+\sum_{j\geq 0}(F_\mu^{q-j}\cap W^\mu_{p+q-j-1})V_\bbC\biggr)\cap (\ol F_\mu^q\cap W^\mu_{p+q})V_\bbC,&\mu\in I,
\end{cases}\end{equation}
that is the Deligne splitting of the mixed $\bbC$-Hodge structure $(V_\bbC,W^\mu_\bullet,F_\mu^\bullet,\ol F^\bullet_\mu)$.
By Proposition \ref{prop: strictness of A} and Corollary \ref{cor: sum and intersection}, the $\bbC$-vector space $A_{I,\mu}^{p,q}(V)$ with $\nu$-th filtrations for $\nu\neq\mu$ is an object in $\OF_\bbC^{g-1}$, and we have $A_{I,\nu}^{r,s}\circ A_{I,\mu}^{p,q}(V)=A_{I,\nu}^{r,s}(V)\cap A_{I,\mu}^{p,q}(V)$.
Hence we have the direct decompositions
\begin{align}
	\label{eq: W by partial A}W_n^\mu V_\bbC&=\bigoplus_{\substack{\bsp,\bsq\in\bbZ^g\\ p_\mu+q_\mu\leq n}}A_{I,1}^{p_1,q_1}(V)\cap\cdots\cap A_{I,g}^{p_g,q_g}(V),\\
	\label{eq: F by partial A}F_\mu^p V_\bbC&=\bigoplus_{\substack{\bsr,\bss\in\bbZ^g\\ r_\mu\geq p}}A_{I,1}^{r_1,s_1}(V)\cap\cdots\cap A_{I,g}^{r_g,s_g}(V), & \mu\not\in I,\\
	\label{eq: olF by partial A}\ol F_\mu^p V_\bbC&=\bigoplus_{\substack{\bsr,\bss\in\bbZ^g\\ r_\mu\geq p}}A_{I,1}^{r_1,s_1}(V)\cap\cdots\cap A_{I,g}^{r_g,s_g}(V), & \mu\in I.
\end{align}

Then Proposition \ref{prop: Hodge} follows from the following propositions.

\begin{proposition}\label{prop: intermediate}
	Let $(V_\bbC, \{ W^\mu_\bullet\}, \{F^\bullet_\mu\},\{\ol F_\mu^\bullet\})$ be an object in $\OF_\bbC^g$. Then $ T^g_\bbC(V)$ satisfies the condition ($\textrm{a}_I$) of Definition \ref{def: mixed weak plectic}.
	That is, for any $n \in \bbZ$, $\bsp,\bsq \in \bbZ^g$, and $I\subset\{1,\ldots,g\}$, we have	
	\begin{align}\label{eq: Hodge}
		\bsF_I^\bsp \Gr^W_n V_\bbC&=\bigoplus_{\substack{\bsr,\bss\in\bbZ^g
		\\\bsr\geq\bsp,\,\lvert\bsr+\bss\rvert=n}}(\bsF^\bsr_I \cap\ol\bsF^\bss_I) \Gr^W_n V_\bbC.
	\end{align}	
\end{proposition}

\begin{proof}
	For simplicity, we assume $I=\emptyset$.
	We prove the statement by induction on $g$.
	The statement for $g=1$ is Lemma \ref{lem: candy}.
	Suppose the statement is true for objects in $\OF_\bbC^{g-1}$.
	By Lemma \ref{lem: candy}, we have
	\begin{equation}\label{eq: case g=1}
		\bigoplus_{\substack{r_g,s_g\in\bbZ\\ r_g\geq p_g,\ r_g+s_g=m}}(F_g^{r_g}\cap\ol F_g^{s_g})\Gr^{W^g}_mV_\bbC\xrightarrow[]{\cong} F_g^{p_g}\Gr^{W^g}_mV_\bbC.
	\end{equation}
	By Corollary, \ref{cor: sum and intersection} $(F_g^{r_g}\cap\ol F_g^{s_g})\Gr^{W^g}_mV_\bbC$ is an object in $\OF_\bbC^{g-1}$ with respect to $W^\nu_\bullet$, $F_\nu^\bullet$, and $\ol F_\nu^\bullet$ for $\nu=1,\ldots,g-1$, and \eqref{eq: case g=1} is an isomorphism in $\OF_\bbC^{g-1}$.
	If we denote by $W'_\bullet$ the filtration given by \eqref{eq: total weight} for $\mu=1,\ldots,g-1$, the induction implies
	\begin{equation}\label{eq: ind hyp}
	\bigoplus_{\substack{\bsr',\bss'\in\bbZ^{g-1}\\ \bsr'\geq\bsp',\ \lvert\bsr'+\bss'\rvert=n-m}}(\bsF^{\bsr'}\cap\ol\bsF^{\bss'})\Gr^{W'}_{n-m}(F_g^{r_g}\cap\ol F_g^{s_g})\Gr^{W^g}_mV_\bbC
	\xrightarrow[]{\cong} \bsF^{\bsp'}\Gr^{W'}_{n-m}(F_g^{r_g}\cap\ol F_g^{s_g})\Gr^{W^g}_mV_\bbC.
	\end{equation}
	Note that by Corollary \ref{cor: sum and intersection}, $W'_{n-m}V_\bbC$ with the filtration induced from $W^g_\bullet$, $F_g^\bullet$, and $\ol F_g^\bullet$ is a mixed $\bbC$-Hodge structure.
	Then 
	\[0\rightarrow W'_{n-m-1}\Gr^{W^g}_mV_\bbC\rightarrow W'_{n-m}\Gr^{W^g}_mV_\bbC\rightarrow\Gr^{W'}_{n-m}\Gr^{W^g}_mV_\bbC\rightarrow 0\]
	is an exact sequence of pure $\bbC$-Hodge structures, hence by \eqref{eq: pure}, we have
	\begin{equation}\label{eq: pure2}
	\begin{split}
	(F_g^{r_g}\cap\ol F_g^{s_g})\Gr^{W'}_{n-m}\Gr^{W^g}_mV_\bbC&\cong(F_g^{r_g}\cap\ol F_g^{s_g}\cap W'_{n-m})\Gr^{W^g}_mV_\bbC/(F_g^{r_g}\cap\ol F_g^{s_g}\cap W'_{n-m-1})\Gr^{W^g}_mV_\bbC\\
	&\cong\Gr^{W'}_{n-m}(F_g^{r_g}\cap\ol F_g^{s_g})\Gr^{W^g}_mV_\bbC,
	\end{split}\end{equation}
	which is an isomorphism in $\OF_\bbC^{g-1}$ with respect to $W^\mu_\bullet$, $F_\mu^\bullet$, and $\ol F_\mu^\bullet$ for $\mu=1,\ldots,g-1$.
	Then by \eqref{eq: case g=1}, \eqref{eq: ind hyp}, and \eqref{eq: pure2}, we obtain
	\begin{equation}\label{eq: decomp OF}
	\bigoplus_{m\in\bbZ}\bsF^{\bsp'}\Gr^{W'}_{n-m}F_g^{p_g}\Gr^{W^g}_mV_\bbC\cong\bigoplus_{m\in\bbZ}\bigoplus_{\substack{\bsr',\bss'\in\bbZ^{g-1}\\ \bsr'\geq\bsp',\ \lvert\bsr'+\bss'\rvert=n-m}}\bigoplus_{\substack{r_g,s_g\in\bbZ\\ r_g\geq p_g,\ r_g+s_g=m}}(\bsF^\bsr\cap\ol\bsF^\bss)\Gr^{W'}_{n-m}\Gr^{W^g}_mV_\bbC.
	\end{equation}
	Since $F_\mu^pV_\bbC$ and $W^\mu_l V_\bbC$ can be written as direct sums of $A_{1}^{p_1,q_1}(V)\cap\cdots\cap A_{g}^{p_g,q_g}(V)$ 
	as in \eqref{eq: W by partial A} and \eqref{eq: F by partial A}, the left hand side of \eqref{eq: decomp OF} is isomorphic to $\bsF^\bsp\Gr^W_nV_\bbC$.
	On the other hand, by \eqref{eq: W by partial A}, \eqref{eq: F by partial A}, and \eqref{eq: olF by partial A}, we have an isomorphism $\bigoplus_{m\in\bbZ}\Gr^{W'}_{n-m}\Gr^{W^g}_mV_\bbC\xrightarrow[]{\cong}\Gr^W_nV_\bbC$ in $\OF_\bbC^g$. Hence the right hand side of \eqref{eq: decomp OF} is isomorphic to $\bigoplus_{\substack{\bsr,\bss\in\bbZ^g\\ \bsr\geq\bsp,\ \lvert\bsr+\bss\rvert=n}}(\bsF^\bsr\cap\ol\bsF^\bss)\Gr^W_nV_\bbC$.
\end{proof}

\begin{proposition}\label{prop: OF total}
	Let $(V_\bbC, \{ W^\mu_\bullet\}, \{F^\bullet_\mu\},\{\ol F_\mu^\bullet\})$ be an object in $\OF_\bbC^g$. Then $ T^g_\bbC(V)$ 
	satisfies the condition ($\textrm{b}_I$) of Definition \ref{def: mixed weak plectic}.
	In other words, $(V_\bbC,W_\bullet,F_I^\bullet,\ol F_I^\bullet)$ is a mixed $\bbC$-Hodge structure for any subset $I\subset\{1,\ldots,g\}$.	
\end{proposition}

\begin{proof}
	By Proposition \ref{prop: intermediate}, we have
	\begin{align*}
	F_I^p\Gr^W_nV_\bbC&\cong\bigoplus_{\substack{\bsr,\bss\in\bbZ^g\\ \lvert\bsr\rvert\geq p,\ \lvert\bsr+\bss\rvert=n}}(\bsF_I^\bsr\cap\ol\bsF_I^\bss)\Gr^W_nV_\bbC,\\
	\ol F_I^q\Gr^W_nV_\bbC&\cong\bigoplus_{\substack{\bsr,\bss\in\bbZ^g\\ \lvert\bsr\rvert\geq q,\ \lvert\bsr+\bss\rvert=n}}(\bsF_{I^c}^\bsr\cap\ol\bsF_{I^c}^\bss)\Gr^W_nV_\bbC=
	\bigoplus_{\substack{\bsr,\bss\in\bbZ^g\\ \lvert\bss\rvert\geq q,\ \lvert\bsr+\bss\rvert=n}}(\bsF_I^\bsr\cap\ol\bsF_I^\bss)\Gr^W_nV_\bbC
	\end{align*}
	for any $p,q,n\in\bbZ$.
	Hence we obtain $\Gr^W_nV_\bbC=F_I^p\Gr^W_nV_\bbC\oplus\ol F_I^{n-p+1}\Gr^W_nV_\bbC$ as desired.
\end{proof}

\begin{proposition}\label{prop: pApq}
	Let $V$ be an object in $\OF_\bbC^g$ and $I\subset\{1,\ldots,g\}$ a subset.
	Then we have
	\begin{equation}\label{eq: A and bsA}
	    A_{I,1}^{p_1,q_1}(V)\cap\cdots \cap A_{I,g}^{p_g,q_g}(V)=\bsA_I^{\bsp,\bsq}(T^g_\bbC(V))
	\end{equation}
	for any $\bsp,\bsq\in\bbZ^g$.  Moreover, the homomorphism
	\[
		\rho_I:\bsA_I^{\bsp,\bsq}(T^g_\bbC(V)) \rightarrow (\bsF_I^\bsp\cap\ol\bsF_I^\bsq)\Gr^W_n V_\bbC
	\]
	is an isomorphism, where $n:=|\bsp+\bsq|$.
\end{proposition}

\begin{proof}
	For simplicity we assume $I=\emptyset$. We prove by induction on $g$. The statement for $g=1$ follows by definition. 
	Suppose the statement is true for $g-1$, and let $\{\bsA^{\bsp',\bsq'}(T^{g-1}_\bbC(V))\}$ for indices $\bsp':=(p_1,\ldots,p_{g-1})$ and $\bsq':=(q_1,\ldots,q_{g-1})$
	be the plectic Deligne splitting for the quadruple $(V_\bbC, W'_\bullet, \{F^\bullet_\mu\},\{\ol F_\mu^\bullet\})$, 
	where $W'_\bullet$ is the filtration defined from the filtrations
	 $W^\mu_\bullet$ for $\mu=1,\ldots, g-1$, and the family $\{F^\bullet_\mu\}$ and $\{\ol F_\mu^\bullet\}$ are for the indices $\mu=1,\ldots,g-1$.
	 Then for $n' := \lvert\bsp'+\bsq'\rvert$ and $n_g:=p_g+q_g$, we have
	\begin{equation*}
		A^{p_1,q_1}_{1}(V) \cap\cdots \cap A^{p_g,q_g}_{g}(V) = \bsA^{\bsp',\bsq'}(T^{g-1}_\bbC(V))\cap A_g^{p_g,q_g}(V)
	\end{equation*}
	by the induction hypothesis.  Note that by definition, $\bsA^{\bsp',\bsq'}(T^{g-1}_\bbC(V))\cap A_g^{p_g,q_g}(V)$ is equal to
	\begin{multline*}
		(\bsF^{\bsp'}\cap W'_{n'})V_\bbC  
		\cap\biggl((\ol\bsF^{\bsq'}\cap W'_{n'})V_\bbC+\sum_{\bsj'\geq\boldsymbol{0}}(\ol\bsF^{\bsq'-\bsj'}\cap W'_{n'-\lvert\bsj'\rvert-1})V_\bbC\biggr)\\
		\cap (F^{p_g}_g\cap W^g_{n_g})V_\bbC 
		\cap\biggl((\ol F^{q_g}_g\cap W^g_{n_g})V_\bbC+\sum_{j_g\geq 0}(\ol F_g^{q_g-j_g}\cap W^g_{n_g-j_g-1})V_\bbC\biggr).
	\end{multline*}
	Hence we have
	\begin{equation}\label{eq: R2D2}
		\bsA^{\bsp',\bsq'}(T^{g-1}_\bbC(V))\cap A_g^{p_g,q_g}(V)\subset (\bsF^\bsp\cap W_n)V_\bbC.
	\end{equation}
	 Let $U$ be the mixed $\bbC$-Hodge structure on $U_\bbC=(\ol\bsF^{\bsq'}\cap W'_{n'})V_\bbC+\sum_{\bsj'\geq\boldsymbol{0}}(\ol\bsF^{\bsq'-\bsj'}\cap W'_{n'-\lvert\bsj'\rvert-1})V_\bbC$ with filtrations induced from $W^g_\bullet$, $F_g^\bullet$, and $\ol F_g^\bullet$.
	 Applying Proposition \ref{prop: strictness of A} to the natural inclusion $U\hookrightarrow V$, we have
	\begin{multline*}
		 \biggl((\ol\bsF^{\bsq'}\cap W'_{n'})V_\bbC+\sum_{\bsj'\geq\boldsymbol{0}}(\ol\bsF^{\bsq'-\bsj'}
		\cap W'_{n'-\lvert\bsj'\rvert-1})V_\bbC\biggr)
		\cap\biggl((\ol F^{q_g}_g\cap W^g_{n_g})V_\bbC
		+\sum_{j_g\geq 0}(\ol F^{q_g-j_g}_g\cap W^g_{n_g-j_g-1})V_\bbC\biggr)\\
		= \biggl((\ol\bsF^{\bsq'}\cap W'_{n'})V_\bbC
		+\sum_{\bsj'\geq\boldsymbol{0}}(\ol\bsF^{\bsq'-\bsj'}\cap W'_{n'-\lvert\bsj'\rvert-1})V_\bbC\biggr)\cap(\ol F_g^{q_g}\cap W^g_{n_g})V_\bbC\\
		+\sum_{j_g\geq 0}\biggl((\ol\bsF^{\bsq'}\cap W'_{n'})V_\bbC
		+\sum_{\bsj'\geq\boldsymbol{0}}(\ol\bsF^{\bsq'-\bsj'}\cap W'_{n'-\lvert\bsj'\rvert-1})V_\bbC\biggr) \cap(\ol F_g^{q_g-j_g}\cap W^g_{n_g-j_g-1})V_\bbC.
	\end{multline*}
	By \eqref{filtration of sum:1}, \eqref{filtration of sum:2}, we have
	\begin{multline*}
		\biggl((\ol\bsF^{\bsq'}\cap W'_{n'})V_\bbC
		+\sum_{\bsj'\geq\boldsymbol{0}}(\ol\bsF^{\bsq'-\bsj'}\cap W'_{n'-\lvert\bsj'\rvert-1})V_\bbC\biggr)\cap(\ol F_g^{q_g}\cap W^g_{n_g})V_\bbC\\
		=(\ol\bsF^{\bsq}\cap W'_{n'}\cap W^g_{n_g})V_\bbC
		+\sum_{\bsj'\geq\boldsymbol{0}}(\ol\bsF^{\bsq-(\bsj',0)}\cap W'_{n'-\lvert\bsj'\rvert-1}\cap W^g_{n_g})V_\bbC
	\end{multline*}
	and
    	\begin{multline*}
		\sum_{j_g\geq0}\biggl((\ol\bsF^{\bsq'}\cap W'_{n'})V_\bbC
		+\sum_{\bsj'\geq\boldsymbol{0}}(\ol\bsF^{\bsq'-\bsj'}\cap W'_{n'-\lvert\bsj'\rvert-1})V_\bbC\biggr)\cap(\ol F_g^{q_g-j_g}\cap W^g_{n_g-j_g-1})V_\bbC\\
		= \sum_{j_g\geq 0}(\ol\bsF^{\bsq'} \cap \ol F^{q_g-j_g}_g \cap W'_{n'}\cap W^g_{n_g-j_g-1})V_\bbC
		+\sum_{\bsj'\geq\boldsymbol 0}\sum_{j_g\geq 0}(\ol\bsF^{\bsq-\bsj'}\cap\ol F^{q_g-j_g}_g \cap W'_{n'-\lvert\bsj'\rvert-1}\cap W^g_{n_g-j_g-1})V_\bbC,
	\end{multline*}
	hence we see that both are subsets of 
	\begin{equation}\label{eq: 3CPO}
		(\ol\bsF^{\bsq}\cap W_{n})V_\bbC+\sum_{\bsj\geq\boldsymbol 0}(\ol\bsF^{\bsq-\bsj}\cap W_{n-\lvert\bsj\rvert-1})V_\bbC.
	\end{equation}
	This and \eqref{eq: R2D2}, we have an inclusion
	\begin{equation}\label{eq: inclusion one}
		A_{1}^{p_1,q_1}(V)\cap\cdots\cap A_{g}^{p_g,q_g}(V)=\bsA^{\bsp',\bsq'}(T^{g-1}_\bbC(V))\cap A_{g}^{p_g,q_g}(V) \subset \bsA^{\bsp,\bsq}(T^g_\bbC(V)).
	\end{equation}
	By Proposition \ref{prop: OF total} and Proposition \ref{prop: psplitting} (1), the homomorphism
	\[\rho:\bsA^{\bsp,\bsq}(T^g_\bbC(V))\rightarrow(\bsF^\bsp\cap\ol\bsF^\bsq)\Gr^W_nV_\bbC\]
	is injective.
	Then we obtain
	\begin{equation}\label{injection to Gr}
		V_\bbC=\bigoplus_{\bsp,\bsq\in\bbZ^g}A_{1}^{p_1,q_1}(V)\cap\cdots\cap A_{g}^{p_g,q_g}(V)
		\hookrightarrow\bigoplus_{\bsp,\bsq\in\bbZ^g} \bsA^{\bsp,\bsq}(T^g_\bbC(V))
		\hookrightarrow\bigoplus_{\bsp,\bsq\in\bbZ^g}(\bsF^\bsp\cap\ol\bsF^\bsq)\Gr^W_{n}V_\bbC.
	\end{equation}
	By Proposition \ref{prop: intermediate}, we have
	\[
		\bigoplus_{\bsp,\bsq\in\bbZ^g}(\bsF^\bsp\cap\ol\bsF^\bsq)\Gr^W_{n}V_\bbC = \bigoplus_{n}\Gr^W_{n}V_\bbC.
	\]
	Since $V_\bbC$ and $\bigoplus_{n\in\bbZ}\Gr^W_nV_\bbC$ have the same dimension, \eqref{injection to Gr} is an isomorphism.
	Hence \eqref{eq: inclusion one} and $\rho$ are isomorphisms
	for any $\bsp,\bsq\in\bbZ^g$, as desired.
\end{proof}

Let $V$ be an object in $\OF_\bbC^g$ and $I\subset\{1,\ldots,g\}$.
By Proposition \ref{prop: intermediate} and Proposition \ref{prop: OF total}, $T^g_\bbC(V)$ satisfies ($\textrm{a}_I$) and ($\textrm{b}_I$) in Definition \ref{def: mixed weak plectic}.
Moreover, by Proposition \ref{prop: pApq} and Proposition \ref{prop: psplitting} (2), $T^g_\bbC(V)$ satisfies ($\textrm{c}_I$).
Hence we completed the proof of Proposition \ref{prop: Hodge}.

%
\subsection{Mixed plectic $\bbC$-Hodge structures}
%

In the previous subsection, we have seen that the functor $T^g_\bbC$ induces the functor $T^g_\bbC:\OF^g_\bbC\rightarrow\wMHS_\bbC^g$.
In this subsection we will characterize the essential image of $\OF_\bbC^g$ by $T^g_\bbC$.

For $I\subset\{1,\ldots,g\}$, we define a functor $P^g_I:\Fil^1_g(\bbC)\rightarrow\Fil^g_{g}(\bbC)$ by sending $V=(V_\bbC,W_\bullet,\{F_\mu^\bullet\},\{\ol F_\mu^\bullet\})$ to $P^g_I(V):=(V_\bbC,\{W^{I,\mu}_\bullet\},\{F_\mu^\bullet\},\{\ol F_\mu^\bullet\})$ with
\begin{equation}
	W^{I,\mu}_nV_\bbC:=\sum_{\substack{\bsp,\bsq\in\bbZ^g\\ p_\mu+q_\mu\leq n}}\bsA_I^{\bsp,\bsq}(V). 
\end{equation}

The goal of this subsection is to prove the following proposition.

\begin{proposition}\label{prop: char of MHS}
	\begin{enumerate}
	\item We have $P_I^g\circ T^g_\bbC(U)=U$ and $T^g_\bbC\circ P_I^g(V)=V$ for any object $U$ in $\OF^g_\bbC$, $V$ in $\wMHS_\bbC^g$, and any subset $I\subset\{1,\ldots,g\}$.
	\item Let $V$ be an object in $\wMHS_\bbC^g$.
		Then $V$ lies in the essential image of $\OF_\bbC^g$ by $T^g_\bbC$ if and only if $P_I^g(V)=P_J^g(V)$ for any $I$ and $J$. 	
	\end{enumerate}
\end{proposition}

According to Proposition \ref{prop: char of MHS}, we define the category of mixed $g$-plectic $\bbC$-Hodge structures as follows.

\begin{definition}\label{def: pMHS}
	We define the category of mixed $g$-plectic $\bbC$-Hodge structures $\MHS_\bbC^g$ to be the full subcategory of $\wMHS_\bbC^g$ consisting of objects $V$ satisfying $W^{I,\mu}_nV_\bbC=W^{J,\mu}_nV_\bbC$ for any $I,J\subset\{1,\ldots,g\}$, $\mu=1,\ldots.g$, and $n\in\bbZ$.
	This says that the object $P^g_\bbC(V):=P_I^g(V)$ is independent of $I$.
	We let $W^\mu_n V_\bbC := W^{I,\mu}_n V_\bbC$ for mixed $g$-plectic $\bbC$-Hodge structures.
\end{definition}

Combining Corollary \ref{cor: main cor} and Proposition \ref{prop: char of MHS}, we obtain the following theorem.

\begin{theorem}\label{thm: equiv for C}
	There are equivalences of categories
	\begin{equation}\label{eq: Rep OF MHS}
		\xymatrix{
		\Rep_\bbC(\cG_\bbC^g)\ar@<1mm>[r]^>>>>{\varphi_\bbC^g}&\OF_\bbC^g\ar@<1mm>[l]^>>>>>{\psi_\bbC^g}\ar@<1mm>[r]^>>>>{T^g_\bbC}&\MHS_\bbC^g.\ar@<1mm>[l]^>>>>>{P^g_\bbC}
		}
	\end{equation}
	Moreover $T^g_\bbC$ and $P^g_\bbC$ are isomorphisms of categories.
\end{theorem}

We may define tensor products and internal homomorphisms in $\MHS_\bbC^g$ as follows.
Suppose $U=(U_\bbC,W^\bullet,\{F_\mu^\bullet\},\{\ol F_\mu^\bullet\})$ and $V=(V_\bbC,W^\bullet,\{F_\mu^\bullet\},\{\ol F_\mu^\bullet\})$ are objects in $\MHS_\bbC^g$.
Then we define the tensor product $U\otimes V$ to be the quadruple
\begin{equation}
	U\otimes V=(U_\bbC\otimes_\bbC V_\bbC,W_\bullet,\{F_\mu^\bullet\},\{\ol F_\mu^\bullet\}),
\end{equation}
where the weight filtration is given by
\begin{equation*}
	W_n(U_\bbC\otimes_\bbC V_\bbC):=\sum_{n_1+n_2=n}W_{n_1}U_\bbC\otimes_\bbC W_{n_2}V_\bbC
\end{equation*}
for any $n\in\bbZ$ and the partial Hodge filtrations are given by
\begin{align*}
	F_\mu^p(U_\bbC\otimes_\bbC V_\bbC)&:=\sum_{p_1+p_2=p}F_\mu^{p_1}U_\bbC\otimes_\bbC F_\mu^{p_2}V_\bbC,\\
	\ol F_\mu^q(U_\bbC\otimes_\bbC V_\bbC)&:=\sum_{q_1+q_2=q}\ol F_\mu^{q_1}U_\bbC\otimes_\bbC\ol F_\mu^{q_2}V_\bbC
\end{align*}
for any $p,q\in\bbZ$ and $\mu=1,\ldots,g$.
Next we define the internal homomorphism $\underline{\Hom}(U,V)$ to be the quadruple
\begin{equation}
	\underline{\Hom}(U,V):=(\Hom_\bbC(U_\bbC,V_\bbC),W_\bullet,\{F_\mu^\bullet\},\{\ol F_\mu^\bullet\}),
\end{equation}
where the weight filtration are given by
\begin{equation*}
	W_n(\Hom_\bbC(U_\bbC,V_\bbC)):=\{\alpha\in\Hom_\bbC(U_\bbC,V_\bbC)\mid\forall m\in\bbZ\ \ \ \alpha(W_mU_\bbC)\subset W_{m+n}V_\bbC\}
\end{equation*}
for any $n \in \bbZ$ and the partial Hodge filtrations are given by
\begin{align*}
	F_\mu^p(\Hom_\bbC(U_\bbC,V_\bbC))&:=\{\alpha\in\Hom_\bbC(U_\bbC,V_\bbC)\mid\forall m\in\bbZ\ \ \ \alpha(F_\mu^mU_\bbC)\subset F^{m+p}_\mu V_\bbC\}\\
	\ol F_\mu^q(\Hom_\bbC(U_\bbC,V_\bbC))&:=\{\alpha\in\Hom_\bbC(U_\bbC,V_\bbC)\mid\forall m\in\bbZ\ \ \ \alpha(\ol F_\mu^mU_\bbC)\subset \ol F_\mu^{m+q}V_\bbC\}
\end{align*}
for any $p,q\in\bbZ$ and $\mu=1,\ldots,g$.
Then one can see that the tensor products and internal homomorphisms in $\MHS_\bbC^g$ are compatible with those in $\Rep_\bbC(\cG_\bbC^g)$ via the equivalences \eqref{eq: Rep OF MHS}.
In particular we obtain the following corollary.

\begin{corollary}
	The category $\MHS_\bbC^g$ is a neutral tannakian category over $\bbC$ with respect to the fiber functor
	\begin{equation}
		\omega^g_\bbC:\MHS_\bbC^g\rightarrow\Vec_\bbC
	\end{equation}
	associating to $V=(V_\bbC,W_\bullet,\{F_\mu^\bullet\},\{\ol F_\mu^\bullet\})$ the $\bbC$-vector space
	\[\Gr^{W^1}_\bullet\cdots\Gr^{W^g}_\bullet V_\bbC:=\bigoplus_{n_1,\ldots,n_g\in\bbZ}\Gr^{W^1}_{n_1}\cdots\Gr^{W^g}_{n_g}V_\bbC.\]
\end{corollary}

In order to prove Proposition \ref{prop: char of MHS}, we prepare some results concerning the pure case.

\begin{definition}[pure plectic $\bbC$-Hodge structure]
	Let $n$ be an integer.
	A \textit{pure $g$-plectic $\bbC$-Hodge structure of weight $n$} is a pure weak $g$-plectic $\bbC$-Hodge structure (Definition \ref{def: pure weak plectic}) which is a mixed $g$-plectic $\bbC$-Hodge structure (Definition \ref{def: pMHS}) via the weight filtration given by $W_{n-1}V_\bbC:=\{0\}$ and $W_nV_\bbC:=V_\bbC$.
\end{definition}

Note that, for a pure weak $g$-plectic $\bbC$-Hodge structure $V$ of weight $n$, the partial weight filtrations on $V_\bbC$ are given by
	\begin{equation}\label{eq: pure partial weight}
		W^{I,\mu}_m V_\bbC: =\bigoplus_{\substack{\bsp,\bsq\in\bbZ^g, |\bsp+\bsq|=n\\ p_\mu+q_\mu\leq m}}(\bsF_I^{\bsp}\cap\ol\bsF_I^{\bsq})V_\bbC.
	\end{equation}

\begin{lemma}\label{lem: Gr pure}
	Let $V$ be an object in $\MHS_\bbC^g$. Then for any $n\in\bbZ$, $W_nV$ is also an object in $\MHS_\bbC^g$, and $\Gr^W_nV$ is a pure $g$-plectic $\bbC$-Hodge structure of weight $n$.
\end{lemma}

\begin{proof}
	By Corollary \ref{cor: coincides}, $W_nV$ is an object in $\wMHS_\bbC^g$ and $\Gr^W_nV$ is a pure weak $\bbC$-Hodge structure of weight $n$.
	By Corollary \ref{cor: strictness of pA}, we have
	\begin{equation}\label{eq: pW for W}
		\bigoplus_{\substack{\bsp,\bsq\in\bbZ^g\\ p_\mu+q_\mu\leq m}}\bsA_I^{\bsp,\bsq}(W_nV)
		=W_nV_\bbC\cap\biggl(\bigoplus_{\substack{\bsp,\bsq\in\bbZ^g\\ p_\mu+q_\mu\leq m}}\bsA_I^{\bsp,\bsq}(V)\biggr)
	\end{equation}
	and
	\begin{equation}\label{eq: pW for Gr}
		\bigoplus_{\substack{\bsp,\bsq\in\bbZ^g\\ p_\mu+q_\mu\leq m}}\bsA_I^{\bsp,\bsq}(\Gr^W_nV)
		=\biggl(\bigoplus_{\substack{\bsp,\bsq\in\bbZ^g\\ p_\mu+q_\mu\leq m}}\bsA_I^{\bsp,\bsq}(W_nV)\biggr)/\biggl(\bigoplus_{\substack{\bsp,\bsq\in\bbZ^g\\ p_\mu+q_\mu\leq m}}\bsA_I^{\bsp,\bsq}(W_{n-1}V)\biggr)
	\end{equation}
	for any $m\in\bbZ$, $\mu=1,\ldots,g$, and $I\subset\{1,\ldots,g\}$.
	Since $W^{I,\mu}_mV_\bbC=\bigoplus_{\substack{\bsp,\bsq\in\bbZ^g\\ p_\mu+q_\mu\leq m}}\bsA_I^{\bsp,\bsq}(V)$ is independent of $I$, \eqref{eq: pW for W} and hence \eqref{eq: pW for Gr} are also independent of $I$.
\end{proof}

\begin{example}
	For $\bsn=(n_\mu)\in\bbZ^g$, let
	$
		\bbC(\bsn) = (V_\bbC, \{V^{\bsp,\bsq}\}, \{t_\mu\})
	$
	be the plectic Tate object of Example \ref{example: tate}.  Then the object in $\MHS^g_\bbC$ which is equivalent to $\bbC(\bsn)$
	via the above equivalence of categories, which we again denote by $\bbC(\bsn)$, may be given by
	\begin{equation*}
		\bbC(\bsn)=(V_\bbC, W_\bullet, \{F^\bullet_\mu\},  \{\ol F^\bullet_\mu\}),
	\end{equation*}
	where $V_\bbC:=\bbC$ is a $\bbC$-vector space of dimension one, the weight filtrations on 
	$V_\bbC$ is given by $W_{-2\lvert\bsn\rvert-1} V_\bbC=0$, $W_{-2\lvert\bsn\rvert} V_\bbC=V_\bbC$, and the partial Hodge filtrations
	on $V_\bbC$ are given by 
	\begin{align*}
		F^{-n_\mu}_\mu V_\bbC&= \ol F^{\,-n_\mu}_\mu V_\bbC = V_\bbC, & 
		F^{-n_\mu+1}_\mu V_\bbC&= \ol F^{\,-n_\mu+1}_\mu V_\bbC = \{0\}
	\end{align*} 
	for $\mu=1,\ldots,g$.  The object $\bbC(\bsn)$ is a pure $g$-plectic $\bbC$-Hodge structure of weight $-2\lvert\bsn\rvert$.
\end{example}
%


\begin{lemma}\label{lem: pure plectic}
	Let $n$ be an integer, and let $V$  be a pure $g$-plectic $\bbC$-Hodge structure of weight $n$.
	Then $P^g_\bbC(V)$ is an object in $\OF_\bbC^g$.
\end{lemma}

\begin{proof}
	We will show that for any $\nu\neq\mu$, the $\bbC$-linear subspaces $W^\mu_l V_\bbC$,
	$F^l_\mu V_\bbC$, and $\ol F^l_\mu V_\bbC$ with the $\nu$-th filtrations are mixed $\bbC$-Hodge structure.
	First, for $W^\mu_lV_\bbC$, we have
	\begin{equation*}
		\Gr^{W^\nu}_m W^\mu_l V_\bbC 
		 \cong\bigoplus_{\substack{\bsp,\bsq\in\bbZ^g, |\bsp+\bsq|=n\\p_\nu+q_\nu=m\\ p_\mu+q_\mu\leq l}}(\bsF_I^{\bsp}\cap\ol\bsF_I^{\bsq})V_\bbC
	\end{equation*}
	for any $I$, and
	 \begin{align*}	
		F^p_\nu\Gr^{W^\nu}_mW^\mu_l V_\bbC
		&\cong\bigoplus_{\substack{\bsr,\bss\in\bbZ^g, |\bsr+\bss|=n\\r_\nu\geq p,\, r_\nu+s_\nu=m\\ r_\mu+s_\mu\leq l}}(\bsF^{\bsr}\cap\ol\bsF^{\bss})V_\bbC,\\
		\ol F^q_\nu\Gr^{W^\nu}_m W^\mu_l V_\bbC&\cong
		\bigoplus_{\substack{\bsr,\bss\in\bbZ^g, |\bsr+\bss|=n\\r_\nu\geq q,\,r_\nu+s_\nu=m\\ r_\mu+s_\mu\leq l}}(\bsF_{\{\nu\}}^{\bsr}\cap\ol\bsF_{\{\nu\}}^{\bss})V_\bbC
		=\bigoplus_{\substack{\bsr,\bss\in\bbZ^g, |\bsr+\bss|=n\\s_\nu\geq q,\,r_\nu+s_\nu=m\\ r_\mu+s_\mu\leq l}}(\bsF^{\bsr}\cap\ol\bsF^{\bss})V_\bbC.
	 \end{align*}
	This shows that we have a splitting
	 \begin{equation*}
		 \Gr^{W^\nu}_m W^\mu_l V_\bbC = 
		 F^p_\nu  \Gr^{W^\nu}_m W^\mu_l V_\bbC \oplus \ol F^{m+1-p}_\nu\Gr^{W^\nu}_m W^\mu_l V_\bbC
	 \end{equation*}
	 for any $p,q \in \bbZ$.  Hence we see that $W^\mu_l V_\bbC$ with the $\nu$-th filtrations is a mixed $\bbC$-Hodge structure as desired.
	 Similarly, for $F_\mu^l V_\bbC$, we have
	 	\[\Gr^{W^\nu}_mF_\mu^l V_\bbC\cong\bigoplus_{\substack{\bsp,\bsq\in\bbZ^g,\ \lvert\bsp+\bsq\rvert=n\\ p_\nu+q_\nu=m\\ p_\mu\geq l}}(\bsF_I^\bsp\cap\ol\bsF_I^\bsq)V_\bbC\]
		for any $I\not\ni\mu$, and
		\begin{align*}
			F_\nu^p\Gr^{W^\nu}_mF_\mu^l V_\bbC&\cong\bigoplus_{\substack{\bsr,\bss\in\bbZ^g,\ \lvert\bsr+\bss\rvert=n\\ r_\nu\geq p,\ r_\nu+s_\nu=m\\ r_\mu\geq l}}(\bsF^\bsr\cap\ol\bsF^\bss)V_\bbC,\\
			\ol F_\nu^q\Gr^{W^\nu}_mF_\mu^l V_\bbC&\cong\bigoplus_{\substack{\bsr,\bss\in\bbZ^g,\ \lvert\bsr+\bss\rvert=n\\ r_\nu\geq q,\ r_\nu+s_\nu=m\\ r_\mu\geq l}}(\bsF_{\{\nu\}}^\bsr\cap\ol\bsF_{\{\nu\}}^\bss)V_\bbC
			\cong \bigoplus_{\substack{\bsr,\bss\in\bbZ^g,\ \lvert\bsr+\bss\rvert=n\\ s_\nu\geq q,\ r_\nu+s_\nu=m\\ r_\mu\geq l}}(\bsF^\bsr\cap\ol\bsF^\bss)V_\bbC.
		\end{align*}
	Hence we see that $F_\mu^I V_\bbC$ with $\nu$-th filtrations is a mixed $\bbC$-Hodge structure.
	The assertion for $\ol F_\mu^l V_\bbC$ follows from the same argument.
	\end{proof}

Next we will review some facts concerning the extension of mixed Hodge structures with respect to strict morphisms.
We first define exactness of a sequence in $\Fil^1_1(\bbC)$ and recall 
Lemma \ref{lem: extension} which asserts that mixed $\bbC$-Hodge structures are closed under the extension in $\Fil^1_1(\bbC)$.

\begin{definition}\label{def: strict}
	\begin{enumerate}
	\item A morphism $\alpha: U \rightarrow V$ in $\Fil^1_1(\bbC)$ is said to be \textit{strict} if $\alpha$ is strictly compatible with the filtrations $F^\bullet\cap W_\bullet$ and $\ol F^\bullet\cap W_\bullet$.
	\item A sequence
			\[0\rightarrow T\xrightarrow[]{\alpha} U\xrightarrow[]{\beta} V\rightarrow 0\]
		in $\Fil^1_1(\bbC)$ is said to be \textit{exact} if the sequence of underlying $\bbC$-vector space is exact and $\alpha$ and $\beta$ are strict.
	\end{enumerate}
\end{definition}

\begin{lemma}[\cite{H} Lemma 8.1.4 or \cite{PS} Criterion 3.10]\label{lem: extension}
	Let
	\[0\rightarrow T\rightarrow U\rightarrow V\rightarrow 0\]
	be an exact sequence in $\Fil^1_1(\bbC)$.
	If $T$ and $V$ are mixed $\bbC$-Hodge structures, then $U$ is also a mixed $\bbC$-Hodge structures.
\end{lemma}

\begin{remark}
	The strict compatibility with the filtrations $W_\bullet$, $F^\bullet$, and $\ol F^\bullet$ is not sufficient to prove Lemma \ref{lem: extension}.
	Note that by Proposition \ref{prop: strictness of A}, a morphism of mixed $\bbC$-Hodge structures 
	is automatically strict in the sense of Definition \ref{def: strict}.
\end{remark}

\begin{proof}[Proof of Proposition \ref{prop: char of MHS}]
	(1) follows from Proposition \ref{prop: pApq} and Proposition \ref{prop: psplitting}.
	Then it is enough to show that for any object $V=(V_\bbC,W_\bullet,\{F_\mu^\bullet\},\{\ol F_\mu^\bullet\})$ in $\MHS_\bbC^g$,
	the object $P^g_\bbC(V)=(V_\bbC,\{W^\mu_\bullet\},\{F_\mu^\bullet\},\{\ol F_\mu^\bullet\})$ lies in $\OF_\bbC^g$. Here $W^\mu_\bullet$ denotes $W^{I,\mu}_\bullet$, which is independent of $I$.
	First we show that $(W_n\cap F_\mu^l)V_\bbC$ with $\nu$-th filtrations is a mixed $\bbC$-Hodge structure for any $\mu\neq\nu$ and $n,l\in\bbZ$ by induction on $n$.
	This is true for $n$ sufficiently small.
	Assume $(W_{n-1}\cap F_\mu^l)V_\bbC$ with $\nu$-th filtrations is a mixed $\bbC$-Hodge structure.
	We have a short exact sequence of $\bbC$-vector spaces
	\begin{equation}\label{eq: SES2}
		0\rightarrow (W_{n-1}\cap F_\mu^l)V_\bbC\rightarrow (W_n\cap F_\mu^l)V_\bbC\rightarrow F_\mu^l\Gr^W_nV_\bbC\rightarrow 0.
	\end{equation}
	Since $W_\bullet$, $F_\mu^\bullet$, $W^\nu_\bullet$, and $F_\nu^\bullet$ can be written as direct sums of $\bsA^{\bsp,\bsq}(V)$,
	the sequence \eqref{eq: SES2} is strictly compatible with $F_\nu^\bullet\cap W^\nu_\bullet$.
	Similarly, since $W_\bullet$, $F_\mu^\bullet$, $W^\nu_\bullet$, and $\ol F_\nu^\bullet$ can be written as direct sums of $\bsA_{\{\nu\}}^{\bsp,\bsq}(V)$,
	the sequence \eqref{eq: SES2} is strictly compatible with $\ol F_\nu^\bullet\cap W^\nu_\bullet$.
	Moreover $F_\mu^l\Gr^W_nV_\bbC$ with $\nu$-th filtrations is a mixed $\bbC$-Hodge structure by Lemma \ref{lem: Gr pure} and Lemma \ref{lem: pure plectic}.
	Hence $(W_n\cap F_\mu^l)V_\bbC$ with $\nu$-th filtrations is also a mixed $\bbC$-Hodge structure by Lemma \ref{lem: extension}.
	Since $W_nV_\bbC=V_\bbC$ for $n$ sufficiently large, we see that $F_\mu^l V_\bbC$ with $\nu$-th filtrations is again a mixed $\bbC$-Hodge structure as desired.
	The claims for $W^\mu_l V_\bbC$ and $\ol F_\mu^l V_\bbC$ may be proved in a similar fashion.
\end{proof}

\begin{example}
	We note that $\MHS^g_\bbC$ is strictly smaller than $\wMHS_\bbC^g$ for any $g>1$.
	For example, consider the case when $g=2$ and let $V_\bbC:=\bbC e_0\oplus\bbC e_{-4}$ with the filtrations defined by
	\begin{align*}
		W_n V_\bbC &:= \begin{cases}
				0  &  n \leq -5, \\
				\bbC e_{-4} &  n=-4,\ldots,-1,\\
				V_\bbC  &  n \geq 0,
		\end{cases}\\
		F^{p_1}_1 V_\bbC=\ol F^{p_1}_1 V_\bbC &:= \begin{cases}
				V_\bbC &  p_1 \leq 0, \\
				\bbC e_{-4} &  p_1=1,\\
				0  &  p_1\geq 2,
		\end{cases}
	\end{align*}
	\begin{align*}
		F^{p_2}_2 V_\bbC := \begin{cases}
				V_\bbC &  p_2 \leq -3, \\
				\bbC (e_0+ie_{-4}) &  p_2=-2,-1,0,\\
				0  &  p_2\geq 1,
		\end{cases}
		&&\text{and}&&
		\ol F^{p_2}_2 V_\bbC := \begin{cases}
				V_\bbC &  p_2 \leq -3, \\
				\bbC (e_0-ie_{-4}) &  p_2=-2,-1,0,\\
				0  &  p_2\geq 1.
		\end{cases}
	\end{align*}
	Then one can show that $V=(V_\bbC,W_\bullet,\{F_1^\bullet,F_2^\bullet\},\{\ol F_1^\bullet,\ol F_2^\bullet\})$ defined as above is an object in $\wMHS^2_\bbC$.
	However, since $W^{\emptyset,1}_0 V_\bbC=\bbC(e_0+ie_{-4})$ and $W^{\{2\},1}_0V_\bbC=\bbC(e_0-ie_{-4})$, this $V$ is not an object in $\MHS^2_\bbC$.
	\end{example}

%
%
%
%
%
\section{Mixed plectic $\bbR$-Hodge structures and the calculation of extension groups}
%
%
%
%
%

Let $\cG$ be the tannakian fundamental group of the category of mixed $\bbR$-Hodge structures $\MHS_\bbR$,
and for any integer $g\geq0$, consider the category $\Rep_\bbR(\cG^g)$ of finite representations of $\cG^g$.
In this section, we consider the real version of the theory discussed in the previous sections,
and will calculate the extension groups in the category $\Rep_\bbR(\cG^g)$.
In particular, we will define a functor $\Lambda^\bullet$, 
which associates to a complex $U^\bullet$ in $\Rep_\bbR(\cG^g)$
a complex of $\bbR$-vector spaces.  We will prove in Theorem \ref{thm: old theorem}
that $\Lambda^\bullet(U^\bullet)$ calculates the extension groups $\Ext^m_{\Rep_\bbR(\cG^g)}(\bbR(\boldsymbol{0}),U^\bullet)$ of $U^\bullet$ by $\bbR(\boldsymbol{0})$ in $\Rep_\bbR(\cG^g)$.

%
\subsection{Mixed plectic $\bbR$-Hodge structures}
%

Let $g$ be an integer $\geq 0$.
In this subsection, we first give an explicit description of the category $\Rep_\bbR(\cG^g)$.  We then
define the categories $\MHS_\bbR^g$ of mixed $g$-plectic $\bbR$-Hodge structures and $\OF_\bbR^g$ of 
$g$-orthogonal families of mixed $\bbR$-Hodge structures.

\begin{proposition}\label{prop: explicit R-rep}
	An object $\Rep_\bbR(\cG^g)$ uniquely corresponds to a triple
	$U := (U_\bbR, \{  U^{\bsp,\bsq} \}, \{ t_\mu\})$, where $U_\bbR$ is a finite dimensional $\bbR$-vector space,
	$\{ U^{\bsp,\bsq} \}$ is a $2g$-grading of $U_\bbC:=U_\bbR\otimes_\bbR\bbC$ by $\bbC$-linear subspaces 
	\begin{equation*}
		U_\bbC = \bigoplus_{\bsp,\bsq \in \bbZ^g} U^{\bsp,\bsq}
	\end{equation*}
 	such that $\ol{U^{\bsp,\bsq}} = U^{\bsq,\bsp}$ for any $\bsp,\bsq\in\bbZ^g$, 
	and $t_\mu$ for $\mu =1, \ldots, g$ are 
 	$\bbC$-linear automorphisms of $U_\bbC$ commutative with each other, satisfying  $\ol{t_\mu} = t^{-1}_\mu$ and
 	\begin{equation*}
	 	(t_\mu -1)(U^{\bsp,\bsq}) \subset \bigoplus_{\substack{\bsr,\bss\in\bbZ^g\\ (r_\nu,s_\nu)=(p_\nu,q_\nu)\text{ for }
		\nu\neq\mu\\(r_\mu,s_\mu) < (p_\mu, q_\mu)\quad}} U^{\bsr\!,\!\,\bss}
 	\end{equation*}
	for any $\bsp, \bsq\in \bbZ^g$.
	A morphism in  $\Rep_\bbR(\cG^g)$ uniquely corresponds to an $\bbR$-linear homomorphism of underlying $\bbR$-vector spaces compatible 
	with the $2g$-gradings and commutes with $t_\mu$.
\end{proposition}

\begin{proof}
	Our assertion follows the proof of Corollary \ref{cor: main cor}, noting that the compatibility of the structures for each $\mu$ corresponds to the
	fact that the action of each component of $\cG$ on the representation is commutative.
\end{proof}

\begin{example}[Tate object]\label{example: realtate}
	The plectic Tate object in $\Rep_\bbR(\cG^g)$ is given by
	$
		\bbR(\boldsymbol 1_\mu) := (V_\bbR, \{V^{\bsp,\bsq}\}, \{t_\mu\}),
	$
	where $V_\bbR := (2\pi i) \bbR \subset \bbC$ and the grading of $V_\bbR\otimes_\bbR\bbC=\bbC$
	is the one-dimensional $\bbC$-vector space whose sole non-trivial index is at
	\begin{equation*}
		\bsp,\bsq = (0, \ldots, -1, \ldots, 0)
	\end{equation*}
	where $-1$ is at the $\mu$-th component,
	and $t_\mu$ is the identity map for $\mu=1, \ldots, g$.  For any $\bsn \in\bbZ^g$, we let 
	\begin{equation*}
		\bbR(\bsn) := \bigotimes_{\mu=1}^g \bbR(\boldsymbol 1_\mu)^{\otimes n_\mu} =  \bbR(\boldsymbol 1_1)^{\otimes n_1} \otimes \cdots 
		\otimes \bbR(\boldsymbol 1_g)^{\otimes n_g}.
	\end{equation*}
\end{example}

\begin{definition}[orthogonal family of mixed $\bbR$-Hodge structures]
	Let $V=(V_\bbR,\{W^\mu_\bullet\},\{F_\mu^\bullet\})$ be a triple consisting of
	a finite dimensional $\bbR$-vector space $V_\bbR$,
	a family of finite ascending filtrations $W^\mu_\bullet$ by $\bbR$-linear subspaces on $V_\bbR$ for $\mu=1,\ldots,g$,
	and a family of finite descending filtrations $F_\mu^\bullet$ by $\bbC$-linear subspaces on $V_\bbC:=V_\bbR\otimes_\bbR\bbC$ for $\mu=1,\ldots,g$.
	We again denote by $W^\mu_\bullet$ the filtration on $V_\bbC$ defined by $W^\mu_nV_\bbC:=W^\mu_nV_\bbR\otimes_\bbR\bbC$.
	Let $\ol F_\mu^\bullet$ be the filtration on $V_\bbC$ given by the complex conjugate of $F_\mu^\bullet$.
	Then $V$ is called an \textit{$g$-orthogonal family of mixed $\bbR$-Hodge structures}
	if the quadruple $(V_\bbC,\{W^\mu_\bullet\},\{F_\mu^\bullet\},\{\ol F_\mu^\bullet\})$ is an $g$-orthogonal family of mixed $\bbC$-Hodge structures.

	A morphism of $g$-orthogonal families of mixed $\bbR$-Hodge structures is an $\bbR$-linear homomorphism of the underlying $\bbR$-vector spaces compatible with $W^\mu_\bullet$ and $F_\mu^\bullet$.

	We denote the category of $g$-orthogonal families of mixed $\bbR$-Hodge structures by $\OF_\bbR^g$.
\end{definition}

\begin{definition}[mixed plectic $\bbR$-Hodge structure]
	Let $V=(V_\bbR,W_\bullet,\{F_\mu^\bullet\})$ be a triple consisting of
	a finite dimensional $\bbR$-vector space $V_\bbR$,
	a finite ascending filtration $W_\bullet$ by $\bbR$-linear subspaces on $V_\bbR$,
	and a family of finite descending filtrations $F_\mu^\bullet$ by $\bbC$-linear subspaces on $V_\bbC:=V_\bbR\otimes_\bbR\bbC$ for $\mu=1,\ldots,g$.
	We again denote by $W_\bullet$ the filtration on $V_\bbC$ defined by $W_nV_\bbC:=W_nV_\bbR\otimes_\bbR\bbC$.
	Let $\ol F_\mu^\bullet$ the filtration on $V_\bbC$ given by the complex conjugate of $F_\mu^\bullet$.
	Then $V$ is called a \textit{mixed $g$-plectic $\bbR$-Hodge structure}
	if the quadruple $(V_\bbC,W_\bullet,\{F_\mu^\bullet\},\{\ol F_\mu^\bullet\})$ is a mixed $g$-plectic $\bbC$-Hodge structure.

	A morphism of mixed $g$-plectic $\bbR$-Hodge structures is an $\bbR$-linear homomorphism of the underlying $\bbR$-vector spaces compatible with $W_\bullet$ and $F_\mu^\bullet$.

	We denote the category of mixed $g$-plectic $\bbR$-Hodge structures by $\MHS_\bbR^g$.
\end{definition}

A real structure on a $\bbC$-vector space $V_\bbC$ is an anti-linear involution $\sigma:V_\bbC\rightarrow V_\bbC$.
Then one can regard an object in $\Rep_\bbR(\cG^g)$ (resp. $\OF_\bbR^g$, $\MHS_\bbR^g$) as a pair of an object in $\Rep_\bbC(\cG_\bbC^g)$ (resp. $\OF_\bbC^g$, $\MHS_\bbC^g$) and a real structure, in the following sense.

\begin{lemma}\label{lem: real str}
	\begin{enumerate}
		\item The category $\Rep_\bbR(\cG^g)$ is naturally equivalent to the category $\wt{\Rep_\bbR(\cG^g)}$ consisting of pairs $(U,\sigma)$,
			where $U=(U_\bbC,\{U^{\bsp,\bsq}\},\{t_\mu\})$ is an object in $\Rep_\bbC(\cG_\bbC^g)$, and $\sigma$ is a real structure on $U_\bbC$
			satisfying $\sigma(U^{\bsp,\bsq})=U^{\bsq,\bsp}$ for any $\bsp,\bsq\in\bbZ^g$ and $\sigma\circ t_\mu\circ\sigma=t_\mu^{-1}$ for any $\mu=1,\ldots,g$.
		\item The category $\OF_\bbR^g$ is naturally equivalent to the category $\wt{\OF_\bbR^g}$ consisting of pairs $(V,\sigma)$,
			where $V=(V_\bbC,\{W^\mu_\bullet\},\{F_\mu^\bullet\},\{\ol F_\mu^\bullet\})$ is an object in $\OF_\bbC^g$, and $\sigma$ is a real structure on $V_\bbC$
			satisfying $\sigma(W^\mu_n V_\bbC)=W^\mu_nV_\bbC$ and $\sigma(F_\mu^p V_\bbC)=\ol F_\mu^pV_\bbC$ for any $\mu=1,\ldots,g$ and $n,p\in\bbZ$.
		\item The category $\MHS_\bbR^g$ is naturally equivalent to the category $\wt{\MHS_\bbR^g}$ consisting of pairs $(V,\sigma)$,
			where $V=(V_\bbC,W_\bullet,\{F_\mu^\bullet\},\{\ol F_\mu^\bullet\})$ is an object in $\MHS_\bbC^g$, and $\sigma$ is a real structure on $V_\bbC$
			satisfying $\sigma(W_n V_\bbC)=W_nV_\bbC$ and $\sigma(F_\mu^p V_\bbC)=\ol F_\mu^pV_\bbC$ for any $\mu=1,\ldots,g$ and $n,p\in\bbZ$.
	\end{enumerate}
\end{lemma}

\begin{proof}
	The lemma immediately follows from the fact that a real structure $\sigma$ on $V_\bbC$ uniquely corresponds to
	an $\bbR$-linear subspace $V_\bbR\subset V_\bbC$ such that the natural homomorphism $V_\bbR\otimes_\bbR\bbC\rightarrow V_\bbC$ is an isomorphism, by taking the fixed part of $\sigma$.
\end{proof}

Let $(V,\sigma)$ be an object in $\wt{\OF_\bbR^g}$.
Since each $W^\mu_\bullet$ is stable under $\sigma$, it induces a real structure $\Gr(\sigma)$ of $\Gr^{W^1}_\bullet\cdots\Gr^{W^g}_\bullet V_\bbC$.

\begin{lemma}\label{lem: equiv for real str}
	The associations
	\begin{align*}
		\wt\varphi^g_\bbR(U,\sigma)&:=(\varphi^g_\bbC(U),\sigma),
		&\wt\psi^g_\bbR(V,\sigma)&:=(\psi_\bbC^g(V),\Gr(\sigma)),\\
		\wt T_\bbR^g(V,\sigma)&:=(T_\bbC^g(V),\sigma),
		&\wt P^g_\bbR(V,\sigma)&:=(P^g_\bbC(V),\sigma)
	\end{align*}
	define functors
	\[\xymatrix{
	\wt{\Rep_\bbR(\cG^g)}\ar@<1mm>[r]^>>>>{\wt\varphi_\bbR^g}
	&\wt{\OF_\bbR^g}\ar@<1mm>[l]^>>>>>{\wt\psi_\bbR^g}\ar@<1mm>[r]^>>>>{\wt T^g_\bbR}
	&\wt{\MHS_\bbR^g},\ar@<1mm>[l]^>>>>>{\wt P^g_\bbR}
	}\]
	which are equivalences of categories.
	Moreover $\wt T^g_\bbR$ and $\wt P^g_\bbR$ are isomorphisms of categories.
\end{lemma}

\begin{proof}
	By using Theorem \ref{thm: equiv for C}, one can check straightforwardly.
\end{proof}

By Lemma \ref{lem: real str} and Lemma \ref{lem: equiv for real str}, we obtain the following theorem.

\begin{theorem}\label{thm: equiv for R}
	There are equivalences of categories
	\begin{equation}\label{eq: equiv for R}
		\xymatrix{
		\Rep_\bbR(\cG^g)\ar@<1mm>[r]^>>>>{\varphi_\bbR^g}
		&\OF_\bbR^g\ar@<1mm>[l]^>>>>>{\psi_\bbR^g}\ar@<1mm>[r]^>>>>{T^g_\bbR}
		&\MHS_\bbR^g,\ar@<1mm>[l]^>>>>>{P^g_\bbR}
		}
	\end{equation}
	where the functors $\varphi_\bbR^g$, $\psi_\bbR^g$, $T^g_\bbR$, and $P^g_\bbR$ are induced from the functors $\wt\varphi_\bbR^g$, $\wt\psi_\bbR^g$, $\wt T^g_\bbR$, and $\wt P^g_\bbR$ respectively.
	Moreover $T^g_\bbR$ and $P^g_\bbR$ are isomorphisms of categories.
\end{theorem}

We define the tensor products and internal homomorphisms in $\OF_\bbR^g$ and $\MHS_\bbR^g$ in a similar fashion to $\OF_\bbC^g$ and $\MHS_\bbC^g$.
Then one can see that they are compatible with tensor products and internal homomorphism in $\Rep_\bbR(\cG^g)$ via the equivalences \eqref{eq: equiv for R}.
In particular we have the following corollary.

\begin{corollary}
	The category $\MHS_\bbR^g$ is a neutral tannakian category over $\bbR$ with the fiber functor
	\begin{equation}
		\omega_\bbR^g:\MHS_\bbR^g\rightarrow \Vec_\bbR
	\end{equation}
	associating to $V=(V_\bbR,W_\bullet,\{F_\mu^\bullet\})$ the $\bbR$-vector space
	\[\Gr^{W^1}_\bullet\cdots\Gr^{W^g}_\bullet V_\bbR:=\bigoplus_{n_1,\ldots,n_g\in\bbZ}\Gr^{W^1}_{n_1}\cdots\Gr^{W^g}_{n_g}V_\bbR.\]
\end{corollary}

%
\subsection{Representations of products of affine group schemes}
%

In this subsection, we will prove Theorem \ref{theorem: external} concerning a property of the representations of products of affine group schemes,
and as a corollary, we show in Corollary \ref{cor: subobject} that any object in $\Rep_\bbR(\cG^g)$ is isomorphic to a subquotient of a
$g$-fold exterior product of objects in $\Rep_\bbR(\cG)$.
This result will be used later in the proof of Theorem \ref{thm: old theorem}.

Let $\cH$ be an affine group scheme over a field $k$.
We let $A: = k(\cH)$ be the affine coordinate ring of $\cH$ so that $\cH = \Spec A$.
Then $A$ is a commutative $k$-algebra, and the group scheme structure on $\cH$ is equivalent to the comultiplication, counit, and inversion maps
\begin{align*}
	\Delta  &: A \rightarrow A \otimes_k A,  &
	\varepsilon  &:  A \rightarrow k, &
	\iota  &:  A \rightarrow A
\end{align*}
which are homomorphisms of $k$-algebras satisfying
\begin{align*}
	(\id \otimes \Delta) \circ \Delta &= (\Delta\otimes\id) \otimes \Delta, &
	 (\varepsilon \otimes\id) \circ \Delta &=  (\id\otimes\varepsilon) \circ \Delta =  \id, \\
	m \circ (\iota\otimes \id) \circ \Delta &= m \circ (\id \otimes \iota) \circ \Delta  = i \circ\varepsilon,&
\end{align*}
where $i: k \rightarrow A$ is the inclusion giving the $k$-algebra structure of $A$
and $m: A \otimes_k A \rightarrow A$ is the multiplication.
A commutative $k$-algebra $A$ with the above additional structures is
called a \textit{commutative $k$-Hopf algebra} (or a \textit{$k$-bialgebra} in \cite{DM}).

In what follows, all unmarked tensor products $\otimes$ are tensor products $\otimes_k$ over the field $k$.
For a $k$-vector space $V$, an \textit{$A$-comodule structure} on $V$ is a $k$-linear homomorphism $\phi: V \rightarrow V \otimes  A$ such that 
the composite
\begin{equation*}
	V \xrightarrow\phi V \otimes  A \xrightarrow{\id\otimes\varepsilon} V \otimes  k \cong V
\end{equation*}
is the identity map and
\begin{equation*}
	(\id\otimes\Delta)\circ\phi=(\phi\otimes\id)\circ\phi.
\end{equation*}
By \cite[Proposition 2.2]{DM}, there exists a one-to-one correspondence between $A$-comodule structures on $V$
and $k$-linear representations of $\cH$ on $V$.  In what follows, a representation will always signify a $k$-linear representation
on a $k$-vector space.

For the special case $U: = A$ with the comodule structure 
\begin{equation*}
	\Delta:  U \rightarrow U \otimes  A
\end{equation*}
induced from the multiplication of $\cH$, the corresponding representation of $\cH$ on $U$ is called the
\textit{regular representation} of $\cH$.  The regular representation $U$ of $\cH$ is faithful; in other words, $\Ker(\cH \rightarrow GL_U) = \{ 1 \}$.

Consider affine group schemes $\cH_1$ and $\cH_2$ over a field $k$, and let $A_1 := k(\cH_1)$ and $A_2:= k(\cH_2)$ be the affine coordinate 
rings of $\cH_1$ and $\cH_2$.
For representations $U_1$ and $U_2$ of $\cH_1$ and $\cH_2$, we denote by $U_1 \boxtimes U_2$ the \textit{exterior product} of $U_1$ and $U_2$, which is
a representation of $\cH_1 \times \cH_2:= \Spec(A_1 \otimes  A_2)$ corresponding to the $A_1 \otimes  A_2$-comodule structure
\begin{equation*}
	U_1 \otimes  U_2 \xrightarrow{\phi_1\otimes\phi_2} (U_1 \otimes  A_1) \otimes  (U_2 \otimes  A_2)
	\cong (U_1 \otimes  U_2) \otimes  (A_1 \otimes  A_2)
\end{equation*}
on $U_1 \otimes  U_2$.
Then we have the following.

\begin{lemma}\label{lemma: regular}
	Let $\cH_1$ and $\cH_2$ be affine group schemes over $k$, and suppose $U_1$ and $U_2$ are regular representations of $\cH_1$ and $\cH_2$.
	Then $U:= U_1 \boxtimes U_2$ is the regular representation of $\cH_1 \times \cH_2$.
\end{lemma}

\begin{proof}
	Let $A_1 := k(\cH_1)$ and $A_2:= k(\cH_2)$.  Then the multiplication of $\cH_1 \times \cH_2$ corresponds to the map of $k$-algebras
	\begin{equation*}
		A_1 \otimes  A_2 \xrightarrow{\Delta_1 \otimes \Delta_2} (A_1 \otimes  A_1) \otimes  (A_2 \otimes  A_2) 
		\cong  (A_1 \otimes  A_2) \otimes  (A_1 \otimes  A_2).
	\end{equation*}
	If we let $U_1 := A_1$ and $U_2 := A_2$, then the above map becomes
	\begin{equation*}
		U_1 \otimes  U_2 \xrightarrow{\Delta_1 \otimes \Delta_2} (U_1 \otimes  A_1) \otimes  (U_2 \otimes  A_2) 
		\cong  (U_1 \otimes  U_2) \otimes  (A_1 \otimes  A_2),
	\end{equation*}
	which by the definition of the exterior product is exactly the $A_1\otimes A_2$-comodule structure on 
	$U_1 \otimes  U_2$ giving the exterior product $U_1 \boxtimes U_2$.
\end{proof}

In what follows, a \textit{finite representation} of $\cH$ will signify a $k$-linear representation of $\cH$ on a finite dimensional $k$-vector space.
Let $\Rep_k(\cH)$ be the category of finite representations of $\cH$.
The purpose of this subsection is to prove the following result.

\begin{theorem}\label{theorem: external}
	For $\mu=1, \ldots, g$, let $\cH_{\mu}$ be an affine group scheme over $k$.
	If $V$ is a finite representation of $\cH_1\times\cdots\times \cH_g$, then $V$ is isomorphic  to a subquotient of an 
	object of the form $V_1 \boxtimes\cdots\boxtimes V_g$ for some finite representations $V_{\mu}$ of $\cH_{\mu}$.
\end{theorem}

We say that an affine group scheme $\cH$ over $k$ is an \textit{algebraic group}, if the affine coordinate ring $A := k(\cH)$ is  finitely 
generated as an algebra over $k$.  
We will first prove Proposition \ref{prop: external}, which is a particular case of Theorem \ref{theorem: external} 
when $\cH_{\mu}$ are algebraic groups.
The following result characterizes algebraic groups.

\begin{proposition}[\cite{DM} Corollary 2.5] 
	Suppose $\cH$ is an affine group scheme.
	Then $\cH$ is an algebraic group if and only if there exists a finite faithful representation of $\cH$.
\end{proposition}

We say that a finite representation $W$ of $\cH$ is a \textit{tensor generator} of $\Rep_k(\cH)$, if every object $V$ in $\Rep_k(\cH)$ is 
isomorphic to a subquotient of $P_V(W, W^\vee)$ for some polynomial $P_V(X,Y) \in \bbN[X,Y]$.
Note that if $P_V(X, Y) = \sum_{m,n \in \bbN} a_{mn} X^m Y^n  \in \bbN[X,Y]$, then 
\begin{equation*}
	P_V(W, W^\vee) :=  \bigoplus_{m,n \geq 0}(W^{\otimes  m}\otimes  W^{\vee \otimes  n})^{\oplus a_{mn}}.
\end{equation*}
Proposition \ref{prop: external} will be proved using the following result.

\begin{proposition}[\cite{DM} Proposition 2.20(b)]\label{prop: generator}
	Suppose $\cH$ is an algebraic group.  If $W$ is a finite faithful representation of $\cH$,
	then $W$ is a tensor generator of $\Rep_k(\cH)$.  
	Conversely, any tensor generator of $\Rep_k(\cH)$ is a finite faithful representation of $\cH$.
\end{proposition}

\begin{proposition}\label{prop: external}
	For $\mu=1, \ldots, g$, let $\cH_{\mu}$ be an affine algebraic group over $k$.  
	If $V$ is a finite representation of $\cH_1\times \cdots \times \cH_g$, then $V$ is isomorphic  to a subquotient of an 
	object of the form $V_1 \boxtimes \cdots \boxtimes V_g$ for some finite representations $V_{\mu}$  of $\cH_{\mu}$.
\end{proposition}

\begin{proof}
	Let $U_{\mu}$ be the regular representations of $\cH_{\mu}$. 
	Then  by Lemma \ref{lemma: regular}, $U: = U_1 \boxtimes \cdots \boxtimes U_g$ is the regular representation of $\cH:=\cH_1\times \cdots \times \cH_g$.
	By \cite[Corollary 2.4]{DM}, $U$ is the directed union $U = \bigcup_\alpha U^\alpha$ of finite subrepresentations 
	$U^\alpha$ of $\cH$.  Since $U$ is regular and is in particular faithful, we have 
	\begin{equation*}
		\Ker(\cH \rightarrow GL_U) = \bigcap_\alpha \Ker(\cH \rightarrow GL_{U^\alpha}) = \{1\}.
	\end{equation*}
	Since $\cH$ is Noetherian as a topological space, we have $\Ker(\cH \rightarrow GL_{U^\alpha}) = \{1\}$ for some $\alpha$.
	Hence $U^\alpha$ is a finite dimensional faithful representation of $\cH$.  Let $\{  w^{(i)} \}_{i}$ be a $k$-basis of $U^\alpha$.
	Since $U^\alpha \subset U = U_1 \boxtimes \cdots \boxtimes U_g$, we may write $w^{(i)}$ as a finite sum 
	$w^{(i)} = \sum_{j} a_{ij} w_1^{(i,j)} \otimes \cdots \otimes w_g^{(i,j)}$
	for $a_{ij} \in k$ and $w_{\mu}^{(i,j)} \in U_{\mu}$.  By \cite[Proposition 2.3]{DM}, there exists a finite representation 
	$W_{\mu} \subset U_{\mu}$ of $\cH_{\mu}$ containing
	$\{ w_{\mu}^{(i,j)} \}_{i,j}$.
	Then $W := W_1 \boxtimes \cdots \boxtimes W_g$ is a finite representation of $\cH$, which is faithful since it contains $U^\alpha$ by construction.
	Hence by Proposition \ref{prop: generator}, $W$ is a tensor generator of $\Rep_k(\cH)$.
	By definition of the tensor generator, there exists $P_V(X,Y) \in \bbN[X,Y]$ such that $V$ is isomorphic to a subquotient of $P_V(W, W^\vee)$.
	Since
	\begin{equation*}
		P_V(W, W^\vee) = P_V(W_1 \boxtimes \cdots\boxtimes W_g, W_1^\vee \boxtimes \cdots \boxtimes W_g^\vee) \subset P_V(W_1, W_1^\vee) \boxtimes 
		\cdots \boxtimes P_V(W_g, W_g^\vee),
	\end{equation*}
 	if we let $V_{\mu} := P_V(W_{\mu}, W_{\mu}^\vee)$, then we see that $V$ is isomorphic to a subquotient of $V_1 \boxtimes \cdots \boxtimes V_g$
	as desired.
\end{proof}

The following result will be used to reduce the proof of Theorem \ref{theorem: external} to the case of algebraic groups.

\begin{lemma}[\cite{DM} Proposition 2.6]\label{lemma: bialgebra}
	Let $A$ be a commutative $k$-Hopf algebra.
	Every finite subset of $A$ is contained in a commutative $k$-Hopf subalgebra that is finitely generated as a commutative $k$-algebra.
\end{lemma}

\begin{proof}[Proof of Theorem \ref{theorem: external}]
	Suppose $V$ is a finite representation of $\cH_1 \times \cdots\times \cH_g$.   Let $A_\mu:=k(\cH_\mu)$ for $\mu=1,\ldots,g$.
	Then the representation $V$ is given by some $A_1\otimes \cdots \otimes A_g$-comodule structure
	\begin{equation}\label{eq: comodule}
		\phi: V \rightarrow V \otimes (A_1 \otimes \cdots\otimes A_g)
	\end{equation}
	on $V$.
	Let $\{ v^{(i)}\}_i$ be a $k$-basis of $V$.  Then $\phi(v^{(i)})$ may be written as a finite sum
	\begin{equation*}
		\phi(v^{(i)}) = \sum_{j} v^{(j)} \otimes \biggl( \sum_{k}  a_1^{(i,j,k)} \otimes  \cdots \otimes a_g^{(i,j,k)} \biggr)
	\end{equation*}
	for some $a_{\mu}^{(i,j,k)} \in A_{\mu}$.  By Lemma \ref{lemma: bialgebra}, there exists a Hopf subalgebra
	 $A_{\mu}'$ of $A_{\mu}$ containing $\{  a_{\mu}^{(i,j,k)} \}_{i,j,k}$ which is finitely generated as a $k$-algebra.  
	 Then $\cH'_{\mu}: = \Spec A'_{\mu}$ is an algebraic group over $k$ which is a quotient group scheme of $\cH_{\mu}$. 
	By construction,  the comodule structure \eqref{eq: comodule} on $V$ induces the comodule structure
	\begin{equation*}
		\phi: V \rightarrow V \otimes (A'_1 \otimes \cdots \otimes A'_g),
	\end{equation*}
	hence $V$ is a representation of the algebraic group $\cH_1' \times \cdots\times \cH'_g$.
	By Proposition \ref{prop: external}, $V$ is isomorphic  to a subquotient of an 
	object of the form $V_1 \boxtimes \cdots \boxtimes V_g$ for some finite representations $V_{\mu}$ of $\cH'_{\mu}$.
	Since $\cH'_{\mu}$ is a quotient of $\cH_{\mu}$, the representation $V_{\mu}$ 
	may also be regarded as representation of $\cH_{\mu}$.  Hence $V_1, \ldots,V_g$ satisfy the desired property of our assertion.
\end{proof}

We now return to the case of mixed $\bbR$-Hodge structures.
Let $\cG$ be the tannakian fundamental group of the category of mixed $\bbR$-Hodge structures $\MHS_\bbR$,
and for any integer $g\geq0$, consider the category $\Rep_\bbR(\cG^g)$ of finite representations of $\cG^g$.
Note that the category $\Rep_\bbR(\cG^0)$ is the category 
$\Vec_\bbR$ of finite dimensional $\bbR$-vector spaces.  For $g>0$,
the category $\Rep_\bbR(\cG^g)$ is equivalent to the $g$-fold Deligne tensor product of $\Rep_\bbR(\cG)$ over $\bbR$.
Recall that the \textit{Deligne tensor product} $\sA \boxtimes \sB$ of $k$-linear abelian categories $\sA$ and $\sB$ over a field $k$ is a 
$k$-linear abelian category with a $k$-bilinear functor 
\begin{equation*}
	\boxtimes:  \sA \times  \sB \rightarrow  \sA \boxtimes \sB
\end{equation*}
right exact in each variable, characterized by the property that for any $k$-linear abelian category $\sC$, the induced functor
\begin{equation*}
	\Rex[  \sA \boxtimes \sB,  \sC ] \rightarrow  \Rex_\bil[   \sA \times \sB,  \sC ]
\end{equation*}
gives an equivalence of categories, where 
$\Rex[ \sA \boxtimes \sB, \sC]$ denotes the category of right exact $k$-linear functors from $\sA \boxtimes \sB$ to $\sC$, 
and  $\Rex_\bil[   \sA \times \sB,  \sC ]$ denotes the category of $k$-bilinear functors $\sA \times \sB \rightarrow \sC$ which are 
right exact in each variable.

Since $\MHS_\bbR$ is a tannakian category, it satisfies condition \cite[(2.12.1)]{D2}.  
Hence by \cite[Proposition 5.13 (i)]{D2}, the Deligne tensor products of $\MHS_\bbR$ over $\bbR$ exist.  
A group scheme may be regarded as a groupoid whose class of objects consists of a single element, hence is transitive as a groupoid.  
Then by \cite[5.18]{D2}, there exists a natural equivalence of categories $\Rep_\bbR(\cG^g) \cong\Rep_\bbR(\cG) \boxtimes \cdots \boxtimes \Rep_\bbR(\cG)$,
which gives the equivalence of categories
\begin{equation*}
	\Rep_\bbR(\cG^g)
	\cong  \Rep_\bbR(\cG) \boxtimes \cdots \boxtimes \Rep_\bbR(\cG) 
	\cong \MHS_\bbR \boxtimes \cdots \boxtimes \MHS_\bbR.
\end{equation*}

Hence as a corollary of Theorem \ref{theorem: external}, we have the following.

\begin{corollary}\label{cor: subobject}
	Let $V$ be an object in $\Rep_\bbR(\cG^g)$.
	Then $V$ is isomorphic to a subquotient of $V_1 \boxtimes \cdots \boxtimes V_g$
	for some objects $V_1, \ldots, V_g$ in $\Rep_\bbR(\cG)$.
\end{corollary}

%
\subsection{The functor $\Lambda^\bullet$.}
%

In this subsection, we will define the functor $\Lambda^\bullet$.
In what follows, for any abelian category $\sA$, we denote by $\sC^b(\sA)$ the category of bounded complexes in $\sA$.
We denote its homotopy and derived categories by $\sK^b(\sA)$ and $\sD^b(\sA)$.

Let $U =  (U_\bbR, \{  U^{\bsp,\bsq} \}, \{ t_\mu\})$ be an object in $\Rep_\bbR(\cG^g)$. 
For each integer $\mu=1,\ldots,g$, we let
\begin{align}\label{eq: A}
	\cA^0_\mu(U) &:= \bigoplus_{\substack{\bsp,\bsq \in \bbZ^g\\ p_\mu = q_\mu = 0}} U^{\bsp,\bsq},  &
	\cA^1_\mu(U) &:= \bigoplus_{\substack{\bsp,\bsq \in \bbZ^g\\ p_\mu, q_\mu < 0}} U^{\bsp,\bsq}.
\end{align}

\begin{definition}\label{def: gamma nu}
	For any non-negative integer $\mu \leq g$, note that we have a natural decomposition $\cG^g =  \cG^\mu \times \cG^{g-\mu}$ of pro-algebraic groups.
	By taking the fixed part with respect to the action of $\cG^{g-\mu}$, we have a functor
	\begin{equation*}
		\Gamma_{\mu}: \Rep_\bbR(\cG^g)\rightarrow\Rep_\bbR(\cG^\mu).
	\end{equation*} 
	On the level of objects, this functor may be described by associating to any object $U$ in $\Rep_\bbR(\cG^g)$ the $\bbR$-vector space
	\begin{equation*}
		\Gamma_{\mu}(U)_\bbR :=  \bigl\{ u \in \cA^0_{\mu+1}(U) \cap \cdots \cap \cA^0_g(U) \mid \ol u=u, \text{$(t_{\nu}-1) u =0$ $(\mu < \nu \leq g)$} \bigr\}
	\end{equation*}
	with the induced $2\mu$-grading and $\bbC$-linear automorphism $t_\nu$ for $\nu=1,\ldots, \mu$, giving an object in $\Rep_\bbR(\cG^\mu)$.
\end{definition}

The functor $\Gamma_\mu:  \Rep_\bbR(\cG^g) \rightarrow \Rep_\bbR(\cG^\mu)$ defines a functor 
\begin{equation*}
	\Gamma_{\mu}:  \sC^b(\Rep_\bbR(\cG^g)) \rightarrow \sC^b(\Rep_\bbR(\cG^\mu))
\end{equation*} 
from the category of complexes of $\Rep_\bbR(\cG^g)$ to that of $\Rep_\bbR(\cG^\mu)$.
Let $T^\bullet$ and $U^\bullet$ be complexes in $\sC^b(\Rep_\bbR(\cG^g))$.  We let $\ul\Hom^\bullet(T^\bullet, U^\bullet)$ be the complex
\begin{equation*}
	\ul\Hom^n(T^\bullet, U^\bullet) := \prod_{i \in \bbZ} \ul\Hom(T^i, U^{i+n})
\end{equation*}
given by the internal homomorphisms in $\Rep_\bbR(\cG^g)$,
whose differential is defined by
\begin{equation*}
	d^n(\{f^i\}) := \{d_U^{i+1}\circ f^i-(-1)^nf^{i+1}\circ d^i_T\}
\end{equation*}
for any $\{f^i\}\in \ul\Hom^n(T^\bullet,U^\bullet)_\bbR$.  
Then we have the following.

\begin{lemma}\label{trivial}
	For any $m \in \bbZ$, we have
	\begin{equation*}
		H^m(\Gamma_0(\ul\Hom^\bullet(T^\bullet, U^\bullet))) = \Hom_{\sK^b(\Rep_\bbR(\cG^g))}(T^\bullet, U^\bullet[m]).
	\end{equation*}
\end{lemma}

\begin{proof}
	An element $f \in \ul\Hom^m(T^\bullet, U^\bullet)_\bbR=  \prod_{n \in \bbZ} \ul\Hom(T^{n}, U^{m+n})_\bbR$ 
	defines an $\bbR$-linear homomorphism $f: T^\bullet_\bbR \rightarrow U^\bullet_\bbR[m]$ if and only if $f$ is an $m$-cocycle.
	Such an $f$ preserves the grading if and only if $f\in \ul\Hom^m(T^\bullet, U^\bullet)^{\boldsymbol{0},\boldsymbol{0}}$,
	and commutes with $t_\mu$ if and only if $t_\mu(f) = f$ in $\ul\Hom^m(T^\bullet, U^\bullet)_\bbC$.
	Finally, the map of complexes induced by $f$ is homotopic to zero if and only if $f$ is a coboundary.
\end{proof}

In order to study the functor $\Gamma_0$, we 
first define a series of exact functors
$
	\cA^{m_1,\ldots,m_g} 
$
as follows.
\begin{definition}\label{def: A}
Let $(m_1, \ldots, m_g) \in \{ 0, 1 \}^g$.  We define the functor
$
	\cA^{m_1,\ldots,m_g} : \Rep_\bbR(\cG^g) \rightarrow \Vec_\bbR
$
by associating to any $U \in \Rep_\bbR(\cG^g)$ the $\bbR$-vector space
\begin{equation*}
	\cA^{m_1,\ldots,m_g}(U) := \bigl\{ v \in \cA^{m_1}_1(U)  \cap  \cdots \cap \cA^{m_g}_g(U) \mid (-t_1)^{m_1} \cdots (-t_g)^{m_g} \ol v = v\bigr\},
\end{equation*}
where  $\cA^{m_\mu}_\mu(U)$ are defined as in \eqref{eq: A}.
\end{definition}

\begin{lemma}\label{lem: exact}
	The functors
	$
		\cA^{m_1,\ldots,m_g}	
	$ 
	are exact.
\end{lemma}

\begin{proof}
		By definition, the functor $\cA^{m_1}_1 \cap \cdots \cap \cA^{m_g}_g$ is exact,
		hence the functor $\cA^{m_1, \ldots, m_g}$ is left exact.  Suppose we have a surjective map  $T \rightarrow U$ in $\Rep_\bbR(\cG^g)$.
		For $v \in \cA^{m_1, \ldots, m_g}(U) \subset \cA^{m_1}_1(U) \cap \cdots \cap \cA^{m_g}_g(U)$, take a lift 
		$	
			u \in \cA^{m_1}_1(T) \cap \cdots \cap \cA^{m_g}_g(T).
		$
		Then
		\begin{equation*}
			u' := (u+(-t_1)^{m_1} \cdots (-t_g)^{m_g} \ol u)/2
		\end{equation*}
		is again a lift of $v$ satisfying $u' \in \cA^{m_1, \ldots, m_g}(T)$.
\end{proof}

Suppose $U$ is an object in $\Rep_\bbR(\cG^g)$.  Then $\cA^{\bullet, \ldots, \bullet}(U)$ gives a $g$-tuple complex,
with the $\mu$-th differential given by
\begin{equation*}\xymatrix{%
	\partial_\mu^{m_1,\ldots, m_{\mu-1}, 0, m_{\mu+1}, \ldots,m_g}: \cA^{m_1, \ldots,m_{\mu-1},0,m_{\mu+1},\ldots, m_g}(U) 
	\ar[r]^<<<<{t_\mu-1}&  \cA^{m_1, \ldots,m_{\mu-1},1,m_{\mu+1},\ldots,m_g}(U).
}\end{equation*}%
\begin{example}
	For $g=2$, the double complex $\cA^{\bullet,\bullet}(U)$ for $U$ in $\Rep_\bbR(\cG^2)$ is given by
	\begin{equation*}
	\xymatrix{%
		\cA^{0,0}(U)  \ar[d]_{t_2-1} \ar[r]^{t_1-1}& \cA^{1,0}(U)  \ar[d]^{t_2-1}\\
		\cA^{0,1}(U)  \ar[r]^{t_1-1}&\cA^{1,1}(U).
	}%
	\end{equation*}
\end{example}

If $U^\bullet$ is a complex in $\sC^b(\Rep_\bbR(\cG^g))$, then $\cA^{\bullet, \ldots, \bullet}(U^\bullet)$ becomes a $(g+1)$-tuple complex, 
with the $(g+1)$-st differential being the differential induced from that of $U^\bullet$.  
Let $h$ be an integer $>0$.
For any $h$-tuple complex $U^{\bullet,\ldots,\bullet}$, we define the total complex $\Tot^\bullet(U^{\bullet,\ldots,\bullet})$ to be the complex
whose $m$-th term is given by
\begin{equation*}
	\Tot^m(U^{\bullet,\ldots,\bullet}) :=  \bigoplus_{\substack{(m_1,\ldots,m_h)\in\bbZ^h\\m_1+\cdots+m_h =m}} U^{m_1,\ldots,m_h}
\end{equation*}
and whose $m$-th differential $d^m: \Tot^m(U^{\bullet,\ldots,\bullet}) \rightarrow \Tot^{m+1}(U^{\bullet,\ldots,\bullet})$ is given by
\begin{equation*}
	d^m:= \sum_{\substack{(m_1,\ldots,m_h)\in\bbZ^h\\m_1+\cdots+m_h =m}}  \partial^{m_1}_1+(-1)^{m_1} \partial^{m_2}_2 + \cdots + (-1)^{m_1+\cdots+m_{h-1}}\partial^{m_h}_h,
\end{equation*}
where $\partial^{m_\mu}_{\mu}$ is the partial differential on $U^{m_1,\ldots,m_h}$.

\begin{definition}\label{def: explicit complex}
	We define the functor
	$
		\Lambda^{\bullet}:\sC^b\left(\Rep_\bbR(\cG^g)\right)\rightarrow \sC^b\left(\Vec_{\bbR}\right)
	$
	by
	\begin{equation*}
		\Lambda^{\bullet}\left(U^{\bullet}\right):=\Tot^{\bullet}\left(\cA^{\bullet,\ldots,\bullet}\left(U^{\bullet}\right)\right).
	\end{equation*}
\end{definition}

\begin{lemma}\label{lem: democracy}
 	If $U^\bullet\rightarrow V^\bullet$ is a quasi-isomorphism in $\sC^b(\Rep_\bbR(\cG^g))$, then
	$\Lambda^\bullet(U^\bullet) \rightarrow\Lambda^\bullet(V^\bullet)$ is a quasi-isomorphism of complexes of $\bbR$-vector spaces.
\end{lemma}

\begin{proof}
	This follows from Lemma \ref{lem: exact}, which states that $\cA^{m_1,\ldots,m_g}$ are exact functors.
\end{proof}

We will use the functor $\Lambda^\bullet$ to calculate the functor $\Gamma_0$.
We will define intermediate functors $\cB$ and $\cC$ which will be used to relate the functors $\Lambda^\bullet$ and $\Gamma_0$.
Let $(m_1, \ldots, m_g) \in \{0,1\}^g$.  For $\mu=0, \ldots, g$, 
we inductively define the functors
\begin{align*}
	\cB^{m_1, \ldots, m_\mu}_\mu &: \Rep_\bbR(\cG^g) \rightarrow \Vec_\bbR,  &
	\cC^{m_1, \ldots, m_\mu}_\mu &: \Rep_\bbR(\cG^g) \rightarrow \Vec_\bbR
\end{align*}
as follows.  For $\mu =g$, we let
$
	\cB^{m_1, \ldots, m_g}_g(U) := \cA^{m_1, \ldots, m_g}(U)
$
and
$
	\cC^{m_1, \ldots, m_g}_g(U) :=0.
$
For an integer $\mu\geq0$, if $\cB^{m_1, \ldots, m_{\mu+1}}_{\mu+1}$ is defined, we define the functors for $\mu$ by
\begin{equation*}
	\cB^{m_1, \ldots, m_\mu}_\mu(U) := \Ker \left( \cB^{m_1, \ldots, m_\mu,0}_{\mu+1}(U) \xrightarrow{t_{\mu+1}-1} \cB^{m_1, \ldots, m_\mu,1}_{\mu+1}(U)  \right)
\end{equation*}
and
\begin{equation*}
	\cC^{m_1, \ldots, m_\mu}_\mu(U) := \Coker \left( \cB^{m_1, \ldots, m_\mu,0}_{\mu+1}(U) \xrightarrow{t_{\mu+1}-1} \cB^{m_1, \ldots, m_\mu,1}_{\mu+1}(U)  \right).
\end{equation*}
Note that we have
\begin{equation}\label{eq: Gamma0}
	\Gamma_0(U)= \cB_0(U):=  \Ker \left( \cB^0_1(U) \xrightarrow{t_{1}-1} \cB^{1}_{1}(U)  \right).
\end{equation}

\begin{example}
	The $\bbR$-vector spaces $\cB^{m_1, \ldots,m_\mu}_\mu(U)$ and  $\cC^{m_1, \ldots,m_\mu}_\mu(U)$ for $g=2$ fit into the following diagram,
	whose horizontal and vertical sequences are exact.
	\begin{equation*}
		\xymatrix{%
				&  & 0\ar[d] &0\ar[d] &  &\\
			0 \ar[r] &  \Gamma_0(U)  \ar[r]  &  \cB^0_1(U)  \ar[d]  \ar[r]^{t_1-1}   & \cB^1_1(U)   \ar[d] \ar[r]  &  \cC_0(U)    \ar[r] & 0\\
				&  & \cB^{0,0}_2(U)  \ar[d]_{t_2-1} \ar[r]^{t_1-1}& \cB^{1,0}_2(U)  \ar[d]^{t_2-1} &  &\\
				&  & \cB^{0,1}_2(U)  \ar[d]\ar[r]^{t_1-1}&\cB^{1,1}_2(U)  \ar[d] &  &\\
				&  & \cC^0_1 (U) \ar[d] \ar[r]^{t_1-1}& \cC^1_1(U) \ar[d] &  &\\
				&  & 0 &0. &  &
		}%
	\end{equation*}
	Note that $\cB^{m_1, m_2}_2(U) = \cA^{m_1, m_2}(U)$ in this case.
\end{example}

Again, if $U^\bullet$ is a complex in $\sC^b(\Rep_\bbR(\cG^g))$, then $\cB^{\bullet, \ldots, \bullet}_\mu(U^\bullet)$ and  $\cC^{\bullet, \ldots, \bullet}_\mu(U^\bullet)$
becomes $(\mu+1)$-tuple complexes  with the $(\mu+1)$-st differential being the differential induced from that of $U^\bullet$.  
We have an exact sequence of complexes
\begin{equation}\label{eq: triangle}
	\xymatrix{%
		0 \ar[r]&	\Tot^{\bullet}(\cB^{\bullet,\ldots,\bullet}_\mu(U^{\bullet})) \ar[r] & \Tot^{\bullet}(\cB^{\bullet,\ldots,\bullet}_{\mu+1}(U^{\bullet})) \ar[r] & \Tot^\bullet(\cC^{\bullet, \ldots, \bullet}_\mu(U^\bullet))[-1]\ar[r]&0.
	}%
\end{equation}
Note that we have
\begin{align}\label{eq: B}
	\Gamma_0(U^\bullet) &= \Tot^{\bullet}(\cB_0(U^{\bullet})), &
	\Lambda^\bullet(U^\bullet) &= \Tot^{\bullet}(\cB^{\bullet,\ldots,\bullet}_g(U^{\bullet})) 
\end{align}
by \eqref{eq: Gamma0}, the definition of the functor $\cB^{\bullet,\ldots,\bullet}_g$, and Definition \ref{def: explicit complex}.

%
\subsection{The vanishing of classes}
%

The main goal of this subsection is to prove Proposition \ref{prop: crucial}.
\begin{proposition}\label{prop: crucial}
	Let $\mu =0, \ldots, g$ and $m_1, \ldots, m_\mu \in \{0,1\}$.  For any $U^\bullet \in \sC^b(\Rep_\bbR(\cG^g))$, we have
	\begin{equation*}
		\lim_{s: U^\bullet \rightarrow V^\bullet} H^m(\cC^{m_1, \ldots, m_\mu}_\mu(V^\bullet) ) = 0
	\end{equation*}
	for any $m \in \bbZ$,
	where the direct limit is over quasi-isomorphisms $s: U^\bullet \rightarrow V^\bullet$.
\end{proposition}

We will give the proof of Proposition \ref{prop: crucial} at the end of this subsection.
The main idea of the proof is to reduce the statement to Lemma \ref{lemma: S}, which is the case when $U$ is a single object in $\Rep_\bbR(\cG^g)$ given
as a quotient of an exterior product $T \boxtimes T'$ of objects $T$ and $T'$ in $\Rep_\bbR(\cG^\mu)$.
In order to prove Lemma \ref{lemma: S},
we will first prove that the functor $\cA^{m_1, \ldots, m_\mu}$ preserves
exterior products.

\begin{lemma}\label{lemma: tensor}
	Let $T$ and $T'$ be objects in $\Rep_\bbR(\cG^\mu)$ and $\Rep_\bbR(\cG)$ respectively.
	The natural injection
	\begin{equation*}
		\cA^{m_1,\ldots,m_\mu}(T)\otimes_\bbR \cA^{m_{\mu+1}}(T') \rightarrow \cA^{m_1,\ldots,m_{\mu+1}}(T\boxtimes T')
	\end{equation*}
	is an isomorphism.
\end{lemma}

\begin{proof}
		Let $w := \sum_{k=1}^N u_k \otimes v_k \in \cA^{m_1,\ldots,m_{\mu+1}}(T \boxtimes T')$
		for some $u_k \in (\cA^{m_1}_1 \cap\cdots \cap \cA^{m_\mu}_\mu)(T)$ and $v_k \in \cA_1^{m_{\mu+1}}(T')$. 
		Then
		\begin{align*}
			u'_k := (u_k + (-t_1)^{m_1} \cdots (-t_\mu)^{m_\mu} \ol u_k)/2
			&&\text{and}&&
			u''_k := i (u_k-(-t_1)^{m_1}\cdots (-t_\mu)^{m_\mu} \ol u_k)/2
		\end{align*}
		are elements in $\cA^{m_1, \ldots, m_\mu}(T)$, and
		\begin{align*}
			v'_k := (v_k + (-t)^{m_{\mu+1}} \ol v_k)/2
			&&\text{and}&&
			v''_k:= i  (v_k - (-t)^{m_{\mu+1}} \ol v_k)/2
		\end{align*}
		are elements in $\cA^{m_{\mu+1}}(T')$. 
		Then we see that
		\begin{equation*}
			w =(w + (-t_1)^{m_1} \cdots (-t_{\mu+1})^{m_{\mu+1}} \ol w)/2 = \sum_{k=1}^N (u'_k \otimes v'_k - u''_k \otimes v''_k) 
		\end{equation*}
		is an element in $\cA^{m_1,\ldots,m_\mu}(T) \otimes_\bbR \cA^{m_{\mu+1}}(T')$ as desired.
\end{proof}

We will now prove Lemma \ref{lemma: S}.

\begin{lemma}\label{lemma: S}
	Let $R$ be an object in $\Rep_\bbR(\cG^{\mu+1})$.
	For any $\xi \in \cC^{m_1,\ldots, m_\mu}_\mu(R)$, there exists  an injection 
	$R \hookrightarrow S$ in $\Rep_\bbR(\cG^{\mu+1})$ such that the image of $\xi$ 
	in $\cC^{m_1,\ldots, m_\mu}_{\mu}(S)$ is \textit{zero}.
\end{lemma}

\begin{proof}
	By Theorem \ref{theorem: external}, we can reduce to the case when 
	$R=(T\boxtimes T')/N$, where $T$, $T'$ are objects respectively in $\Rep_\bbR(\cG^\mu)$, $\Rep_\bbR(\cG)$ and $N$ is a subobject of $T\boxtimes T'$. 
	We let $\wt\xi$ be an element of $\cB^{m_1,\ldots,m_\mu,1}_{\mu+1}(R) = \cA^{m_1, \ldots, m_\mu, 1}(R)$ representing $\xi$.
	By definition of the functor $\cC^{m_1,\ldots,m_\mu}_{\mu}$, it is sufficient to show that there exists an injection $R \hookrightarrow S$ in $\Rep_\bbR(\cG^{\mu+1})$ such
	that $\wt\xi$ is in the image of $t_{\mu+1}-1$ on $S$.
	Since the functor $\cA^{m_1, \ldots, m_\mu, 1}$ is exact, we have a surjection 
	\begin{equation}\label{eq: surj}
		 \cA^{m_1, \ldots, m_\mu}(T) \otimes_\bbR \cA^1(T') \cong \cA^{m_1, \ldots, m_\mu, 1}(T \boxtimes T') \rightarrow \cA^{m_1, \ldots, m_\mu, 1}(R),
	\end{equation}
	where the first isomorphism is given by Lemma \ref{lemma: tensor}.
	Hence there exists an element
	\begin{equation*}
		 \sum_{k=1}^N u_k \otimes u'_k  \in  \cA^{m_1, \ldots, m_\mu}(T) \otimes_\bbR \cA^1(T') \subset T_\bbC \otimes_\bbC T'_\bbC
	\end{equation*}	
	mapping by \eqref{eq: surj} to $\wt\xi$. We let $S = (S_\bbR, \{ S^{\bsp,\bsq} \}, \{ t_\mu \})$ be an object in 
	$\Rep_\bbR(\cG^{\mu+1})$ given as an extension
	\begin{equation*}\begin{CD}
		0  @>>> R @>>> S @>>> \bigoplus_{k=1}^N  T \boxtimes \bbR(0) @>>> 0,
	\end{CD}\end{equation*}
	whose underlying $\bbR$-vector space is the direct sum
	\begin{equation*}
		S_\bbR:= R_\bbR \oplus \bigoplus_{k=1}^N  (T \boxtimes \bbR(0))_\bbR,
	\end{equation*}
	the $2(\mu+1)$-grading and the $\bbC$-linear automorphisms $t_1, \ldots, t_{\mu}$ on $S_\bbC$ are also given by the direct sum, 
	and the $\bbC$-linear automorphism $t_{\mu+1}$ is given by $t_{\mu+1} := \id \otimes t$ when restricted to $R_\bbC$ and 
	\begin{equation*}
		t_{\mu+1}(w_1, \ldots, w_N): = (w_1, \ldots, w_N)+ \sum_{k=1}^N [w_k \otimes u_k']
	\end{equation*}
	for any $(w_1, \ldots, w_N)$ in 
	$
		\bigoplus_{k=1}^N  (T \boxtimes \bbR(0))_\bbC = \bigoplus_{k=1}^N  T_\bbC,
	$
	where $\sum_{k=1}^N [w_k \otimes u_k']$ is the image of  $\sum_{k=1}^N w_k \otimes u_k'$ by
	the surjection $(T \boxtimes T')_\bbC \rightarrow R_\bbC$.
	We show that $\ol{t_{\mu+1}} = t_{\mu+1}^{-1}$ from the fact that $t(\ol u'_k) = - u'_k$ since $u_k' \in \cA^1(T')$.  Then 
	$S$ defined as above is an object in $\Rep_\bbR(\cG^{\mu+1})$.  
	If we let
	\begin{equation*}	
		\eta: = (u_1, \ldots, u_N) \in \bigoplus_{k=1}^N  (T \boxtimes \bbR(0))_\bbC \subset S_\bbC,
	\end{equation*}
	then $\eta \in \cA^{m_1, \ldots, m_\mu,0}(S)$ by construction, and we have
	\begin{equation*}
		(t_{\mu+1} -1) \eta = \sum_{k=1}^N [u_k \otimes u_k'] = \wt\xi.
	\end{equation*}
	This shows that the class of $\xi$ in $\cC^{m_1,\ldots, m_\mu}_\mu(S)$ is zero as desired.
\end{proof}

Suppose $U$ is an object in $\Rep_\bbR(\cG^{\mu+1})$.  Then by Remark \ref{rem: inclusion}, we may view $U$ as an object in $\Rep_\bbR(\cG^g)$.
Since $t_{\mu+2}, \ldots, t_g$ for $U$ is the identity map, we have $\cB^{m_1, \ldots,m_{\mu+1}}_{\mu+1}(U) = \cA^{m_1, \ldots,  m_{\mu+1}}(U)$,
hence
\begin{equation*}
	\cC^{m_1,\ldots, m_\mu}_\mu(U) = \Coker\bigl(\cA^{m_1, \ldots, m_{\mu},0}(U) \xrightarrow{t_{\mu+1}-1} \cA^{m_1, \ldots, m_{\mu},1}(U)\bigr)
\end{equation*}
in this case.
Now we are ready to prove Proposition \ref{prop: crucial}.

\begin{proof}[Proof of Proposition \ref{prop: crucial}]
	It is sufficient to show that for any $U^\bullet \in \sC^b(\Rep_\bbR(\cG^g))$ and $m$-cocycle $\xi\in \cC^{m_1, \ldots, m_\mu}_\mu(U^m)$,
	there exists a quasi-isomorphism $s: U^\bullet \rightarrow V^\bullet$ such that $s(\xi)$ is zero in $\cC^{m_1, \ldots, m_\mu}_\mu(V^m)$.
	Let $R: = \Gamma_{\mu+1}(U^m)$, which is a  mixed $(\mu+1)$-plectic $\bbR$-Hodge structure of Definition \ref{def: gamma nu}.
	Then by definition, we have
	\begin{equation*}
		\cB^{m_1, \ldots, m_{\mu+1}}_{\mu+1}(U^m) = \cA^{m_1, \ldots, m_{\mu+1}}(R),
	\end{equation*}
	which shows that
	\begin{equation*}
		\cC^{m_1, \ldots, m_\mu}_\mu(U^m) = \cC^{m_1, \ldots, m_\mu}_\mu(R).
	\end{equation*}
	By Lemma \ref{lemma: S}, there exists an injection $\iota:R \hookrightarrow S$ in $\Rep_\bbR(\cG^{\mu+1})$ such that the image of $\xi$ in $\cC^{m_1, \ldots, m_\mu}_\mu(S)$ is zero, which we also view as an injection in $\Rep_\bbR(\cG^g)$.
	Then we have a commutative diagram
	\begin{equation*}
		\xymatrix{%
			R \,\,\ar@{^{(}->}[r]^{\iota}\ar@{^{(}->}[d]_r & S \ar@{^{(}->}[d] \\
			U^m  \,\,\ar@{^{(}->}[r] &  (U^m \oplus S)/R
		}%
	\end{equation*}
	in $\Rep_\bbR(\cG^g)$,
	where $r$ is the natural inclusion and the quotient $(U^m\oplus S)/R$ is taken by the injection $(r,-\iota):R\hookrightarrow U^m\oplus S$.
	Note that the image of $\xi$ in $\cC^{m_1, \ldots, m_\mu}_\mu( (U^m \oplus S)/R)$ is zero.	
	We let $V^\bullet$ be the complex obtained from $U^\bullet$ by replacing $U^m$ by $ (U^{m} \oplus S)/R$
	and $U^{m+1}$ by $(U^{m+1} \oplus S)/R$, with the differential induced by $d^m_U \oplus \id : U^m\oplus S \rightarrow U^{m+1}\oplus S$.
	Now we have an exact sequence of complexes
	\begin{equation*}
	\xymatrix{%
		&\ar[d] & \ar[d]    &  \ar[d]  &\\
		0\ar[r] & 0  \ar[d] \ar[r]&  U^{m-1} \ar[d]^{(d,0)} \ar[r]^\id& U^{m-1} \ar[d]\ar[r] & 0\\
		0\ar[r]  &R \ar[d]_{\id} \ar[r]^<<<<<<{(r,-\iota)}& U^{m} \oplus S \ar[r] \ar[d]^{d \oplus \id}&    (U^{m} \oplus S)/R \ar[d] \ar[r]& 0\\
		0\ar[r]  &R \ar[d] \ar[r]^<<<<<{(r,-\iota)} & U^{m+1} \oplus S\ar[d]^{d\oplus 0} \ar[r]& (U^{m+1} \oplus S)/R \ar[d]\ar[r]&0 \\
		  0\ar[r]&0  \ar[d]\ar[r]& U^{m+2}\ar[d]\ar[r]^\id&U^{m+2} \ar[d]\ar[r]&0\\
		   & & & &,
	}%
	\end{equation*}
	in which the left vertical complex is acyclic and the middle vertical complex is quasi-isomorphic to $U^\bullet$.  Hence
	the right vertical complex $V^\bullet$ is quasi-isomorphic to $U^\bullet$ with respect to the natural inclusion
	$U^\bullet \hookrightarrow V^\bullet$.  Then the complex $V^\bullet$ satisfies the desired assertion.
\end{proof}

%
\subsection{The calculation of the extension groups}\label{section: 4-3}
%

The purpose of this subsection is to prove Theorem \ref{thm: old theorem}, which calculates the extension groups in 
$\Rep_\bbR(\cG^g)$ in terms of the functor $\Lambda^\bullet$.
%
%
%
%
%
%
%
\begin{theorem}\label{thm: old theorem}
	For any object $U^{\bullet}$ in $\sC^b\left(\Rep_\bbR(\cG^g)\right)$ and $m\in\bbZ$, there exists a canonical isomorphism
	\begin{equation*}
		\Ext^m_{\Rep_\bbR(\cG^g)}\left(\bbR(0),U^{\bullet}\right)\rightarrow H^m\left(\Lambda^{\bullet}\left(U^{\bullet}\right)\right).
	\end{equation*}
\end{theorem}
\begin{proof}
	By \eqref{eq: triangle}, we have a distinguished triangle 
	\begin{equation*}\xymatrix{%
		\Tot^{\bullet}(\cB^{\bullet,\ldots,\bullet}_\mu(U^{\bullet})) \ar[r]&
		\Tot^{\bullet}(\cB^{\bullet,\ldots,\bullet}_{\mu+1}(U^{\bullet})) \ar[r]&
		\Tot^{\bullet}\left(\cC_{\mu}^{\bullet,\ldots,\bullet}\left(U^{\bullet}\right)\right)[-1]
	}\end{equation*}
	in $\sK^b\left(\Rep_\bbR(\cG^g)\right)$ for $\mu = 0,\ldots,g-1$.
	By Proposition \ref{prop: crucial} we have
	\begin{equation*}
		\varinjlim_{U^{\bullet}\rightarrow V^{\bullet}}H^n\left(\Tot^{\bullet}\left(\cC_{\mu}^{\bullet\ldots\bullet}\left(V^{\bullet}\right)\right)\right)=0,
	\end{equation*}
	where the direct limit is over quasi-isomorphisms $s: U^\bullet \rightarrow V^\bullet$.
	Hence we have an isomorphism
	\begin{equation*}\xymatrix{
		\displaystyle\varinjlim_{U^{\bullet}\rightarrow V^{\bullet}}
		H^m(\Tot^{\bullet}(\cB^{\bullet,\ldots,\bullet}_\mu(V^{\bullet})) )
		\ar[r]^<<<<\cong &
		\displaystyle\varinjlim_{U^{\bullet}\rightarrow V^{\bullet}}H^m(\Tot^{\bullet}(\cB^{\bullet,\ldots,\bullet}_{\mu+1}(V^{\bullet})) ),
	}\end{equation*}
	since direct limit preserves exactness.
	By \eqref{eq: B}, we have by induction an isomorphism
	\begin{equation}\label{isom1}
	\xymatrix{
		\displaystyle\varinjlim_{U^{\bullet}\rightarrow V^{\bullet}}
		H^m(\Gamma_0(V^{\bullet}))
		\ar[r]^<<<<\cong &
		\displaystyle\varinjlim_{U^{\bullet}\rightarrow V^{\bullet}}H^m(\Lambda^\bullet(V^{\bullet})).
	}
	\end{equation}
	By Lemma \ref{lem: democracy}, the map
	\begin{equation}\label{isom2}
	\xymatrix{%
		H^m\left(\Lambda^{\bullet}\left(U^{\bullet}\right)\right)\ar[r]&
		\displaystyle\varinjlim_{U^{\bullet}\rightarrow V^{\bullet}}
		H^m\left(\Lambda^{\bullet}\left(V^{\bullet}\right)\right)
	}%
	\end{equation}
	is an isomorphism.  
	On the other hand, we have
	\begin{multline}\label{isom3}
		\Ext^m_{\Rep_\bbR(\cG^g)}\left(\bbR(0),U^{\bullet}\right) := \Hom_{\sD^b\left(\Rep_\bbR(\cG^g)\right)}\left(\bbR(0),U^{\bullet}[m]\right) \\
		=\varinjlim_{U^{\bullet}\rightarrow V^{\bullet}}\Hom_{\sK^b\left(\Rep_\bbR(\cG^g)\right)}\left(\bbR(0),V^{\bullet}[m]\right)
		\cong\varinjlim_{U^{\bullet}\rightarrow V^{\bullet}}H^m\left(\Gamma_0(V^{\bullet})\right),
	\end{multline}
	where the last isomorphism is Lemma \ref{trivial}.
	Hence the composition of isomorphisms \eqref{isom1}, \eqref{isom2}, and \eqref{isom3} gives our assertion.
\end{proof}

\begin{example}
	Let $\bsn\in\bbZ^g$.
	When $\bbR(\bsn)$ is the plectic Tate object of Example \ref{example: realtate}, then we have by \eqref{eq: A}
	\begin{align*}
		\cA^{0}_\mu(\bbR(\bsn)) 
		&=\begin{cases}
			0   &   n_\mu \neq 0, \\
			\bbR &  n_\mu=0,
		\end{cases}
		&
		\cA^{1}_\mu(\bbR(\bsn)) 
		&=\begin{cases}
			0   &   n_\mu \leq 0, \\
			\bbR &  n_\mu>0.
		\end{cases}
	\end{align*}
	In particular, if $\bsn=(n,\ldots,n)$ for some $n\in\bbZ$, then we have
	\begin{equation*}
		\cA^{\bsm}(\bbR(\bsn)) =
		\begin{cases}
			\bbR  & n=0, \quad \bsm=(0,\ldots,0),\\
			(2\pi i)^{g(n-1)}\bbR  & n>0, \quad \bsm=(1,\ldots,1),\\
			0 &\text{otherwise}.
		\end{cases}
	\end{equation*}
	Then all of the differentials of the complex $\cA^{\bullet,\ldots,\bullet}(\bbR(\bsn))$ are zero maps, hence Theorem \ref{thm: old theorem} shows that we have
	\begin{align*}
		\Ext^0_{\Rep_\bbR(\cG^g)}(\bbR(0), \bbR(\bsn)) 
		& =
		  \begin{cases}
		  	\bbR  &  n=0, \\
			0  & \text{otherwise},
		  \end{cases}  \\
		  \Ext^g_{\Rep_\bbR(\cG^g)}(\bbR(0), \bbR(\bsn)) 
		& =
		  \begin{cases}
		  	(2\pi i)^{g(n-1)} \bbR  & n>0,  \\
			0  & \text{otherwise},
		  \end{cases}  
	\end{align*}
	and $\Ext^m_{\Rep_\bbR(\cG^g)}(\bbR(0), \bbR(\bsn)) =0$ for $m\neq0,g$.
\end{example}

\begin{corollary}
	For an object $U^{\bullet}$ in $\sC^b\left(\Rep_\bbR(\cG^g)\right)$, there exists a spectral sequence
	\begin{equation}\label{spectralseq}
		E_2^{m,n}=\Ext_{\Rep_\bbR(\cG^g)}^m\left(\bbR(0),H^n\left(U^{\bullet}\right)\right)
		\Rightarrow \Ext^{m+n}_{\Rep_\bbR(\cG^g)}\left(\bbR(0),U^{\bullet}\right),
	\end{equation}
	which degenerates at $E_{g+1}$.
\end{corollary}

\begin{proof}
	Let $\Ind(\Rep_\bbR(\cG^g))$ be the ind-category of $\Rep_\bbR(\cG^g)$ (See \cite{KS} Definition 6.1.1).
	By \cite{Stauffer} Theorem 2.2, $\Ind(\Rep_\bbR(\cG^g))$ is an abelian category with enough injectives and 
	the canonical fully faithful functor $\Rep_\bbR(\cG^g)\rightarrow\Ind(\Rep_\bbR(\cG^g))$ is exact, 
	since $\Rep_\bbR(\cG^g)$ is essentially small.
	Then for an object $U^{\bullet}$ in $\sC^b\left(\Rep_\bbR(\cG^g)\right)$, we have a spectral sequence
	\begin{equation*}
		E_1^{m,n}=\Ext_{\Ind\left(\Rep_\bbR(\cG^g)\right)}^m\left(\bbR(0),H^n\left(U^{\bullet}\right)[-n]\right)
		\Rightarrow \Ext_{\Ind\left(\Rep_\bbR(\cG^g)\right)}^{m+n}\left(\bbR(0),U^{\bullet}\right)
	\end{equation*}
	associated to the canonical filtration on $U^{\bullet}$ (See \cite{D1} 1.4.5 and 1.4.6).
	By renumbering this gives
	\begin{equation*}
		E_2^{m,n}=\Ext^m_{\Ind\left(\Rep_\bbR(\cG^g)\right)}\left(\bbR(0),H^n\left(U^{\bullet}\right)\right)
		\Rightarrow\Ext^{m+n}_{\Ind\left(\Rep_\bbR(\cG^g)\right)}\left(\bbR(0),U^{\bullet}\right).
	\end{equation*}
	Since $\Rep_\bbR(\cG^g)$ is noetherian, $\sD^b\left(\Rep_\bbR(\cG^g)\right)\rightarrow \sD^b\left(\Ind(\Rep_\bbR(\cG^g))\right)$ 
	is fully faithful by \cite{H2} Proposition 2.2. Hence, when $U^{\bullet}$ is lying in $\sC^b\left(\Rep_\bbR(\cG^g)\right)$ 
	we obtain the spectral sequence \eqref{spectralseq}.
	By Theorem \ref{thm: old theorem} we have $\Ext_{\Rep_\bbR(\cG^g)}^m\left(\bbR(0),H^n\left(U^{\bullet}\right)\right)=0$ 
	for $m>g$, hence \eqref{spectralseq} degenerates at $E_{g+1}$.
\end{proof}
 
\begin{corollary}
	Let $U_1,\ldots, U_g$ be objects in $\Rep_\bbR(\cG)$.
	Then there exists a canonical isomorphism
	\begin{equation*}
		\bigoplus_{\substack{(m_1,\ldots,m_g)\in\bbZ^g\\m_1+\cdots+m_g=m}}\bigotimes_{1\leq\mu\leq g}\Ext_{\Rep_\bbR(\cG)}^{m_\mu}\left(\bbR(0),U_\mu\right)
		\rightarrow\Ext^m_{\Rep_\bbR(\cG^g)}\left(\bbR(0),U_1\boxtimes\cdots\boxtimes U_g\right),
	\end{equation*}
	for each $m\in\bbZ$. In particular, we have a canonical isomorphism
	\begin{equation*}
		\bigotimes_{1\leq\mu\leq g}\Ext_{\Rep_\bbR(\cG)}^1\left(\bbR(0),U_\mu\right)\rightarrow\Ext^g_{\Rep_\bbR(\cG^g)}\left(\bbR(0),U_1\boxtimes\cdots\boxtimes U_g\right).
	\end{equation*}
\end{corollary}

\begin{proof}
	By Lemma \ref{lemma: tensor} we have 
	\begin{equation*}
		\cA^{m_1,\ldots,m_g}\left(U_1\boxtimes\cdots\boxtimes U_g\right)=\cA^{m_1}\left(U_1\right)\otimes_{\bbR}\cdots\otimes_{\bbR}\cA^{m_g}\left(U_g\right).
	\end{equation*}
	Since every $\bbR$-module is flat, we have an isomorphism.  This proves our assertion.
\end{proof}

\subsection*{Acknowledgement}
	This article is a result of a series of workshops held at Keio University attended by the authors to understand the article \cite{NS}.
	The authors would like to thank the KiPAS program at the Faculty of Science and Technology at Keio University for continuous support for this research.
	The authors would also like to thank the coordinator Masato Kurihara of the JSPS Core-to-Core program 
	``Foundation of a Global Research Cooperative Center in Mathematics focused on Number Theory and Geometry'' for funding our research.
	The authors thank the referee for comments concerning the article.

\end{document}